\documentclass[10pt]{amsart}
\usepackage[a4paper,marginratio=1:1]{geometry}
\usepackage[non-sorted-cites,initials]{amsrefs}
\usepackage{tensor,braket,bm,fouridx,tikz-cd}

\newtheorem{thm}{Theorem}[section]
\newtheorem{prop}[thm]{Proposition}
\newtheorem{lem}[thm]{Lemma}

\theoremstyle{definition}
\newtheorem{dfn}[thm]{Definition}
\theoremstyle{remark}
\newtheorem{rem}[thm]{Remark}
\newtheorem{ex}[thm]{Example}

\DeclareMathOperator{\Vol}{Vol}
\DeclareMathOperator{\tr}{tr}
\DeclareMathOperator{\tf}{tf}
\DeclareMathOperator{\Ric}{Ric}
\DeclareMathOperator{\Scal}{Scal}
\DeclareMathOperator{\re}{Re}
\DeclareMathOperator{\Tor}{Tor}
\DeclareMathOperator{\End}{End}
\DeclareMathOperator{\im}{Im}
\DeclareMathOperator{\id}{id}
\DeclareMathOperator{\spanof}{span}

\DeclareMathOperator{\Sym}{Sym}
\DeclareMathOperator{\contr}{contr}

\newcommand{\entrydot}{\mathord{\cdot}}
\newcommand{\PSU}{\mathit{PSU}}

\newcommand{\cl}[1]{\overline{#1}}
\newcommand{\abs}[1]{\lvert#1\rvert}
\newcommand{\conj}[1]{\overline{#1}}

\newcommand{\conjbeta}{{\conj{\beta}}}
\newcommand{\bdry}{\partial}
\newcommand{\thetatangent}{{\fourIdx{\Theta}{}{}{}{TX}}}

\newcommand{\thetacotangent}{{\fourIdx{\Theta}{}{}{}{T^*X}}}

\newcommand{\thetadiff}{{\fourIdx{\Theta}{}{}{}{\mathrm{Diff}}}}

\def\crn#1#2{{\vcenter{\vbox{
        \hbox{\kern#2pt \vrule width.#2pt height#1pt
           }
          \hrule height.#2pt}}}}
\def\intprod{\mathchoice\crn54\crn54\crn{3.75}3\crn{2.5}2}
\def\contraction{\mathbin{\intprod}}

\numberwithin{equation}{section}

\title[Partially integrable CR manifolds]{GJMS operators, $Q$-curvature, and obstruction tensor of partially integrable CR manifolds}
\author{Yoshihiko Matsumoto}
\thanks{Partially supported by Grant-in-Aid for JSPS Fellows (10J06494).}
\address{Department of Mathematics, Tokyo Institute of Technology,
	2-12-1 Ookayama, Meguro, Tokyo 152-8551, Japan}
\email{matsumoto@math.titech.ac.jp}

\begin{document}

\begin{abstract}
	We extend the notions of CR GJMS operators and $Q$-curvature to the case of partially integrable CR structures.
	The total integral of the CR $Q$-curvature turns out to be a global invariant of compact
	nondegenerate partially integrable CR manifolds equipped with an orientation of the bundle of contact forms,
	which is nontrivial in dimension at least five.
	It is shown that its variation is given by the curvature-type quantity called the CR obstruction tensor,
	which is introduced in the author's previous work.
	Moreover, we consider the linearized CR obstruction operator.
	Based on a scattering-theoretic characterization,
	we discuss its relation to the CR deformation complex of integrable CR manifolds.
	The same characterization is also used to determine the Heisenberg principal symbol of the linearized
	CR obstruction operator.
\end{abstract}

\maketitle

\section*{Introduction}

Recent development of general theory of parabolic geometries predicts the possibility of understanding
geometric objects defined for any particular type of geometries in a broader context.
With this general idea in mind, in the current article we discuss the CR versions of
the GJMS operators, Branson's $Q$-curvature, and the Fefferman--Graham obstruction tensor,
which are originally studied in conformal
geometry~\cite{Graham_Jenne_Mason_Sparling_92,Branson_95,Fefferman_Graham_85}.

The CR versions of the first two have been known for quite a long time.
The CR $Q$-curvature was defined by Fefferman and Hirachi~\cite{Fefferman_Hirachi_03} as the push-forward of
the conformal $Q$-curvature associated to the Fefferman metric, which is a certain Lorentzian metric on
the total space of a trivial circle bundle over a given manifold~\cite{Fefferman_76,Lee_86}.
Gover and Graham~\cite{Gover_Graham_05} constructed two independent families of CR-invariant powers of
the sub-Laplacian, one by the tractor calculus approach and the other by the Fefferman metric approach,
the latter of which we call the CR GJMS operators.
Hislop, Perry, and Tang~\cite{Hislop_Perry_Tang_08} discussed them from the viewpoint of scattering theory
following the argument of Graham and Zworski~\cite{Graham_Zworski_03} in the conformal case.

However, the studies described above were all restricted to the case of \emph{integrable} CR structures.
They are actually only a part of CR geometries, if we mean by this term parabolic geometries
modeled on the standard CR sphere $S^{2n+1}\subset\mathbb{C}^{n+1}$.
More precisely, let $G=\PSU(n+1,1)$ be the group of CR automorphisms of $S^{2n+1}$ and
$P$ the stabilizer of any fixed point on $S^{2n+1}$.
Then \v{C}ap and Schichl~\cite{Cap_Schichl_00} showed (see also~\cite{Cap_Slovak_09}*{Subsection 4.2.4}) that
the category of normal regular parabolic geometries of type $(G,P)$ is equivalent to that of
strictly pseudoconvex partially integrable CR manifolds,
where an almost CR manifold $(M,T^{1,0}M)$ is said to be \emph{partially integrable} if the following is satisfied:
\begin{equation}
	\label{eq:partial_integrability}
	[\Gamma(T^{1,0}M),\Gamma(T^{1,0}M)]\subset\Gamma(T^{1,0}M\oplus\conj{T^{1,0}M}).
\end{equation}
Therefore, from the viewpoint of parabolic geometries, it is natural to consider this extended class of
CR structures.

In this article, we shall work on nondegenerate partially integrable CR manifolds
with the bundle of contact forms oriented, which we simply call \emph{partially integrable CR manifolds}.
When we take a contact form $\theta$, we always implicitly assume
that it is positive with respect to the given orientation.

The purpose of this paper is twofold: on one hand we strengthen the parallelism between the conformal and CR cases,
and on the other hand we discuss a rather subtle nature of the CR case.
Firstly, we construct CR GJMS operators $P_{2k}$ and $Q$-curvature $Q_\theta$ for
partially integrable CR structures by generalizing the scattering-theoretic approach of
Hislop, Perry, and Tang to asymptotically complex hyperbolic Einstein metrics\footnote{The other
	possible approach, which utilizes the Fefferman metric, is not discussed in this article.
	There is a work of Leitner~\cite{Leitner_10} on a ``gauged'' Fefferman construction on
	partially integrable CR manifolds, and this direction should also be pursued further.}.
Then the total integral of the $Q$-curvature becomes a global CR invariant, and we further discuss its
variational properties.
Secondly, we consider the linearized obstruction operator $\mathcal{O}^\bullet$,
and explore its properties in connection with the CR deformation complex of Rumin~\cite{Rumin_94} and
Akahori, Garfield, Lee~\cite{Akahori_Garfield_Lee_02}.
The important ingredient is a scattering-theoretic characterization of $\mathcal{O}^\bullet$
based on the idea of~\cite{Matsumoto_13}.

The concept of asymptotically complex hyperbolic metrics (hereafter ACH metrics) is due to
Epstein, Melrose, and Mendoza~\cite{Epstein_Melrose_Mendoza_91}.
They are generalizations of Bergman-type complete K\"ahler metrics on
bounded strictly pseudoconvex domains $\Omega\subset\mathbb{C}^{n+1}$,
i.e., K\"ahler metrics given by a potential function of the form $\log(1/r)$ up to constant multiple,
where $r$ is a boundary defining function of $\Omega$.
The definition goes as follows: we first introduce $\Theta$-structures on manifolds-with-boundary,
by which the notion of $\Theta$-metrics makes sense, and then define ACH metrics as $\Theta$-metrics
satisfying some extra conditions.
The idea of $\Theta$-structures comes from the fact that, in the case of strictly pseudoconvex domains,
$\frac{i}{2}(\partial r-\conj{\partial}r)|_{\bdry\Omega}$ gives a conformal class of sections of
$T^*\cl{\Omega}$ over $\bdry\Omega$ that is independent of $r$.
If $X$ is a $\Theta$-manifold, i.e., a $C^\infty$-smooth manifold-with-boundary equipped with
a $\Theta$-structure, then it has two important features:
\begin{itemize}
	\item the boundary $M=\partial X$ is equipped with a contact distribution $H$,
		with an orientation of the bundle $H^\perp\subset T^*M$ of contact forms;
	\item the usual tangent bundle $TX$ is blown up along the boundary to define the
		\emph{$\Theta$-tangent bundle} $\thetatangent$ in a way that, in the case of strictly pseudoconvex domains,
		Bergman-type metrics continuously extend up to the boundary as $\Theta$-metrics, i.e.,
		metrics of $\thetatangent$.
\end{itemize}
Then ACH metrics form a special class of $\Theta$-metrics.
They are defined in such a way that each ACH metric determines a partially integrable CR structure $T^{1,0}M$
on $M$, which we call the \emph{CR structure at infinity}.
It is \emph{compatible} to the $\Theta$-structure in the sense that $\re T^{1,0}M=H$.
In order to allow the Levi form to be of indefinite signature, we do not assume that ACH metrics are definite.

Conversely, given a $\Theta$-manifold $X$, one can think of a compatible partially integrable CR structure
on $M=\partial X$ as a Dirichlet data for ACH metrics on $X$.
Then one naturally tries to solve the Einstein equation under this boundary condition.
While a perturbation result is obtained by Biquard~\cite{Biquard_00},
the author has constructed in~\cite{Matsumoto_14}, for any given partially integrable CR structure on
$\partial X$, a best possible approximate solution which is $C^\infty$-smooth up to the boundary,
and determined its arbitrariness.
Let $\rho\in C^\infty(X)$ be any boundary defining function of $X$.
Then our approximate solution satisfies%
\begin{subequations}
	\label{eq:approx_Einstein}
\begin{equation}
	\label{eq:approx_Einstein_ricci}
	\Ric=-\frac{n+2}{2}g+O(\rho^{2n+2})
\end{equation}
and
\begin{equation}
	\label{eq:approx_Einstein_trace}
	\Scal=-(n+1)(n+2)+O(\rho^{2n+3}),
\end{equation}
\end{subequations}
where $\Ric$ and $\Scal$ are the Ricci tensor and the scalar curvature of $g$, respectively,
in the sense that $E:={\Ric}+\tfrac{n+2}{2}g\in \rho^{2n+2}C^\infty(X;\Sym^2\thetacotangent)$ and
${\Scal}+(n+1)(n+2)\in\rho^{2n+3}C^\infty(X)$.
Moreover, associated to each contact form $\theta$ is a preferred local frame
$\set{\bm{Z}_\infty,\bm{Z}_0,\bm{Z}_\alpha,\bm{Z}_{\conj{\alpha}}}$ of $\mathbb{C}\thetatangent$, and
the proof of the existence of such a $g$ simultaneously shows that, if $\rho$ is the model boundary defining
function associated to $\theta$ (see Subsection \ref{subsec:ACH_metrics}) and $E=\rho^{2n+2}\tilde{E}$, then
the boundary values of the components $\tensor{\tilde{E}}{_\alpha_\beta}=\tilde{E}(\bm{Z}_\alpha,\bm{Z}_\beta)$
invariantly define a density-weighted tensor
$\tensor{\mathcal{O}}{_\alpha_\beta}\in\tensor{\mathcal{E}}{_\alpha_\beta}(-n,-n)$.
This is what we call the \emph{CR obstruction tensor}, and by the construction,
the trivialization of $\tensor{\mathcal{O}}{_\alpha_\beta}$ with respect to $\theta$ has a universal expression
in terms of the Tanaka--Webster connection of $\theta$,
in the sense that it is written as a universal polynomial of covariant derivatives of the Nijenhuis tensor,
those of the Tanaka--Webster torsion and curvature, the Levi form, and its dual.
It is known~\cite{Matsumoto_14}*{Theorem 0.2} that $\tensor{\mathcal{O}}{_\alpha_\beta}$ vanishes for
integrable CR manifolds; hence it is of interest when the dimension is at least $5$, for any 3-dimensional
almost CR manifold is integrable.
The content of this paragraph is recalled in Section \ref{sec:ACHE} in greater detail\footnote{The
	approximate Einstein condition \eqref{eq:approx_Einstein} is slightly modified from
	what we imposed in~\cite{Matsumoto_14} to make the exposition simpler and easier to understand.}.

The CR GJMS operators $P_{2k}$ are characterized as Dirichlet-to-Neumann-type operators associated to
the generalized eigenfunction problem for the Laplacian of our approximate ACH-Einstein metric $g$.
Since the construction of such operators can be executed without assuming that $g$ is (approximately) Einstein,
we consider the following loose setting in the first place:
$g$ is an ACH metric on a $(2n+2)$-dimensional $\Theta$-manifold $X$ that is $C^\infty$-smooth up to the boundary,
and $T^{1,0}M$ is its CR structure at infinity.
We take any contact form $\theta$ and the model boundary defining function associated to $\theta$ is denoted
by $\rho$.
Under this setting, we consider the following equation for any positive integer $k$:
\begin{equation}
	\label{eq:eigenfunction_equation}
	\left(\Delta-\frac{(n+1)^2-k^2}{4}\right)u=0.
\end{equation}
Our convention of the Laplacian is such that it has negative principal symbol when $g$ is a positive definite
metric. Since the indicial roots of the operator $\Delta-((n+1)^2-k^2)/4$ are $n+1\pm k$,
any sufficiently regular solution of \eqref{eq:eigenfunction_equation} must behave asymptotically like
$\rho^{n+1-k}$ or $\rho^{n+1+k}$. If we try power series expansions starting with $\rho^{n+1-k}f$
to solve \eqref{eq:eigenfunction_equation}, where $f\in C^\infty(M)$, then a possible obstruction
appears at the power $\rho^{n+1+k}$, which defines a differential operator $P^g_{2k}$ on the boundary.
In other words, $P^g_{2k}f$ is the coefficient of the first logarithmic term of the solution,
in which way Theorem \ref{thm:dn_operator} is stated.
In particular, $P^g_{2n+2}$ is the obstruction to the $C^\infty$-smooth harmonic extension
of a given function on the boundary.
This is the ``compatibility operator'' which Graham~\cite{Graham_83} studied
when $(\mathring{X},g)$ is the complex hyperbolic space.
Note also that, while $P^g_{2k}$ depends on $g$ in a complicated way, its dependency on $\theta$ is rather trivial.
This is because $\theta$ is involved only in the choice of $\rho$ and has nothing to do with the equation itself.
As a consequence, $P^g_{2k}$ invariantly defines an operator
$\mathcal{E}(-(n+1-k)/2,-(n+1-k)/2)\longrightarrow\mathcal{E}(-(n+1+k)/2,-(n+1+k)/2)$ between densities.

The associated quantity $Q^g_\theta$ is defined in Theorem \ref{thm:log_expansion_definition_of_Q}
as the first logarithmic term coefficient of $U_0$ solving $\Delta(\log\rho+U_0)=(n+1)/2$ formally.
This characterization imitates the work of Fefferman and Graham~\cite{Fefferman_Graham_02} in the conformal case,
and behind this is Branson's original ``analytic continuation in dimension'' argument:
$Q^g_\theta$ is the ``derivative at $N=n$'' of the function $P^g_{2n+2}1$
in dimension $2N+2$, where $N$ is considered as a continuous variable.
Suppose $u_N$ is the formal solution of \eqref{eq:eigenfunction_equation} for $k=n+1$ in dimension $2N+2$
starting with $\rho^{N-n}$ (we take $f$ to be $1$).
Then, by ``differentiating $u_N$ at $N=n$,'' we obtain a function $U$ satisfying $\Delta U=(n+1)/2$
with leading term $\log\rho$ and the next logarithmic term $Q^g_\theta\rho^{2n+2}\log\rho$.

We apply Theorems \ref{thm:dn_operator} and \ref{thm:log_expansion_definition_of_Q}
to the special case where $g$ satisfies \eqref{eq:approx_Einstein} to get $P_{2k}$ and $Q_\theta$,
which are \emph{CR GJMS operators} and the \emph{CR $Q$-curvature},
as justified by the following result proved in Subsection \ref{subsec:GJMS}.
We say that a differential operator $D\colon C^\infty(M;E)\longrightarrow C^\infty(M;F)$ has
\emph{Heisenberg order $\le m$} if it is locally expressed as a matrix with entries of the form
\begin{equation*}
	\sum_{2\alpha_0+\alpha_1+\dots+\alpha_{2n}\le m}
	a_{(\alpha_0,\alpha_1,\dots,\alpha_{2n})}Y_0^{\alpha_0}Y_1^{\alpha_1}\dotsb Y_{2n}^{\alpha_{2n}},
\end{equation*}
where $\set{Y_0, Y_1, \dots, Y_{2n}}$ is a local frame of $TM$ such that
$\set{Y_1, \dots, Y_{2n}}$ spans the contact distribution $H=\re T^{1,0}M$.

\begin{thm}
	\label{thm:dn_operator_of_ACHE}
	Let $g$ be a $C^\infty$-smooth ACH metric on a $(2n+2)$-dimensional $\Theta$-manifold satisfying
	the approximate Einstein condition \eqref{eq:approx_Einstein}.
	Then $P_{2k}=P^g_{2k}$ for $k\le n+1$ and $Q_\theta=Q^g_\theta$ are independent of the ambiguity in $g$
	and determined only by the CR structure at infinity and a contact form $\theta$. Each $P_{2k}$ has the form
	\begin{equation}
		P_{2k}=\prod_{j=0}^{k-1}(\Delta_b+i(k-1-2j)T)
		+(\text{a differential operator with Heisenberg order $\le 2k-1$}),
	\end{equation}
	where $\Delta_b$ is the sub-Laplacian and $T$ is the Reeb vector field of $\theta$.
	If $\Hat{\theta}=e^\Upsilon\theta$, where $\Upsilon\in C^\infty(M)$, then $Q_\theta$ transforms as
	\begin{equation}
		\label{eq:Q_transformation}
		Q_{\Hat{\theta}}=e^{-(n+1)\Upsilon}(Q_\theta+P_{2n+2}\Upsilon).
	\end{equation}
	Moreover, $P_{2k}$ is formally self-adjoint with respect to $\theta\wedge(d\theta)^n$ for all $k\le n+1$, and
	the \emph{critical CR GJMS operator} $P_{2n+2}$ annihilates constant functions.
\end{thm}

By the construction of $g$ and the proofs of
Theorems \ref{thm:dn_operator} and \ref{thm:log_expansion_definition_of_Q},
$P_{2k}$ and $Q_\theta$ have universal expressions in terms of the Tanaka--Webster connection
(for $Q_\theta$, this is the same sense as for $\tensor{\mathcal{O}}{_\alpha_\beta}$;
the expressions for $P_{2k}$ involve the Tanaka--Webster covariant differentiations, of course).

Now we assume that $M=\partial X$ is compact. Then by the transformation law \eqref{eq:Q_transformation},
the self-adjointness of $P_{2n+2}$, and the fact that $P_{2n+2}1=0$, the \emph{total CR $Q$-curvature}
\begin{equation}
	\overline{Q}=\int_M Q_\theta\theta\wedge(d\theta)^n
\end{equation}
is a global CR invariant.
This is also characterized as the first log-term coefficient of the volume expansion of $g$,
and this viewpoint leads to the first variational formula of $\overline{Q}$ with respect to
deformations of partially integrable CR structures via the technique of
Graham and Hirachi~\cite{Graham_Hirachi_05}.
Since it is known that $\overline{Q}=0$ for any 3-dimensional compact CR manifold~\cite{Fefferman_Hirachi_03},
we assume $2n+1\ge 5$.
Recall from~\cite{Matsumoto_14}*{Proposition 6.1} that an infinitesimal deformation of
partially integrable CR structure is given by a density-weighted tensor
$\tensor{\psi}{_\alpha_\beta}\in\tensor{\mathcal{E}}{_(_\alpha_\beta_)}(1,1)$.
In Subsection \ref{subsec:variational_formula}, we prove the following.

\begin{thm}
	\label{thm:first_variational_formula}
	Let $(M,T^{1,0}M)$ be a compact partially integrable CR manifold of dimension $2n+1\ge 5$.
	Let $\psi=\tensor{\psi}{_\alpha_\beta}$ be an infinitesimal deformation of partially integrable CR structure
	and $T^{1,0}_t$ a smooth 1-parameter family of partially integrable CR structures
	with fixed underlying contact structure that is tangent to $\psi$ at $t=0$.
	Let $\smash{\overline{Q}}^t$ be the total CR $Q$-curvature of $(M,T^{1,0}_t)$. Then,
	\begin{equation}
		\label{eq:Variation_TotalQ}
		\left.\left(\frac{d}{dt}\smash{\overline{Q}}^t\right)\right|_{t=0}=
		\frac{4\cdot(-1)^n\cdot n!(n+1)!}{n+2}\int_M
		\re(\tensor{\mathcal{O}}{^\alpha^\beta}\tensor{\psi}{_\alpha_\beta}).
	\end{equation}
	Here, $\tensor{\mathcal{O}}{_\alpha_\beta}\in\tensor{\mathcal{E}}{_(_\alpha_\beta_)}(-n,-n)$
	is the CR obstruction tensor of $(M,T^{1,0}M)$, and the indices are raised by the weighted Levi form.
\end{thm}

This result in particular shows the nontriviality of $\overline{Q}$ in the partially integrable category,
while it seems a reasonable conjecture that $\overline{Q}=0$ for any integrable CR manifold
(see~\cite{Fefferman_Hirachi_03}).
We furthermore derive a formula of the Heisenberg principal symbol of $\mathcal{O}^\bullet$,
the linearization of the operator that gives the CR obstruction tensor, in
Theorem \ref{thm:linearized_obstruction}. It assures that $\mathcal{O}^\bullet$ is nowhere vanishing
when $2n+1\ge 5$, and consequently, that the locus $\set{\overline{Q}=0}$ in the space of
partially integrable CR structures on a given contact manifold is the complement of
an open dense subset with respect to the natural Fr\'echet topology,
for that $\mathcal{O}^\bullet$ is nowhere zero implies that the second variation of $\overline{Q}$
cannot vanish at the critical points of $\overline{Q}$.

Theorem \ref{thm:linearized_obstruction} follows from an interpretation of $\mathcal{O}^\bullet$
as a Dirichlet-to-Neumann-type operator, which is established in Proposition \ref{prop:lichnerowicz_equation}.
From the technical point of view, we remark that the proof of Proposition \ref{prop:lichnerowicz_equation}
involves a careful analysis of Laplacians and divergence operators using the asymptotic K\"ahlerity of
ACH metrics (Lemma \ref{prop:asymptotic_kahlerity}) and the behavior of the curvature tensor at the boundary
(Proposition \ref{prop:asymptotic_complex_hyperbolicity}).
One might be able to see the latter as the specialization of Stenzel's work~\cite{Stenzel_97} to ACH metrics.

If we restrict ourselves to the case of integrable CR manifolds, then the characterization of
$\mathcal{O}^\bullet$ given in Proposition \ref{prop:lichnerowicz_equation} can be studied more deeply with the
language of K\"ahler differential geometry, and we finally get the following properties of $\mathcal{O}^\bullet$
in Subsection \ref{subsec:further_properties}.

\begin{thm}
	\label{thm:obstruction_operator_observation}
	Let $(M,T^{1,0}M)$ be an integrable CR manifold of dimension $2n+1\ge 5$.

	(1) The linearized CR obstruction operator
	$\mathcal{O}^\bullet\colon\tensor{\mathcal{E}}{_(_\alpha_\beta_)}(1,1)\longrightarrow
	\tensor{\mathcal{E}}{_(_\alpha_\beta_)}(-n,-n)$ is a complex-linear, formally self-adjoint operator.

	(2) Let $D\colon\mathcal{E}(1,1)\longrightarrow\tensor{\mathcal{E}}{_(_\alpha_\beta_)}(1,1)$ be defined by
	$\tensor{(Df)}{_\alpha_\beta}
	=\tensor*{\nabla}{^{\mathrm{TW}}_\alpha}\tensor*{\nabla}{^{\mathrm{TW}}_\beta}f+i\tensor{A}{_\alpha_\beta}f$
	and $D^*$ its adjoint, where the density bundles are trivialized by some $\theta$ and
	$\nabla^\mathrm{TW}$ and $A$ denote the associated Tanaka--Webster connection and torsion, respectively.
	Then the following holds:
	\begin{equation}
		\label{eq:double_divergence_free}
		\mathcal{O}^\bullet D=0\qquad\text{and}\qquad
		D^*\mathcal{O}^\bullet=0.
	\end{equation}

	(3) Let $\mathcal{E}_\mathrm{N}(1,1)\subset\tensor{\mathcal{E}}{_\alpha_\beta_\gamma}(1,1)$ be the subspace of
	tensors with Nijenhuis-type symmetry, and
	$N^\bullet\colon\tensor{\mathcal{E}}{_(_\alpha_\beta_)}(1,1)\longrightarrow\mathcal{E}_\mathrm{N}(1,1)$
	the linearized Nijenhuis operator: $\tensor{(N^\bullet\psi)}{_\alpha_\beta_\gamma}
	=2\tensor*{\nabla}{^{\mathrm{TW}}_[_\alpha}\tensor{\psi}{_\beta_]_\gamma}$.
	Then $\mathcal{O}^\bullet$ can be decomposed as follows, where $B_\theta$ is a certain differential
	operator given by a universal formula in terms of the Tanaka--Webster connection:
	\begin{equation}
		\label{eq:linearized_obstruction_decomposition}
		\mathcal{O}^\bullet=B_\theta N^\bullet.
	\end{equation}
	In particular, $\mathcal{O}^\bullet$ vanishes on the space $\ker N^\bullet$ of
	integrable infinitesimal deformations.
\end{thm}

The operators involved here are schematically described as follows, where the weight $(w,w)$ is abbreviated as
$(w)$:
\begin{equation*}
\begin{tikzcd}[column sep=scriptsize]
	\mathcal{E}(1) \arrow{r}{D} &
	\tensor{\mathcal{E}}{_(_\alpha_\beta_)}(1) \arrow{r}{N^\bullet} \arrow[bend left]{rrrrr}{\mathcal{O}^\bullet} &
	\mathcal{E}_\mathrm{N}(1) \arrow[swap,bend right]{rrrr}{B_\theta} & & &
	\mathcal{E}_\mathrm{N}(-n+1) \arrow{r}{(N^\bullet)^*} &
	\tensor{\mathcal{E}}{_(_\alpha_\beta_)}(-n) \arrow{r}{D^*} &
	\mathcal{E}(-n-2).
\end{tikzcd}
\end{equation*}
The two arrows $\mathcal{E}(1)\overset{D}{\longrightarrow}\tensor{\mathcal{E}}{_(_\alpha_\beta_)}(1)
\overset{N^\bullet}{\longrightarrow}\mathcal{E}_\mathrm{N}(1)$ on the left are the first two operators in
the CR deformation complex~\cite{Rumin_94,Akahori_Garfield_Lee_02}; for the operator $D$,
which maps the infinitesimal ``Kuranishi wiggle'' to the resulting infinitesimal change of the CR structure,
see also~\cite{Hirachi_Marugame_Matsumoto_inprep}. The two arrows on the right are their formal adjoints.
Recall that, in the case of the flat CR structure, there exists a CR-invariant ``long'' operator
from each of the three spaces on the left to the corresponding dual space on the right, which is unique by
the composition series of the associated generalized Verma modules
(see~\cite{Collingwood_85}*{Section 8.2} and~\cite{Collingwood_88}*{6.2}):
\begin{equation*}
\begin{tikzcd}[column sep=scriptsize]
	\mathcal{E}(1) \arrow{r}{D} \arrow[bend left]{rrrrrrr}{L_1} &
	\tensor{\mathcal{E}}{_(_\alpha_\beta_)}(1) \arrow{r}{N^\bullet} \arrow[bend left]{rrrrr}{L_2} &
	\mathcal{E}_\mathrm{N}(1) \arrow[bend left]{rrr}{L_3} & & &
	\mathcal{E}_\mathrm{N}(-n+1) \arrow{r}{(N^\bullet)^*} &
	\tensor{\mathcal{E}}{_(_\alpha_\beta_)}(-n) \arrow{r}{D^*} &
	\mathcal{E}(-n-2).
\end{tikzcd}
\end{equation*}
It is known that the composition of any two operators in the latter diagram vanishes if $n\ge 3$,
and even when $n=2$ this is still true except for the compositions at
$\mathcal{E}_\mathrm{N}(1)$ and $\mathcal{E}_\mathrm{N}(-n+1)$.
Part (2) of the theorem above implies that the sequence
$\mathcal{E}(1)\longrightarrow\tensor{\mathcal{E}}{_(_\alpha_\beta_)}(1)
\longrightarrow\tensor{\mathcal{E}}{_(_\alpha_\beta_)}(-n)\longrightarrow\mathcal{E}(-n-2)$ remains to be
a complex for arbitrary integrable CR manifolds;
this should be compared with the work of Branson and Gover~\cite{Branson_Gover_07} in conformal geometry.
Moreover, in the flat case, one can check (by our Proposition \ref{prop:linearized_obstruction_heisenberg})
that the operator $L_2$ locally factors into the composition of $N^\bullet$, some non-CR-invariant
differential operator $\mathcal{E}_\mathrm{N}(1)\longrightarrow\mathcal{E}_\mathrm{N}(-n+1)$, and $(N^\bullet)^*$.
Part (3) of the theorem can be considered as a partial generalization of this to the curved case.

In Theorem \ref{thm:obstruction_operator_observation}, the nontrivial statements are part (1), in particular
the complex-linearity of $\mathcal{O}^\bullet$, and part (3). The other things are more or less easy to see.
The first equality of \eqref{eq:double_divergence_free} is most obvious among them:
since $(\mathcal{O}^\bullet Df)(p)$ depends only on finite jets of $f$ and the CR structure $T^{1,0}M$ at
$p\in M$, we can formally embed $M$ to $\mathbb{C}^{n+1}$
(see Kuranishi's article~\cite{Kuranishi_99} for example),
and in this case the assertion is clear because $Df$ always integrates to a genuine deformation of integrable CR
structure (see~\cite{Akahori_Garfield_Lee_02,Hirachi_Marugame_Matsumoto_inprep}).
If part (1) is taken as a given, then the second equality of \eqref{eq:double_divergence_free} is also clear.
Therefore what we should really discuss are parts (1) and (3), and this is done in Subsection
\ref{subsec:further_properties}.
Also in this subsection, we describe a direct proof of $D^*\mathcal{O}^\bullet=0$ via the
Dirichlet-to-Neumann-type characterization of $\mathcal{O}^\bullet$
for the interest of comparing with the proof of~\cite{Matsumoto_14}*{Theorem 0.2 (3)},
from which we knew that the \emph{imaginary part of $D^*\mathcal{O}^\bullet$} vanishes.
The reason why we have a stronger conclusion for $\mathcal{O}^\bullet$ here are,
firstly we have a K\"ahler structure for the bulk ACH-Einstein metric in this case,
and secondly we are able to characterize $\mathcal{O}^\bullet$ in terms of a PDE associated to this metric and
do not have to concern anymore about the bulk metrics for perturbed partially integrable CR structures.

I thank Kengo Hirachi for various discussions and suggestions on the exposition of the paper.
I am also grateful to Robin Graham for pointing out that Lemma~\ref{lem:Masakis_lemma} in this article
is discussed in~\cite{Graham_83}.
Regarding this lemma, I acknowledge that I was benefited as well from
discussions with Hideaki Hosaka, Masaki Mori, and Masaki Watanabe.

\section{Partially integrable CR manifolds}
\label{sec:partially_integrable}

\subsection{Basic definitions}

Let $(M,T^{1,0}M)$ be a partially integrable CR manifold (that is not necessarily nondegenerate).
A measurement of the failure of $(M,T^{1,0}M)$ not being integrable is given by the \emph{Nijenhuis tensor} $N$,
which is the real $(2,1)$-tensor over $H=\re T^{1,0}M$ whose complexification is given by
\begin{equation*}
	N(X,Y)=[X_{1,0},Y_{1,0}]_{0,1}+[X_{0,1},Y_{0,1}]_{1,0},\qquad
	X,\ Y\in C^\infty(M,\mathbb{C}H),
\end{equation*}
where the subscripts ``$1,0$'' and ``$0,1$'' denote the projections from
$\mathbb{C}H=T^{1,0}M\oplus\conj{T^{1,0}M}$ onto each summand.
If $\set{Z_\alpha}$ is a local frame of $T^{1,0}M$, putting $Z_{\conj{\alpha}}=\conj{Z_\alpha}$
we introduce the index notation with respect to $\set{Z_\alpha,Z_{\conj{\alpha}}}$ and its
dual coframe $\set{\theta^\alpha,\theta^{\conj{\alpha}}}$.
Since the Nijenhuis tensor $N$ is a $(1,2)$-tensor, $N$ is represented by a collection
of functions indexed with two lower indices and a single upper index.
In this case the only components that can be nonzero are
$\tensor{N}{_\alpha_\beta^{\conj{\gamma}}}$ and their complex conjugates
$\tensor{N}{_{\conj{\alpha}}_{\conj{\beta}}^\gamma}$.
Moreover, it is clear from the definition that $\tensor{N}{_\alpha_\beta^{\conj{\gamma}}}$ is skew-symmetric
in $\alpha$ and $\beta$.
Expressing symmetrization (resp.~skew-symmetrization) by round (resp.~square) brackets,
we can write as
\begin{equation}
	\label{eq:nijenhuis_skew_symmetry}
	\tensor{N}{_(_\alpha_\beta_)^{\conj{\gamma}}}=0,\qquad\text{or equivalently,}\qquad
	\tensor{N}{_[_\alpha_\beta_]^{\conj{\gamma}}}=\tensor{N}{_\alpha_\beta^{\conj{\gamma}}}.
\end{equation}
In Penrose's \emph{abstract index notation}~\cite{Penrose_Rindler_84}, which is used throughout this paper,
the symbol $\tensor{N}{_\alpha_\beta^{\conj{\gamma}}}$ is regarded as
denoting the $(T^{1,0}M)^*\otimes(T^{1,0}M)^*\otimes\conj{T^{1,0}M}$ part of the tensor $N$ itself,
not just its components.
Equations \eqref{eq:nijenhuis_skew_symmetry} is considered as an abstract expression of the skew-symmetry of $N$.
Furthermore, in abstract index notation, the vector bundle
$(T^{1,0}M)^*\otimes(T^{1,0}M)^*\otimes\conj{T^{1,0}M}$ in which
$\tensor{N}{_\alpha_\beta^{\conj{\gamma}}}$ takes values is denoted by $\tensor{E}{_\alpha_\beta^{\conj{\gamma}}}$,
and $\tensor{\mathcal{E}}{_\alpha_\beta^{\conj{\gamma}}}$ is the space of its sections.
By $\tensor{E}{_[_\alpha_\beta_]^{\conj{\gamma}}}$ we mean
the bundle of tensors with the symmetry as in \eqref{eq:nijenhuis_skew_symmetry},
and hence $\tensor{N}{_\alpha_\beta^{\conj{\gamma}}}\in\tensor{\mathcal{E}}{_[_\alpha_\beta_]^{\conj{\gamma}}}$
has the same meaning as \eqref{eq:nijenhuis_skew_symmetry}.

Suppose that $\theta$ is any (possibly locally-defined) nowhere-vanishing 1-form on $M$ that annihilates $H$.
Then the \emph{Levi form} $h$ is defined as follows:
\begin{equation}
	\label{eq:levi_form}
	h(Z,\conj{W}):=-i\,d\theta(Z,\conj{W})=i\,\theta([Z,\conj{W}]),
	\qquad\text{$Z$, $W\in C^\infty(M,T^{1,0}M)$}.
\end{equation}
The Levi form itself depends on $\theta$, but $\Hat\theta=e^\Upsilon\theta$ implies $\Hat{h}=e^\Upsilon h$.
Invariantly, we can define the $\mathbb{C}$-linear map
\begin{equation}
	\label{eq:weighted_levi_form}
	T^{1,0}M\otimes\conj{T^{1,0}M}\longrightarrow\mathbb{C}(TM/H),
	\qquad Z\otimes\conj{W}\longmapsto(i[Z,\conj{W}]\text{ mod $\mathbb{C}H$}).
\end{equation}
It is natural to call this $\mathbb{C}(TM/H)$-valued hermitian form the \emph{weighted Levi form}
by the reason explained in Subsection \ref{subsec:density_bundles}.

A partially integrable CR structure $T^{1,0}M$ is \emph{nondegenerate} if the Levi form is
nondegenerate at each point on $M$, which is equivalent to saying that $H$ is a contact distribution.
In this case, any choice of $\theta$ is called a \emph{contact form} or a \emph{pseudohermitian structure.}
The global existence of a contact form is equivalent to the triviality of $H^\perp\subset T^*M$.
A particular example of such situation is when the Levi form has definite signature,
in which case $T^{1,0}M$ is \emph{strictly pseudoconvex}.
As declared in Introduction, in the sequel we always assume that $T^{1,0}M$ is nondegenerate and $H^\perp$ is
oriented.

If a contact form $\theta$ is specified, then by the nondegeneracy of the Levi form, one can lower and raise
indices of various tensors. Note that \eqref{eq:levi_form} implies
\begin{equation}
	\label{eq:levi_form_differential}
	d\theta=i\,\tensor{h}{_\alpha_{\conj{\beta}}}\,\theta^\alpha\wedge\theta^\conjbeta\mod\theta.
\end{equation}
For example, we define
$\tensor{N}{_\alpha_\beta_\gamma}:=\tensor{h}{_\gamma_{\conj{\sigma}}}\tensor{N}{_\alpha_\beta^{\conj{\sigma}}}$.
Then by differentiating \eqref{eq:levi_form_differential} one can show that
$\tensor{N}{_\alpha_\beta_\gamma}+\tensor{N}{_\beta_\gamma_\alpha}+\tensor{N}{_\gamma_\alpha_\beta}=0$.
Choosing a contact form $\theta$ also enables us to pick a canonical vector field $T$,
called the \emph{Reeb vector field}, that is characterized by $\theta(T)=1$ and $T\contraction d\theta=0$.
Note that $T$ is transverse to $H$.
If $\set{Z_\alpha}$ is a local frame of $T^{1,0}M$, then the associated \emph{admissible coframe}
$\set{\theta^\alpha}$ is the collection of $1$-forms vanishing on $\mathbb{C}T\oplus\conj{T^{1,0}M}$ such that
$\set{\theta^\alpha|_{T^{1,0}M}}$ is the dual coframe for $\set{Z_\alpha}$.
This makes $\set{\theta,\theta^\alpha,\theta^{\conj{\alpha}}}$ into the dual coframe for
$\set{T,Z_\alpha,Z_{\conj{\alpha}}}$.
In our index notation of tensors, the index $0$ is used for components corresponding with $T$ or $\theta$.

We next introduce the Tanaka--Webster connection $\nabla$ on a nondegenerate partially integrable CR manifold.
Just as in the integrable case, $\nabla$ is a connection of $TM$ characterized by the fact that
$H$, $T$, $J$, $h$ are all parallel with respect to $\nabla$ and
the torsion tensor $\Tor(X,Y):=\nabla_XY-\nabla_YX-[X,Y]$ satisfies
\begin{subequations}
	\label{eq:torsion_tanaka_webster}
	\begin{alignat}{2}
		\label{eq:torsion_tanaka_webster_1}
			&\Tor(X, JY)-\Tor(JX,Y)=2\,h(X,Y)T, &\qquad &X, Y\in \Gamma(H),\\
		\label{eq:torsion_tanaka_webster_2}
			&\Tor(T, JX)=-J\Tor(T, X), &\qquad &X\in \Gamma(H).
	\end{alignat}
\end{subequations}
This definition leads to the following first structure equation,
where $\set{\theta^\alpha}$ is an admissible coframe and
$\tensor{\omega}{_\alpha^\beta}$ are the connection forms:
\begin{equation}
	\label{eq:FirstStructureEquation2}
	d\theta^{\gamma}
	=\theta^{\alpha}\wedge\tensor{\omega}{_\alpha^\gamma}
	-\tensor{A}{_{\conj{\alpha}}^\gamma}\theta^{\conj{\alpha}}\wedge\theta
	-\tfrac{1}{2}\tensor{N}{_{\conj{\alpha}}_{\conj{\beta}}^\gamma}
	\theta^{\conj{\alpha}}\wedge\theta^{\conj{\beta}}.
\end{equation}
The tensor $A$ is the \emph{Tanaka--Webster torsion tensor}.
Moreover, if $\Pi$ is the curvature of $\nabla$, then the component
$\tensor{\Pi}{_\alpha^\beta_\sigma_{\conj{\tau}}}$ is called the \emph{Tanaka--Webster curvature tensor}
and denoted by $\tensor{R}{_\alpha^\beta_\sigma_{\conj{\tau}}}$.
The other components of $\Pi$ can be written in terms of $N$, $A$, $R$, and their covariant derivatives.

\subsection{Density bundles}
\label{subsec:density_bundles}

Suppose first that we can take an $(n+2)$-nd root of the CR canonical bundle
$K=\smash{\bigwedge}^{n+1}(\conj{T^{1,0}M})^\perp$.
We fix such a line bundle $E(-1,0)$ and write its dual $E(1,0)$. We set
\begin{equation*}
	E(w,w'):=E(1,0)^{\otimes w}\otimes \smash{\conj{E(1,0)}}^{\otimes w'},\qquad\text{$w$, $w'\in\mathbb{Z}$},
\end{equation*}
and call it the \emph{density bundle of biweight $(w,w')$}.
The space of sections of $E(w,w')$ is denoted by $\mathcal{E}(w,w')$, and its elements are called \emph{densities}.
Since there is a canonical isomorphism $E(-n-2,0)\cong K$, we can uniquely
define a compatible connection $\nabla$ on $E(1,0)$.
The bundles and the spaces of density-weighted tensors are indicated by the usual symbols followed by the weight:
for example,
$\tensor{E}{_\alpha_{\conj{\beta}}}(w,w'):=\tensor{E}{_\alpha_{\conj{\beta}}}\otimes E(w,w')$ and the
space of its sections is $\tensor{\mathcal{E}}{_\alpha_{\conj{\beta}}}(w,w')$.

Farris~\cite{Farris_86} observed that, if $\zeta$ is a locally-defined nonvanishing section of $K$,
then there is a unique contact form $\theta$ satisfying
\begin{equation}
	\label{eq:farris_volume_normalization}
	\theta\wedge(d\theta)^{n}
	=i^{n^{2}}n!(-1)^{q}\theta\wedge(T\contraction\zeta)\wedge(T\contraction\conj{\zeta}),
\end{equation}
where $q$ is the number of the negative eigenvalues of the Levi form.
We say that this $\theta$ is \emph{volume-normalized} by $\zeta$.
If we replace $\zeta$ with $\lambda\zeta$, where $\lambda\in C^\infty(M,\mathbb{C}^{\times})$,
then $\theta$ changes to $\lvert\lambda\rvert^{2/(n+2)}\theta$.
We set
\begin{equation*}
	\abs{\zeta}^{2/(n+2)}=\zeta^{1/(n+2)}\otimes\smash{\conj{\zeta}}^{1/(n+2)}\in\mathcal{E}(-1,-1),
\end{equation*}
which is independent of the choice of the $(n+2)$-nd root of $\zeta$.
Let $\lvert\zeta\rvert^{-2/(n+2)}\in\mathcal{E}(1,1)$ be its inverse.
Then we obtain a CR-invariant section $\bm{\theta}$ of $T^*M\otimes E(1,1)$:
\begin{equation*}
	\bm{\theta}:=\theta\otimes\lvert\zeta\rvert^{-2/(n+2)}.
\end{equation*}
Since $\bm{\theta}$ determines a trivialization of
$\mathbb{C}H^\perp\otimes E(1,1)$, there is a canonical identification
\begin{equation}
	\label{eq:contact_form_density_identification}
	\mathbb{C}H^\perp\cong E(-1,-1).
\end{equation}
This is compatible with any Tanaka--Webster connection $\nabla$ because
it is easily observed that $\nabla\bm{\theta}=0$ (see~\cite{Gover_Graham_05}*{Proposition 2.1}).
Dually, there is an identification
\begin{equation}
	\mathbb{C}(TM/H)\cong E(1,1),\qquad (v\text{ mod $\mathbb{C}H$})\longmapsto\bm{\theta}(v).
\end{equation}
We may use these isomorphisms to \emph{define} $E(w,w)$ even if we cannot take an $(n+2)$-nd root of $K$.
Since the Levi form $\tensor{h}{_\alpha_{\conj{\beta}}}$ and $\theta$ have the same scaling factor,
\begin{equation*}
	\tensor{\bm{h}}{_\alpha_{\conj{\beta}}}:=
	\tensor{h}{_\alpha_{\conj{\beta}}}\otimes\theta^{-1}\in\tensor{\mathcal{E}}{_\alpha_{\conj{\beta}}}(1,1)
\end{equation*}
is a parallel CR-invariant tensor, where $\theta$ is considered as a density in $\mathcal{E}(-1,-1)$
via \eqref{eq:contact_form_density_identification}.
This is exactly the weighted Levi form given by \eqref{eq:weighted_levi_form}.
Moreover, $\theta\wedge(d\theta)^n$ multiplies by $e^{(n+1)\Upsilon}$ when $\theta$ is replaced by
$e^\Upsilon\theta$, and thus $E(-n-1,-n-1)$ is identified with the bundle of volume densities.

We say that any weighted symmetric tensor
$\tensor{\psi}{_\alpha_\beta}\in\tensor{\mathcal{E}}{_(_\alpha_\beta_)}(1,1)$ determines an
\emph{infinitesimal deformation of partially integrable CR structure}.
This is because, if $T^{1,0}M$ is modified to a new almost CR structure
$\Hat{T}^{1,0}=\spanof\set{\Hat{Z}_\alpha=Z_\alpha+\tensor{\varphi}{_\alpha^{\conj{\beta}}}Z_{\conj{\beta}}}$ by
$\tensor{\varphi}{_\alpha^{\conj{\beta}}}\in\tensor{\mathcal{E}}{_\alpha^{\conj{\beta}}}$,
then $\Hat{T}^{1,0}$ is partially integrable if and only if
$\tensor{\varphi}{_\alpha_\beta}=\tensor{\bm{h}}{_\alpha_{\conj{\gamma}}}\tensor{\varphi}{_\beta^{\conj{\gamma}}}$
is symmetric.
In the case where we are given a 1-parameter family $T^{1,0}_t$ of partially integrable CR structures
such that $T^{1,0}_0=T^{1,0}M$ and each $T^{1,0}_t$ is described as above by
$\tensor*{\varphi}{^t_\alpha_\beta}$ that is smooth in $t$, then we say that the family $T^{1,0}_t$ is
\emph{tangent} to the infinitesimal deformation
$\tensor{\psi}{_\alpha_\beta}=\tensor*{\varphi}{^\bullet_\alpha_\beta}$.
If moreover the original partially integrable CR structure $T^{1,0}M$ is integrable, then the
differential $N^\bullet$ of the Nijenhuis tensor is given by, for any choice of a contact form,
\begin{equation*}
	\tensor{(N^\bullet\psi)}{_\alpha_\beta_\gamma}
	=2\tensor*{\nabla}{^{\mathrm{TW}}_[_\alpha}\tensor{\psi}{_\beta_]_\gamma},
\end{equation*}
where the last index of
$\tensor{(N^\bullet\psi)}{_\alpha_\beta^{\conj{\gamma}}}\in\tensor{\mathcal{E}}{_\alpha_\beta^{\conj{\gamma}}}$
is lowered by the weighted Levi form.

\section{Summary on asymptotically complex hyperbolic Einstein metrics}
\label{sec:ACHE}

\subsection{$\Theta$-structures}

Let $X$ be a $C^\infty$-smooth manifold-with-boundary of dimension $2n+2$.
Suppose we are given a section $\Theta\in C^\infty(\bdry X,T^*X|_{\bdry X})$.
We assume the following conditions are satisfied, where $\iota\colon\bdry X\hookrightarrow X$ is the
inclusion map:
\begin{enumerate}
	\item $\iota^*\Theta$ is a nowhere vanishing 1-form on $\bdry X$;
	\label{item:nowhere_vanishing_condition}
	\item The kernel $H$ of $\theta=\iota^*\Theta$ is a contact distribution on $\partial X$.
	\label{item:contact_condition}
\end{enumerate}
A \emph{$\Theta$-structure} on $X$ is a conformal class $[\Theta]$ of elements of
$C^\infty(\bdry X,T^*X|_{\bdry X})$
satisfying (\ref{item:nowhere_vanishing_condition}) and (\ref{item:contact_condition}) above,
and a pair $(X,[\Theta])$ is called a \emph{$\Theta$-manifold}.
Any contact form on $\bdry X$ that belongs to the class $\iota^*[\Theta]$ is called a
\emph{compatible contact form}.
Note that picking up the conformal class $\iota^*[\Theta]$ amounts to fixing an orientation of
$H^\perp\subset T^*\bdry X$.

If we are given a $\Theta$-manifold $X$, there is a canonical smooth vector bundle $\thetatangent$,
which we call the \emph{$\Theta$-tangent bundle} of $X$.
Over the interior $\mathring{X}$, $(\thetatangent)|_{\mathring{X}}$ is canonically isomorphic to the usual
tangent bundle $T\mathring{X}$, while its structure near the boundary is described as follows.
Let $p\in\bdry X$, and $\set{N,T,Y_i}=\set{N,T,Y_1,\dots,Y_{2n}}$ a local frame of $TX$ in a neighborhood $U$ of
$p$ such that
\begin{itemize}
	\item $N|_{\bdry X}$ is annihilated by $[\Theta]$,
	\item $T$, $Y_1$, $\dots$, $Y_{2n}$ are tangent to $\bdry X\cap U$, and
	\item $\set{Y_1|_{\bdry X},\dots,Y_{2n}|_{\bdry X}}$ span the contact distribution $H$.
\end{itemize}
Then $\thetatangent|_U$ is spanned by $\set{\rho N,\rho^2T,\rho Y_i}$,
where $\rho\in C^\infty(X)$ is any boundary defining function.
The dual vector bundle of $\thetatangent$ is denoted by $\thetacotangent$,
and sections of tensor products of $\thetatangent$'s and $\thetacotangent$'s are called
\emph{$\Theta$-tensors} in general.
A fiber metric of $\thetatangent$ is called a \emph{$\Theta$-metric}.

A virtue of the concept of the $\Theta$-tangent bundle is
that the space $\mathcal{V}_\Theta$ of its sections is closed under the Lie bracket.
Due to this fact, the Levi-Civita connection associated to any $\Theta$-metric can be naturally
considered as a $\Theta$-connection on $\thetatangent$.
Here we say that $\nabla$ is a \emph{$\Theta$-connection} on a vector bundle $E$ if it is
an $\mathbb{R}$-linear mapping
\begin{equation*}
	\nabla\colon C^\infty(X;E)\longrightarrow C^\infty(X;\thetacotangent\otimes E),\qquad
	s\longmapsto(v\longmapsto\nabla_vs),
\end{equation*}
satisfying the usual Leibniz rule.
As a consequence, the Riemann curvature tensor of a $\Theta$-metric is regarded as a $\Theta$-tensor,
and so is the Ricci tensor.

A differential operator on functions is called a \emph{$\Theta$-differential operator}
if it is locally expressed as a polynomial in elements of $\mathcal{V}_\Theta$, and the set of such operators
is denoted by $\thetadiff(X)$.
If $E$ and $F$ are vector bundles over $X$, then the bundle version $\thetadiff(X;E,F)$ is similarly defined.
As is easily observed, $\Theta$-connections are typical examples of elements of
$\thetadiff(X;E,\thetacotangent\otimes E)$.

\subsection{ACH metrics}
\label{subsec:ACH_metrics}

We explain what we call ACH metrics only in a rather practical form; a more intrinsic definition can be
found in~\cite{Epstein_Melrose_Mendoza_91,Guillarmou_SaBarreto_08,Matsumoto_14}.
We start with a partially integrable CR manifold $(M,T^{1,0}M)$.
The \emph{standard $\Theta$-structure} on the product manifold $M\times[0,\infty)$
is the class $[\Theta]$ that annihilates the vector field $\partial_\rho=\partial/\partial\rho$,
where $\rho$ is the second coordinate of $M\times[0,\infty)$,
and pulls back to the given conformal class of contact forms on the boundary $M=M\times\set{0}$.
The manifold $X=M\times[0,\infty)$ equipped with the standard $\Theta$-structure is called the
\emph{product $\Theta$-manifold}.
Suppose we take a contact form $\theta$ on $M$ and a local frame $\set{Z_\alpha}$ of $T^{1,0}M$.
Then we write
\begin{equation*}
	\bm{Z}_\infty=\rho\partial_\rho,\qquad
	\bm{Z}_0=\rho^2T,\qquad
	\bm{Z}_\alpha=\rho Z_\alpha,\qquad\text{and}\qquad
	\bm{Z}_{\conj{\alpha}}=\rho Z_{\conj{\alpha}},
\end{equation*}
where $T$ is the Reeb vector field.
The set $\set{\bm{Z}_I}=\set{\bm{Z}_\infty,\bm{Z}_0,\bm{Z}_\alpha,\bm{Z}_{\conj{\alpha}}}$ spans the complexified
$\Theta$-tangent bundle $\mathbb{C}\thetatangent$.

If $U\subset M\times[0,\infty)$ is an open neighborhood of $M$, we say that a $\Theta$-metric $g$
defined on $U$ is an \emph{ACH metric} if, for some choice of $\theta$,
the boundary values of its components with respect to $\set{\bm{Z}_I}$ are as follows:
\begin{alignat*}{5}
	\tensor{g}{_\infty_\infty}&=4,&\qquad
	\tensor{g}{_\infty_0}&=0,&\qquad
	\tensor{g}{_\infty_\alpha}&=0,\\
	\tensor{g}{_0_0}&=1,&\qquad
	\tensor{g}{_0_\alpha}&=0,&\qquad
	\tensor{g}{_\alpha_{\conj{\beta}}}&=\tensor{h}{_\alpha_{\conj{\beta}}},&\qquad
	\tensor{g}{_\alpha_\beta}&=0&\qquad
	&\text{on $M$}.
\end{alignat*}
Here $\tensor{h}{_\alpha_{\conj{\beta}}}$ is the components of the Levi form with respect to $\set{Z_\alpha}$.

When $X$ is an arbitrary $\Theta$-manifold, by a \emph{compatible partially integrable CR structure} on
$M=\bdry X$ we mean a partially integrable CR structure whose underlying contact distribution is
the one induced by the $\Theta$-structure.
Then we define ACH metrics on $X$ as follows, where
an $\Theta$-diffeomorphism between $\Theta$-manifolds means
a diffeomorphism that preserves the $\Theta$-structures.

\begin{dfn}
	Let $X$ be a $\Theta$-manifold of dimension $2n+2$.
	Then a $\Theta$-metric $g$ on $X$ is called an \emph{ACH metric} if there exist following:
	\begin{enumerate}
		\item A compatible partially integrable CR structure $T^{1,0}M$ on $M=\bdry X$;
		\item An open neighborhood $U$ of $M$ in the product $\Theta$-manifold $M\times[0,\infty)$,
			an open neighborhood $V$ of $M$ in $X$,
			and a $\Theta$-diffeomorphism $\Phi\colon U\longrightarrow V$
			such that $\Phi|_M=\id_M$ and $\Phi^*g$ is an ACH metric on $U$.
	\end{enumerate}
	While $\Phi$ is not unique, it is known that $T^{1,0}M$ is determined by $g$.
	This is called the \emph{CR structure at infinity}.
\end{dfn}

An ACH metric $g$ on a neighborhood $U$ of $M\subset M\times[0,\infty)$ is called \emph{normalized} if
\begin{equation*}
	\tensor{g}{_\infty_\infty}=4,\qquad
	\tensor{g}{_\infty_0}=0,\qquad
	\tensor{g}{_\infty_\alpha}=0\qquad
	\text{everywhere in $U$}.
\end{equation*}
The following proposition, whose proof is given in~\cite{Guillarmou_SaBarreto_08},
makes this terminology reasonable.

\begin{prop}
	\label{prop:ACH_normalization}
	Let $X$ be a $\Theta$-manifold and $g$ a $C^\infty$-smooth ACH metric.
	If $\theta$ is any compatible contact form on $M=\bdry X$,
	then one can take a $\Theta$-diffeomorphism $\Phi\colon U\longrightarrow V$ so that $\Phi^*g$ is normalized.
	For any given $\theta$, the germ of $\Phi$ along $M$ is uniquely determined.
\end{prop}

This asserts in particular that there is a distinguished (germ of) boundary defining function(s) for
each $\theta$. This is called the \emph{model boundary defining function} in~\cite{Guillarmou_SaBarreto_08}.

\begin{ex}
	\label{ex:complex_hyperbolic}
	The model of ACH metrics is, of course, the complex hyperbolic metric.
	In the Siegel upper-half space model, it is the K\"ahler metric $g$ on
	$\Omega=\set{(z',w)\in\mathbb{C}^n\times\mathbb{C}|\im w>\abs{z'}^2}$
	with a global potential $4\log(1/r)$, where $r=\im w-\abs{z'}^2$:
	\begin{equation*}
		g=4\frac{\partial^2}{\partial z^i\partial\conj{z}^j}\left(\log\frac{1}{r}\right)dz^id\conj{z}^j
		=4\left(\frac{r_ir_{\conj{j}}}{r^2}-\frac{r_{i\conj{j}}}{r}\right)dz^id\conj{z}^j.
	\end{equation*}
	The boundary of $\Omega$ is the Heisenberg group $\mathcal{H}$, which is identified with
	$\mathbb{C}^n\times\mathbb{R}$ by $(z',w)\longmapsto(z',\re w)$.
	Extending this identification, we consider the following diffeomorphism between $\overline{\Omega}$ and
	$\mathcal{H}\times[0,\infty)=\mathbb{C}^n\times\mathbb{R}\times[0,\infty)$:
	\begin{equation*}
		\overline{\Omega}\longrightarrow\mathbb{C}^n\times\mathbb{R}\times[0,\infty),\qquad
		(z',w)\longmapsto(z',\re w,r=\im w-\abs{z'}^2).
	\end{equation*}
	Let $\Phi$ be the inverse of this mapping: $\Phi(z',t,r)=(z',t+i(r+\abs{z'}^2))$. Then,
	\begin{equation*}
		\Phi^*(r_idz^i)=\Phi^*\left(-\sum_{\alpha=1}^n\conj{z}^\alpha dz^\alpha+\frac{1}{2i}dw\right)
		=\frac{1}{2}dr-\frac{i}{2}dt
		-\frac{1}{2}\sum_{\alpha=1}^n(\conj{z}^\alpha dz^\alpha-z^\alpha d\conj{z}^\alpha)
		=\frac{1}{2}dr-i\theta,
	\end{equation*}
	where $\theta$ is the standard contact form. Therefore, by setting $\rho=\sqrt{r/2}$ we obtain
	\begin{equation*}
		\Phi^*g=\frac{dr^1}{r^2}+4\frac{\theta^2}{r^2}
		+\frac{4}{r}\sum_{\alpha=1}^ndz^\alpha d\conj{z}^\alpha
		=4\frac{d\rho^2}{\rho^2}+\frac{\theta^2}{\rho^4}
		+\frac{2}{\rho^2}\sum_{\alpha=1}^ndz^\alpha d\conj{z}^\alpha.
	\end{equation*}
	This shows that $g$ is an ACH metric on $\overline{\Omega}$ with $C^\infty$-structure replaced by
	the one that $\set{(z',t,\rho)}$ defines, and the CR structure at infinity is the standard one.
\end{ex}

\begin{ex}
	\label{ex:Bergman_type}
	We had to take the square root of $r$ in the example above.
	This generalizes to the square root construction of Epstein, Melrose,
	and Mendoza~\cite{Epstein_Melrose_Mendoza_91}.
	If $\Omega$ is a domain in a complex manifold $\mathcal{N}$ with $C^\infty$-smooth Levi nondegenerate boundary,
	then we can canonically define a $\Theta$-manifold $X$ called the \emph{square root} of $\overline{\Omega}$,
	which is
	\begin{itemize}
		\item $\cl{\Omega}$ with $C^\infty$-structure replaced by adjoining the square roots of $C^\infty$-smooth
			boundary defining functions; and
		\item equipped with the $\Theta$-structure given by the pullback of
			$\frac{i}{2}(\partial r-\conj{\partial}r)|_{T\bdry\Omega}$ by the identity map
			$\iota\colon X\longrightarrow\cl{\Omega}$.
	\end{itemize}
	The map $\iota$ is $C^\infty$-smooth but not vice versa;
	nevertheless it gives diffeomorphisms $\mathring{X}\cong\Omega$ and $\bdry X\cong\bdry\Omega$.
	Let $g$ be any Bergman-type metric on $\Omega$:
	\begin{equation}
		\label{eq:Bergman_type_metric}
		g=4\frac{\partial^2}{\partial z^i\partial\conj{z}^j}\left(\log\frac{1}{r}\right)dz^id\conj{z}^j,
	\end{equation}
	where $r$ is a boundary defining function with respect to the original $C^\infty$-structure of
	$\overline{\Omega}$ that is at least $C^2$.
	Then $g$ can be naturally interpreted as an ACH metric on $X$.
	The CR structure at infinity $\bdry X=M$ is exactly the integrable CR structure
	induced by the complex structure of $\mathcal{N}$.
\end{ex}

\subsection{Approximate solutions to the Einstein equation}
\label{subsec:approximate_ACHE}

Now we describe the existence and uniqueness result of approximate solutions of the Einstein equation
under the assumption of $C^\infty$ boundary regularity.

\begin{thm}
	\label{thm:existence}
	Let $X$ be a $\Theta$-manifold of dimension $2n+2$
	and $T^{1,0}M$ a compatible partially integrable CR structure on $M=\bdry X$.
	Then there exists a $C^\infty$-smooth ACH metric $g$ on $X$ with CR structure at infinity $T^{1,0}M$ for which
	\eqref{eq:approx_Einstein} is satisfied.
	Such a metric $g$ is, up to $\Theta$-diffeomorphism actions that restrict to the identity on $M$,
	unique modulo $O(\rho^{2n+2})$ symmetric 2-$\Theta$-tensors with $O(\rho^{2n+3})$ traces.
\end{thm}

This is essentially~\cite{Matsumoto_14}*{Theorem 1.1} (see Theorem 6.2 in the same article for the uniqueness),
but the approximate Einstein condition is slightly modified.
We only sketch the proof of this version here as the complete one can be found in~\cite{Matsumoto_13_Thesis}.
It is sufficient by Proposition \ref{prop:ACH_normalization}
to consider the case where $g$ is defined near the boundary of $M\times[0,\infty)$ and is normalized.
One proves under this assumption the existence of a metric $g$ satisfying
\eqref{eq:approx_Einstein} and the uniqueness up to $O(\rho^{2n+2})$ symmetric 2-$\Theta$-tensors with
$O(\rho^{2n+3})$ traces.
This is done inductively: we perturb $g$ by an $O(\rho^m)$ 2-$\Theta$-tensor $\psi$,
and compute the change $\Psi$ of $E:=\Ric+\frac{n+2}{2}g$ modulo $O(\rho^{m+1})$.
The conclusion is that, if $1\le m\le 2n+1$, the correspondence $\psi\longmapsto\Psi$ is
a one-to-one $\mathbb{R}$-linear mapping from $\mathcal{S}_m/\mathcal{S}_{m+1}$ onto itself, where
$\mathcal{S}_m$ denotes the space of $C^\infty$-smooth symmetric 2-$\Theta$-tensors that are $O(\rho^m)$.
As a result we obtain an ACH metric $g$ for which \eqref{eq:approx_Einstein_ricci} holds.
Then one fixes the trace of $g$ modulo $O(\rho^{2n+3})$ by requiring \eqref{eq:approx_Einstein_trace},
which needs a bit subtler observation of the map $\psi\longmapsto\Psi$.

Let $g$ be a $C^\infty$-smooth ACH metric satisfying \eqref{eq:approx_Einstein}.
If $\theta$ is a compatible contact form, then for the normalization of $g$ with respect to $\theta$,
we set $E=\rho^{2n+2}\tilde{E}$.
Then $\tensor{\tilde{E}}{_\alpha_\beta}|_M=\tilde{E}(\bm{Z}_\alpha,\bm{Z}_\beta)|_M$ is uniquely determined by
the CR structure at infinity and $\theta$, and it invariantly defines the \emph{CR obstruction tensor}
$\tensor{\mathcal{O}}{_\alpha_\beta}\in\tensor{\mathcal{E}}{_(_\alpha_\beta_)}(-n,-n)$.

\begin{rem}
	The approximate solution $g$ is necessarily even (up to the ambiguous terms)
	in the sense of Guillarmou and S\`a Barreto~\cite{Guillarmou_SaBarreto_08}*{Section 3.2}.
	This is because \eqref{eq:approx_Einstein} remains to be satisfied
	even if we formally replace $\rho$ with $-\rho$ in the expansion of $g$.
	However, we do not need the evenness in our subsequent discussion.
\end{rem}

In particular, consider the case in which $M$ is the boundary of a domain $\Omega$ in $\mathbb{C}^{n+1}$.
Then the Bergman-type metric \eqref{eq:Bergman_type_metric} given by Fefferman's approximate solution
$r$ to the complex Monge--Amp\`ere equation~\cite{Fefferman_76} satisfies,
if considered as a $\Theta$-metric on the square root $X$ of $\cl{\Omega}$,
\begin{equation}
	\label{eq:approx_Einstein_Kahler}
	\Ric=-\frac{n+2}{2}g+O(\rho^{2n+4}).
\end{equation}
Hence we have $\tensor{\mathcal{O}}{_\alpha_\beta}=0$ in this case. Moreover, since
$\tensor{\mathcal{O}}{_\alpha_\beta}$ admits an expression in terms of the Tanaka--Webster connection,
we conclude by formal embedding that $\tensor{\mathcal{O}}{_\alpha_\beta}=0$ holds for an arbitrary integrable
CR manifold. For details, see~\cite{Matsumoto_14}*{Proposition 5.5}.

\section{Dirichlet problems and volume expansion}
\label{sec:dirichlet_prob}

\subsection{Laplacian on functions}

We consider an arbitrary $C^\infty$-smooth normalized ACH metric $g$ defined near the boundary of
a $(2n+2)$-dimensional product $\Theta$-manifold $X=M\times[0,\infty)$.
The purpose of this subsection is to prove the following formula of the Laplacian of $g$.

\begin{prop}
	\label{prop:ACH_Laplacian}
	The Laplacian of a $C^\infty$-smooth ACH metric $g$ normalized with respect to
	$\theta$ is a $\Theta$-differential operator of the form
	\begin{equation}
		\label{eq:ACHLaplacianWithSublaplacian}
		\Delta=-\frac{1}{4}(\rho\partial_\rho)^2+\frac{n+1}{2}\rho\partial_\rho
		+\rho^2\Delta_b-\rho^4T^2+\rho\Psi,
		\qquad \Psi\in\thetadiff(X),
	\end{equation}
	where $\Delta_b$ is the sub-Laplacian and $T$ is the Reeb vector field associated to $\theta$.
\end{prop}

We start with the following expression of $g$, where $k$ is a 2-tensor over the subbundle
whose complexification is spanned by $\set{\bm{Z}_0,\bm{Z}_\alpha,\bm{Z}_{\conj{\alpha}}}$.
Here, $\set{\bm{\theta},\bm{\theta}^\alpha,\bm{\theta}^{\conj{\alpha}}}$ denotes the dual coframe:
\begin{equation*}
	g=4\frac{d\rho^2}{\rho^2}+k,\qquad
	k=\bm{\theta}^2+2\tensor{h}{_\alpha_{\conj{\beta}}}\bm{\theta}^\alpha\bm{\theta}^{\conj{\beta}}+O(\rho).
\end{equation*}
If we identify $(\thetatangent)|_{\mathring{X}}$ with $T\mathring{X}$,
then on each hypersurface $M_\rho=M\times\set{\rho}$, $k$ gives a Riemannian metric of $M_\rho$.
By the standard identification $M_\rho\cong M$, we can regard $k$ as
a 1-parameter family $k_\rho$ of Riemannian metrics on $M$. In this sense, by abusing the notation we write
\begin{equation}
	\label{eq:normalized_ACH}
	g=4\frac{d\rho^2}{\rho^2}+k_\rho.
\end{equation}
Let $\nabla^{k_\rho}$ be the Levi-Civita connection of $k_\rho$
and $\Delta^{k_\rho}\colon C^\infty(M)\longrightarrow C^\infty(M)$ the associated Laplacian on functions.
Then, as stated in~\cite{Guillarmou_SaBarreto_08}*{Equation (5.1)}, $\Delta$ is given by
\begin{equation}
	\label{eq:ACHLaplacian}
	\Delta=-\frac{1}{4}(\rho\partial_\rho)^2+\frac{n+1}{2}\rho\partial_\rho+\Delta^{k_\rho}
	-\frac{1}{8}\rho\partial_\rho(\log\abs{\det k_\rho})\rho\partial_\rho,
\end{equation}
where $\det k_\rho$ is the determinant of the matrix representing $k_\rho$ with respect to
$\set{\rho^2T,\rho Y_i}$.
Therefore, to show Proposition \ref{prop:ACH_Laplacian}, it suffices to verify the following lemma.

\begin{lem}
	\label{lem:parametrized_laplacian}
	In the situation above, $\Delta^{k_\rho}$ is a $\Theta$-differential operator and is expressed as
	\begin{equation}
		\label{eq:Laplacian_of_metric_on_level_sets}
		\Delta^{k_\rho}=\rho^2\Delta_b-\rho^4T^2+\rho\Psi,
		\qquad\Psi\in\thetadiff(X).
	\end{equation}
\end{lem}

\begin{proof}
	We compute using a local frame $\set{\bm{Z}_i}=\set{\bm{Z}_0,\bm{Z}_\alpha,\bm{Z}_{\conj{\alpha}}}$
	of $TM\cong TM_\rho$.
	Let us introduce the tensor $K$ defined by
	\begin{equation*}
		\nabla^{k_\rho}_{\bm{Z}_i}\bm{Z}_j=\nabla^\mathrm{TW}_{\bm{Z}_i}\bm{Z}_j+\tensor{K}{^k_i_j}\bm{Z}_k,
	\end{equation*}
	which measures the difference of $\nabla^{k_\rho}$ and the Tanaka--Webster connection $\nabla^\mathrm{TW}$.
	Then we obtain
	\begin{equation*}
		\Delta^{k_\rho}f
		=-\tensor{(k_\rho^{-1})}{^i^j}\tensor*{\nabla}{^{\mathrm{TW}}_i}\tensor*{\nabla}{^{\mathrm{TW}}_j}f
		-\tensor{(k_\rho^{-1})}{^i^j}\tensor{(k_\rho^{-1})}{^k^l}\tensor{K}{_k_i_j}
		\tensor*{\nabla}{^{\mathrm{TW}}_l}f,
		\qquad f\in C^\infty(M),
	\end{equation*}
	where the upper index of $K$ is lowered by $k_\rho$.
	The first term of the right-hand side expands as $\rho^2\Delta_bf-\rho^4T^2f+\rho\Psi'f$,
	where $\Psi'\in\thetadiff(X)$. On the other hand, one can show that
	\begin{equation*}
		\tensor{K}{_k_i_j}=
		\frac{1}{2}(\tensor*{\nabla}{^{\mathrm{TW}}_i}\tensor{(k_\rho)}{_j_k}
		+\tensor*{\nabla}{^{\mathrm{TW}}_j}\tensor{(k_\rho)}{_i_k}
		-\tensor*{\nabla}{^{\mathrm{TW}}_k}\tensor{(k_\rho)}{_i_j}
		-\tensor{\Tor}{_k_i_j}+\tensor{\Tor}{_i_j_k}+\tensor{\Tor}{_j_i_k}),
	\end{equation*}
	where $\tensor{\Tor}{^k_i_j}$ is the torsion of the Tanaka--Webster connection and again the first index is
	lowered by $k_\rho$. Since $\tensor{(k_\rho^{-1})}{^i^j}\tensor{(k_\rho^{-1})}{^k^l}
	(\tensor{\Tor}{_k_i_j}-\tensor{\Tor}{_i_j_k}-\tensor{\Tor}{_j_i_k})$ equals
	$-2\tensor{(k_\rho^{-1})}{^k^l}\tensor{\Tor}{^j_j_k}$ and $\tensor{\Tor}{^j_j_k}$ is actually zero, we have
	\begin{equation}
		\label{eq:diff_LC_TW_contracted}
		\tensor{(k_\rho^{-1})}{^i^j}\tensor{(k_\rho^{-1})}{^k^l}\tensor{K}{_k_i_j}=
		\frac{1}{2}\tensor{(k_\rho^{-1})}{^i^j}\tensor{(k_\rho^{-1})}{^k^l}
		(\tensor*{\nabla}{^{\mathrm{TW}}_i}\tensor{(k_\rho)}{_j_k}
		+\tensor*{\nabla}{^{\mathrm{TW}}_j}\tensor{(k_\rho)}{_i_k}
		-\tensor*{\nabla}{^{\mathrm{TW}}_k}\tensor{(k_\rho)}{_i_j}).
	\end{equation}
	Since $\nabla^\mathrm{TW}$ annihilates
	$\bm{\theta}^2+2\tensor{h}{_\alpha_{\conj{\beta}}}\bm{\theta}{^\alpha}\bm{\theta}{^{\conj{\beta}}}$,
	the right-hand side vanishes at $\rho=0$,
	from which we conclude that \eqref{eq:Laplacian_of_metric_on_level_sets} holds.
\end{proof}

\subsection{Construction of $P^g_{2k}$ and $Q^g_\theta$}

Proposition \ref{prop:ACH_Laplacian} provides sufficient knowledge of $\Delta$ to analyze
our Dirichlet-type problems.

\begin{thm}
	\label{thm:dn_operator}
	Let $g$ be a $C^\infty$-smooth ACH metric on a $(2n+2)$-dimensional $\Theta$-manifold $X$, and
	$T^{1,0}M$ its CR structure at infinity. We take a contact form $\theta$ on $M$ and $\rho$ denotes the
	associated model boundary defining function. Let $k$ be a positive integer.
	Then, for any real-valued function $f\in C^\infty(M)$, there exists $u\in C^\infty(\mathring{X})$ of the form
	\begin{equation}
		\label{eq:form_of_generalized_eigenfunction}
		u=\rho^{n+1-k}F+\rho^{n+1+k}\log\rho\cdot G,\qquad F,G\in C^\infty(X),\qquad F|_{\partial X}=f
	\end{equation}
	that solves
	\begin{equation}
		\label{eq:formal_eigenfunction_equation}
		\left(\Delta-\frac{(n+1)^2-k^2}{4}\right)u=O(\rho^\infty).
	\end{equation}
	The function $F$ is unique modulo $O(\rho^{2k})$, and $G$ is unique modulo $O(\rho^\infty)$.
	Moreover, there is a differential operator $P^g_{2k}\colon C^\infty(M)\longrightarrow C^\infty(M)$
	determined by $g$ and $\theta$ such that
	\begin{equation}
		\label{eq:definition_of_P}
		G|_M=c_kP^g_{2k}f,\qquad c_k=\frac{2\cdot(-1)^{k+1}}{k!(k-1)!}
	\end{equation}
	with the form
	\begin{equation}
		\label{eq:principal_part_of_GJMS}
		P^g_{2k}=\prod_{j=0}^{k-1}(\Delta_b+i(k-1-2j)T)+(\text{an operator with Heisenberg order $\le 2k-1$}),
	\end{equation}
	where $\Delta_b$ is the sub-Laplacian and $T$ is the Reeb vector field.
	The operators $P^g_{2k}$ are formally self-adjoint, and $P^g_{2n+2}$ annihilates constant functions.
\end{thm}

\begin{proof}
	By Proposition \ref{prop:ACH_normalization}, we may assume that $g$ is normalized.
	For any $s\in\mathbb{R}$, let
	\begin{equation*}
		\Delta_s:=\Delta-s(n+1-s).
	\end{equation*}
	We consider $F\in C^\infty(M\times[0,\infty))$ with expansion
	\begin{equation*}
		F\sim\sum_{j=0}^\infty\rho^jf^{(j)},\qquad f^{(j)}\in C^\infty(M),
	\end{equation*}
	and try to solve $\Delta_s(\rho^{2(n+1-s)}F)=O(\rho^\infty)$ by determining $f^{(j)}$'s.
	Proposition \ref{prop:ACH_Laplacian} implies
	\begin{equation*}
		\Delta_s(\rho^{2(n+1-s)+j}f^{(j)})
		\sim\rho^{2(n+1-s)+j}\cdot\left(-\frac{1}{4}j(j-4s+2n+2)f^{(j)}+\rho D^j_sf^{(j)}\right),
	\end{equation*}
	where $D^j_s$ is a formal power series in $\rho$ with coefficients in the space of
	linear differential operators on $M$.
	This formula is used to determine the expansion of $F$.
	First we set $f^{(0)}$ to be the given function $f$. If we write $F_0=f^{(0)}=f$, then
	\begin{equation*}
		\Delta_s(\rho^{2(n+1-s)}F_0)=\rho^{2(n+1-s)}(0+\rho D_{0,s}f)=O(\rho^{2(n+1-s)+1}).
	\end{equation*}
	We inductively define $f^{(j)}$, as far as $j-4s+2n+2\not=0$, by
	\begin{equation*}
		\frac{1}{4}j(j-4s+2n+2)f^{(j)}=
		(\text{the $\rho^{2(n+1-s)+j}$-coefficient of $\Delta_s(\rho^{2(n+1-s)}F_{j-1})$}),
	\end{equation*}
	and set $F_j:=F_{j-1}+\rho^jf^{(j)}$ so that $\Delta_sF_j=O(\rho^{2(n+1-s)+j+1})$.
	Then $F_j$ is written as follows using some linear differential operators $p_{l,s}$ on $M$:
	\begin{equation}
		\label{eq:Expansion_F_with_p}
		F_j=f+\rho p_{1,s}f+\rho^2p_{2,s}f+\dots+\rho^{j}p_{j,s}f.
	\end{equation}
	If we furthermore set $p_{0,s}:=1$, then $p_{j,s}$ is recursively given by
	\begin{equation*}
		\label{eq:recursive_formula_p}
		p_{j,s}
		=\frac{4}{j(j-4s+2n+2)}\sum_{l=0}^{j-1}(\text{the $\rho^{j-1-l}$-coefficient of $D^l_s$})p_{l,s}.
	\end{equation*}
	
	In our situation, $s$ is taken to be $(n+1+k)/2$, which implies that $4s-2n-2=2k$ is a positive integer.
	Hence the procedure above only works while $j<2k$.
	As a result, $F_{2k-1}$ is determined so that $\Delta_s(\rho^{n+1-k}F_{2k-1})=O(\rho^{n+1+k})$.
	In the next step, we have seen that in general we cannot solve
	$\Delta_s(\rho^{n+1-k}F_{2k})=O(\rho^{n+1+k+1})$ by polynomials in $\rho$, and so here we need the first
	logarithmic term. For $g^{(j)}\in C^\infty(M)$, we have
	\begin{multline*}
		\Delta_{(n+1+k)/2}(\rho^{n+1+k+j}\log\rho\cdot g^{(j)})
		=-\frac{1}{4}j(j+2k)\rho^{n+1+k+j}\log\rho\cdot g^{(j)}
		-\frac{1}{2}(j+k)\rho^{n+1+k+j}g^{(j)}\\
		\mod \rho^{n+1+k+j+2}(C^\infty(M\times[0,\infty))+\log\rho\cdot C^\infty(M\times[0,\infty))).
	\end{multline*}
	So we can uniquely take $g^{(0)}$ so that
	$\Delta_s(\rho^{n+1-k}F_{2k-1}+\rho^{n+1+k}\log\rho\cdot g^{(0)})=O(\rho^{n+1+k})$ holds:
	\begin{equation}
		\label{eq:LowestTerm_G}
		g^{(0)}:=\frac{2}{k}\tilde{p}_{2k,(n+1+k)/2}f,\qquad\text{where}\quad
		\tilde{p}_{2k,s}=\sum_{l=0}^{2k-1}(\text{the $\rho^{2k-1-l}$-coefficient of $D^l_s$})p_{l,s}.
	\end{equation}
	We set $G_0:=g^{(0)}$ and $F_{2k}:=F_{2k-1}+\rho^{2k}f^{(2k)}$ by choosing
	$f^{(2k)}\in C^\infty(M)$ arbitrarily.
	Then $\Delta_s(\rho^{n+1-k}F_{2k}+\rho^{n+1+k}\log\rho\cdot G_0)=O(\rho^{n+1+k})$, or more precisely,
	\begin{equation*}
		\Delta_s(\rho^{n+1-k}F_{2k}+\rho^{n+1+k}\log\rho\cdot G_0)=\rho^{n+1+k}(R_0+\log\rho\cdot S_0),
	\end{equation*}
	where $R_0$, $S_0\in C^\infty(M\times[0,\infty))$ vanish on the boundary.
	Now we can continue the induction to determine $F_{2k+j}$ and $G_j$, which are polynomials in $\rho$ of
	degrees $2k+j$ and $j$, respectively, by adding higher-order terms to $F_{2k}$ and $G_0$ so that
	\begin{equation*}
		\Delta_s(\rho^{n+1-k}F_{2k+j}+\rho^{n+1+k}\log\rho\cdot G_j)=\rho^{n+1+k}(R_j+\log\rho\cdot S_j)
	\end{equation*}
	with some $R_j$, $S_j\in C^\infty(M\times[0,\infty))$ that are $O(\rho^{j+1})$,
	which are again uniquely achieved.
	By Borel's Lemma, we obtain $F$ and $G\in C^\infty(M\times[0,\infty))$ such that
	\begin{equation*}
		\Delta_s(\rho^{n+1-k}F+\rho^{n+1+k}\log\rho\cdot G)=O(\rho^\infty).
	\end{equation*}
	Thus we obtain a solution of the form \eqref{eq:form_of_generalized_eigenfunction}, and
	the ambiguity lives only in $f_{2k}$.

	Self-adjointness of $P^g_{2k}$ can be shown by, as~\cite{Fefferman_Graham_02}*{Proposition 3.3}, looking at
	the logarithmic term in the expansion of the integral of $\braket{du_1,du_2}-\frac{1}{4}((n+1)^2-k^2)u_1u_2$,
	where $u_1$ and $u_2$ solve \eqref{eq:formal_eigenfunction_equation}. We omit the details.
	The fact that $P^g_{2n+2}1=0$ is immediate from the definition because $u=1$ is of course harmonic.

	Formula \eqref{eq:principal_part_of_GJMS} is shown as follows, as in~\cite{Graham_83}.
	If we write $\rho D^j_s=\rho^2\Delta_b-\rho^4T^2+\tilde{D}^j_s$,
	then by Proposition \ref{prop:ACH_Laplacian},
	the Heisenberg order of the $\rho^w$-coefficient of $\tilde{D}^j_s$ is less than $w$. Hence, if we set
	\begin{equation*}
		p_{2l,s}=c_{2l,s}P_{2l,s},\qquad\text{where}\quad
		c_{2l,s}:=\prod_{\nu=0}^{l-1}\frac{1}{(l-\nu)(l-\nu-2s+n+1)},
	\end{equation*}
	then modulo differential operators of Heisenberg order $\le 2l-1$,
	\begin{equation*}
		\begin{split}
			P_{2l,(n+1+k)/2}
			&=\Delta_bP_{2l-2,(n+1+k)/2}-(l-1)(k-l+1)(iT)^2P_{2l-4,(n+1+k)/2}.
		\end{split}
	\end{equation*}
	Since $P^g_{2k}$ equals $P_{2k,(n+1+k)/2}$ by \eqref{eq:LowestTerm_G},
	the proof of \eqref{eq:principal_part_of_GJMS} is reduced to the next lemma.
\end{proof}

\begin{lem}
	\label{lem:Masakis_lemma}
	Let $k$ be a fixed nonnegative integer. If we define the polynomials $q_l\in\mathbb{C}[x,y]$ by
	\begin{equation*}
		q_0=1,\qquad
		q_1=x,\qquad
		q_l=xq_{l-1}-(l-1)(k-l+1)y^2q_{l-2},
	\end{equation*}
	then $q_k=\prod_{j=0}^{k-1}(x+(k-1-2j)y)$.
\end{lem}

For the proof of Lemma~\ref{lem:Masakis_lemma}, see~\cite{Graham_83}*{2.5.~Proposition}
or~\cite{Graham_84}*{Section 1}.

One can also observe by the proof of Theorem \ref{thm:dn_operator} above
that the operator $P_{2k,s}$ is a polynomial with respect to $s$.
This allows us to construct $Q^g_\theta$ by the Graham--Zworski argument.
Here we do not take this route, as explained in Introduction, and prove the following theorem instead.

\begin{thm}
	\label{thm:log_expansion_definition_of_Q}
	Let $g$ be a $C^\infty$-smooth ACH metric on a $(2n+2)$-dimensional $\Theta$-manifold $X$, and
	$T^{1,0}M$ its CR structure at infinity. We take a contact form $\theta$ on $M$ and $\rho$ denotes the
	associated model boundary defining function. Then there exists $U\in C^\infty(\mathring{X})$ of the form
	\begin{equation}
		\label{eq:log_expansion}
		U=\log\rho+A+B\rho^{2n+2}\log\rho,\qquad
		A, B\in C^\infty(X),\qquad
		A|_M=0
	\end{equation}
	such that
	\begin{equation}
		\label{eq:equation_for_log_expansion}
		\Delta U=(n+1)/2\mod O(\rho^\infty).
	\end{equation}
	The function $A$ is unique modulo $O(\rho^{2n+2})$ and $B$ is unique modulo $O(\rho^\infty)$.
	Moreover, we define $Q^g_\theta\in C^\infty(M)$ by
	\begin{equation}
		B|_M=\frac{(-1)^n}{n!(n+1)!}Q^g_\theta.
	\end{equation}
	Then, for any two contact forms $\theta$ and $\Hat{\theta}=e^\Upsilon\theta$, the following holds:
	\begin{equation}
		\label{eq:Qg_transformation}
		Q^g_{\Hat{\theta}}=e^{-(n+1)\Upsilon}(Q^g_\theta+P^g_{2n+2}\Upsilon).
	\end{equation}
\end{thm}

\begin{proof}
	Again we may assume that $g$ is normalized.
	By Proposition \ref{prop:ACH_Laplacian}, $\Delta\log\rho$ is $C^\infty$-smooth up to the boundary
	and $(\Delta\log\rho)|_M=(n+1)/2$.
	Moreover, if $a$, $b\in C^\infty(M)$, then
	\begin{equation}
		\label{eq:ACHLaplacianActingOnSmoothTerm}
		\Delta(\rho^ja)=-\frac{1}{4}j(j-2n-2)\rho^ja+O(\rho^{j+1})
	\end{equation}
	and
	\begin{equation}
		\label{eq:ACHLaplacianActingOnLogTerm}
		\Delta(\rho^j\log\rho\cdot b)=
		-\frac{1}{2}(j-n-1)\rho^jb+O(\rho^{j+1})-\frac{1}{4}\log\rho\cdot(j(j-2n-2)\rho^jb+O(\rho^{j+1})),
	\end{equation}
	where the terms denoted by $O(\rho^{j+2})$ are all $C^\infty$-smooth.
	By using \eqref{eq:ACHLaplacianActingOnSmoothTerm} inductively, we can show that there is
	a unique finite expansion
	\begin{equation*}
		A_{2n+1}=\sum_{j=1}^{2n+1}\rho^ja^{(j)},\qquad a^{(j)}\in C^\infty(M)
	\end{equation*}
	such that $\Delta(\log\rho+F_{2n+1})=(n+1)/2+O(\rho^{2n+2})$.
	The next thing to do is to introduce a $(\rho^{2n+2}\log\rho)$-term so that $\Delta$ applied to it kills
	the $\rho^{2n+2}$-coefficient of the error term.
	This is possible in view of \eqref{eq:ACHLaplacianActingOnLogTerm} because $j-n-1$ is nonzero for $j=2n+2$:
	there uniquely exists $b^{(0)}\in C^\infty(M)$ for which
	\begin{equation*}
		\Delta(\log\rho+A_{2n+1}+\rho^{2n+2}\log\rho\cdot b^{(0)})=\frac{n+1}{2}+O(\rho^{2n+3}\log\rho).
	\end{equation*}
	We set $B_0=b^{(0)}$ while $A_{2n+2}=A_{2n+1}+\rho^{2n+2}a^{(2n+2)}$,
	where we choose $a^{(2n+2)}\in C^\infty(M)$ arbitrarily.
	Equation \eqref{eq:ACHLaplacianActingOnSmoothTerm} implies that
	$\Delta(\log\rho+A_{2n+2}+\rho^{2n+2}\log\rho\cdot B_0)=(n+1)/2+O(\rho^{2n+3}\log\rho)$ holds.
	Then, inductively in $k$, we can determine $a^{(2n+2+k)}$, $b^{(k)}\in C^\infty(M)$ such that
	\begin{equation*}
		A_{2n+2+k}=A_{2n+2+k-1}+\rho^{2n+2+k}a^{(2n+2+k)}\qquad\text{and}\qquad
		B_k=B_{k-1}+\rho^{2n+2+k}\log\rho\cdot b^{(k)}
	\end{equation*}
	satisfies $\Delta(\log\rho+A_{2n+2+k}+\rho^{2n+2}\log\rho\cdot B_k)=(n+1)/2+O(\rho^{2n+2+k+1}\log\rho)$.
	Finally, by Borel's Lemma, we obtain $A$, $B\in C^\infty(M\times[0,\infty))$ for which
	$\Delta(2\log\rho+A+\rho^{2n+2}\log\rho\cdot B)=n+1+O(\rho^\infty)$.
	The construction shows the ambiguity of $A$ and $B$ are as stated.

	To show \eqref{eq:Qg_transformation}
	we first remark that $\tilde{\Upsilon}|_M=\Upsilon$ if $\Hat{\rho}=e^{\tilde{\Upsilon}/2}\rho$ is the model
	boundary defining function for the new contact form $\Hat{\theta}=e^\Upsilon\theta$. Let
	$\Hat{U}=\log\Hat{\rho}+\Hat{A}+\Hat{B}\Hat{\rho}^{2n+2}\log\Hat{\rho}$
	be a solution of \eqref{eq:equation_for_log_expansion} associated to $\Hat{\rho}$. Then
	\begin{equation*}
		\Hat{U}
		=\log\rho+\frac{\tilde{\Upsilon}}{2}+\Hat{A}
		+e^{(n+1)\tilde{\Upsilon}}\Hat{B}\rho^{2n+2}\left(\log\rho+\frac{\tilde{\Upsilon}}{2}\right)
	\end{equation*}
	and hence
	\begin{equation*}
		\Hat{U}-U
		=\left(\frac{\tilde{\Upsilon}}{2}+\Hat{A}
			+\frac{1}{2}e^{(n+1)\tilde{\Upsilon}}\rho^{2n+2}\Hat{B}\tilde{\Upsilon}\right)
		+\rho^{2n+2}\log\rho\cdot(e^{(n+1)\tilde{\Upsilon}}\Hat{B}-B).
	\end{equation*}
	Since $\Hat{U}-U$ solves $\Delta(\Hat{U}-U)=O(\rho^\infty)$, by Theorem \ref{thm:dn_operator}, we have
	\begin{equation*}
		e^{(n+1)\Upsilon}\Hat{B}|_M-B|_M=2^{-1}c_{n+1}P^g_{2n+2}\Upsilon,
	\end{equation*}
	or equivalently $e^{(n+1)\Upsilon}Q^g_{\Hat{\theta}}-Q^g_\theta=P^g_{2n+2}\Upsilon$,
	which is \eqref{eq:Qg_transformation}.
\end{proof}

\subsection{Volume expansion}

We relate the integral of $Q^g_\theta$ with the volume expansion of $g$.
Since $dV_g$ extends to a nowhere vanishing section of the $\Theta$-volume bundle $\abs{\det\thetacotangent}$,
it diverges at the rate of $\rho^{-2n-3}$ at $\bdry X$ as the usual volume density on $X$.
Hence if $M$ is compact, for some arbitrarily fixed $\varepsilon_0$,
the volume of the subset $\set{\varepsilon\le\rho\le\varepsilon_0}\subset X$ has
the following asymptotic behavior when $\varepsilon\to 0$:
\begin{equation}
	\label{eq:volume_expansion}
	\Vol(\set{\varepsilon\le\rho\le\varepsilon_0})
	=\sum_{j=-2n-2}^{-1}c_j\varepsilon^j+L\log\frac{1}{\varepsilon}+O(1).
\end{equation}

\begin{prop}
	\label{prop:volume_log_term_and_Q}
	Let $M$ be compact, and $g$ a $C^\infty$-smooth ACH metric normalized with respect to $\theta$.
	Then the coefficient $L$ in \eqref{eq:volume_expansion} is given by
	\begin{equation}
		\label{eq:VolumeLogTermAndQ}
		L=\frac{2\cdot(-1)^{n+1}}{n!^2(n+1)!}\smash{\overline{Q}}^g,\qquad
		\smash{\overline{Q}}^g:=\int_MQ^g_\theta\,\theta\wedge(d\theta)^n=n!\int_MQ^g_\theta\,dV_{\theta^2+h}.
	\end{equation}
\end{prop}

\begin{proof}
	We use Green's formula.
	Let $g=4\rho^{-2}d\rho^2+k_\rho$, and we define the family $\tilde{k}_\rho$ of Riemannian metrics
	on $M$ by
	\begin{equation*}
		\tilde{k}_\rho=\delta_\rho^*k_\rho,\qquad\text{where}\quad
		\delta_\rho T=\rho^2 T\quad\text{and}\quad \delta_\rho Y_i=\rho Y_i.
	\end{equation*}
	Then $\tilde{k}_\rho$ smoothly extends to $\rho=0$ and $\tilde{k}_0=\theta^2+h$.
	Since the outward unit normal vector field along the hypersurface
	$M_\varepsilon=\set{\rho=\varepsilon}$ is $-\frac{1}{2}\varepsilon\partial_\rho$, for any $U\in C^2(X)$
	\begin{equation*}
		\int_{\varepsilon\le\rho\le\varepsilon_0}\Delta U\,dV_g
		=\int_{M_\varepsilon}\frac{1}{2}\varepsilon(\partial_\rho U) dV_{k_\varepsilon}+O(1)
		=\frac{1}{2}\varepsilon^{-2n-1}\int_{M_\varepsilon}(\partial_\rho U)dV_{\tilde{k}_\varepsilon}+O(1).
	\end{equation*}
	Now let $U$ be a solution of the Dirichlet problem in Theorem \ref{thm:log_expansion_definition_of_Q}.
	Then since $\Delta U=(n+1)/2+O(\rho^\infty)$, the equality above reads
	\begin{equation*}
		\frac{n+1}{2}\Vol(\set{\varepsilon\le\rho\le\varepsilon_0})
		=\frac{1}{2}\varepsilon^{-2n-1}\int_{M_\varepsilon}(\partial_\rho U)dV_{\tilde{k}_\varepsilon}+O(1).
	\end{equation*}
	Comparing the coefficients of $\log(1/\varepsilon)$ we can see
	\begin{equation*}
		\frac{n+1}{2}L=-(n+1)\int_M(B|_M)dV_{\theta^2+h}
		=\frac{(-1)^{n+1}}{n!^2}\int_MQ^g_\theta\,dV_{\theta^2+h}.
	\end{equation*}
	Thus we obtain \eqref{eq:VolumeLogTermAndQ}.
\end{proof}

\section{GJMS operators and $Q$-curvature}

\subsection{Invariance of $P^g_{2k}$ and $Q^g_\theta$ for approximately Einstein ACH metrics}
\label{subsec:GJMS}

Now we show Theorem \ref{thm:dn_operator_of_ACHE} by applying the results of the previous section
to approximate ACH-Einstein metrics $g$.
In order to carry out this idea, we have to discuss the dependence of $\Delta$ on the ambiguity of $g$
to ensure that it is not problematic for our construction.

\begin{lem}
	\label{lem:dependence_of_ACH_laplacian}
	Let $(M,T^{1,0}M)$ be a partially integrable CR manifold, and $g$ a $C^\infty$-smooth normalized ACH metric
	satisfying the approximate Einstein condition \eqref{eq:approx_Einstein}.
	If $F\in C^\infty(X)$, then $\Delta F$ modulo $O(\rho^{2n+3})$ is irrelevant to the ambiguity of $g$.
	Moreover, if $\rho$ is a $C^\infty$-smooth boundary defining function, then
	$\Delta(\log\rho)$ is irrelevant modulo $O(\rho^{2n+3}\log\rho)$ to the ambiguity of $g$.
\end{lem}

\begin{proof}
	By \eqref{eq:diff_LC_TW_contracted},
	$\tensor{(k_\rho^{-1})}{^i^j}\tensor{(k_\rho^{-1})}{^k^l}\tensor{K}{_k_i_j}\tensor{\nabla}{_l}F$ is
	uniquely determined modulo $O(\rho^{2n+3})$.
	Since $\tensor{(k_\rho^{-1})}{^i^j}$ is determined modulo $O(\rho^{2n+2})$ and
	$\nabla_i\nabla_jF$ is $O(\rho)$, $\tensor{(k_\rho^{-1})}{^i^j}\nabla_i\nabla_jF$ is also determined modulo
	$O(\rho^{2n+3})$.
	Thus we conclude by Lemma \ref{lem:parametrized_laplacian} that $\Delta^{k_\rho}F$ is determined up to the
	error of $O(\rho^{2n+3})$.
	Furthermore, the trace condition imposed on $g$ implies that
	$\abs{\det k_\rho}$ is determined modulo $O(\rho^{2n+3})$,
	and hence so is $\rho\partial_\rho(\log\abs{\det k_\rho})$.
	Therefore, by Proposition \ref{prop:ACH_Laplacian}, $\Delta F$ is determined modulo $O(\rho^{2n+3})$.
	The second statement is shown similarly.
\end{proof}

\begin{proof}[Proof of Theorem \ref{thm:dn_operator_of_ACHE}]
	Recall the Dirichlet-type problem in Theorem \ref{thm:dn_operator}.
	Lemma \ref{lem:dependence_of_ACH_laplacian} and the proof of Theorem \ref{thm:dn_operator} show that,
	if $k\le n+1$, then the correspondence of $F|_M=f$ and $G|_M$ depends only on $(M,T^{1,0}M)$ and $\theta$,
	and not on the ambiguity of $g$.
	Therefore, the operators $P^g_{2k}$ for such $k$ are determined invariantly.
	Likewise, the function $B|_M$ in Theorem \ref{thm:log_expansion_definition_of_Q} depends only on
	$(M,T^{1,0}M)$ and $\theta$, and so is $Q^g_\theta$.
\end{proof}

\subsection{First variational formula of $\overline{Q}$}
\label{subsec:variational_formula}
 
We prove Theorem \ref{thm:first_variational_formula} using the characterization of $\overline{Q}$
given in Proposition \ref{prop:volume_log_term_and_Q}.
The key to the proof is introducing the first logarithmic term to our $C^\infty$-smooth approximate ACH-Einstein
metric $g^\mathrm{smooth}$, namely, one satisfying \eqref{eq:approx_Einstein}.
Let $g^\mathrm{smooth}=4\rho^{-2}d\rho^2+k_\rho$ be normalized by $\theta$,
and $\tensor{\mathcal{O}}{_\alpha_\beta}$ the CR obstruction tensor trivialized by $\theta$.
We define the new ACH metric $g$ by correcting $\tensor*{g}{^{\mathrm{smooth}}_\alpha_\beta}$ with the
additional term $4(n+1)^{-1}\cdot\tensor{\mathcal{O}}{_\alpha_\beta}\rho^{2n+2}\log\rho$;
or, in the matrix form with respect to
$\set{\bm{Z}_I}=\set{\bm{Z}_\infty,\bm{Z}_0,\bm{Z}_\alpha,\bm{Z}_{\conj{\alpha}}}$,
\begin{equation*}
	g=\begin{pmatrix}
		4 & 0 & 0 & 0 \\
		0 & 1 & \tensor{(k_\rho)}{_0_\beta} & \tensor{(k_\rho)}{_0_{\conj{\beta}}} \\
		0 & \tensor{(k_\rho)}{_\alpha_0}
		& \tensor{(k_\rho)}{_\alpha_\beta}+\frac{4}{n+1}\tensor{\mathcal{O}}{_\alpha_\beta}\rho^{2n+2}\log\rho
		& \tensor{(k_\rho)}{_\alpha_{\conj{\beta}}} \\
		0 & \tensor{(k_\rho)}{_{\conj{\alpha}}_0} & \tensor{(k_\rho)}{_{\conj{\alpha}}_\beta}
		& \tensor{(k_\rho)}{_{\conj{\alpha}}_{\conj{\beta}}}
		+\frac{4}{n+1}\tensor{\mathcal{O}}{_{\conj{\alpha}}_{\conj{\beta}}}\rho^{2n+2}\log\rho \\
	\end{pmatrix}.
\end{equation*}
Then the vanishing order of the $(\alpha\beta)$-component of $E={\Ric}+\frac{n+2}{2}g$ improves to
$O(\rho^{2n+3}\log\rho)$; see~\cite{Matsumoto_14}*{Equations (8.2) and (8.7)}.

\begin{proof}[Proof of Theorem \ref{thm:first_variational_formula}]
	Let $T_t^{1,0}$ be a smooth 1-parameter family of partially integrable CR structures
	that is tangent to the given infinitesimal deformation
	$\tensor{\psi}{_\alpha_\beta}\in\tensor{\mathcal{E}}{_(_\alpha_\beta_)}(1,1)$.
	Then the construction of $g$ above shows that we can take a family $g^t$ of such ACH metrics,
	each of which is associated to $T_t^{1,0}$ and satisfies
	$E^t=\Ric(g^t)+\frac{n+2}{2}g^t=O(\rho^{2n+3}\log\rho)$,
	so that the coefficients of each components of $g_t$ smoothly depend on the parameter $t$.
	Let $\set{Z_\alpha}$ be a local frame of the original partially integrable CR structure $T_0^{1,0}=T^{1,0}M$.
	If we compute with $\set{\bm{Z}_0,\bm{Z}_\alpha,\bm{Z}_{\conj{\alpha}}}$,
	where $\bm{Z}_\alpha=\rho Z_\alpha$,
	then the derivatives of the components of $k_\rho^t$ at $t=0$, which we write $k_\rho^\bullet$, are
	\begin{alignat*}{2}
		\tensor*{(k_\rho^\bullet)}{_0_0}&=O(\rho),&\qquad
		\tensor*{(k_\rho^\bullet)}{_0_\alpha}&=O(\rho),\\
		\tensor*{(k_\rho^\bullet)}{_\alpha_\conjbeta}&=O(\rho),&\qquad
		\tensor*{(k_\rho^\bullet)}{_\alpha_\beta}&=-2\tensor{\psi}{_\alpha_\beta}+O(\rho).
	\end{alignat*}
	Now we start with the fact that there is a uniform estimate on the scalar curvature of $g^t$:
	\begin{equation*}
		\Scal^t=-(n+1)(n+2)+O(\rho^{2n+3}).
	\end{equation*}
	From this we can see that
	\begin{equation*}
		\int_{\varepsilon\le\rho\le\varepsilon_0}\Scal^\bullet dV_g=O(1)\qquad\text{as $\varepsilon\to 0$}.
	\end{equation*}
	On the other hand, the well-known formula of the first variation of the scalar curvature implies
	\begin{equation}
		\label{eq:variation_of_scalar_curvature}
		\begin{split}
			\Scal^\bullet
			&=\tensor{(g^\bullet)}{_I_J_,^I^J}-\tensor{(g^\bullet)}{_I^I_,_J^J}
			-\tensor{\Ric}{^I^J}\tensor*{g}{^\bullet_I_J}\\
			&=\tensor{(g^\bullet)}{_I_J_,^I^J}-\tensor{(g^\bullet)}{_I^I_,_J^J}
			+\frac{1}{2}(n+2)\tensor{g}{^I^J}\tensor*{g}{^\bullet_I_J}+O(\rho^{2n+4}\log\rho).
		\end{split}
	\end{equation}
	Since $dV_g^\bullet=\frac{1}{2}\tensor{g}{^I^J}\tensor*{g}{^\bullet_I_J}dV_g$,
	\eqref{eq:variation_of_scalar_curvature} integrates to
	\begin{equation}
		\label{eq:integral_of_variation_of_scalar_curvature}
		\int_{\varepsilon\le\rho\le\varepsilon_0}
		(\tensor{(g^\bullet)}{_I_J_,^I^J}-\tensor{(g^\bullet)}{_I^I_,_J^J})dV_g
		+(n+2)\int_{\varepsilon\le\rho\le\varepsilon_0}dV^\bullet_g=O(1).
	\end{equation}
	The unit outward normal vector for $g$ along $\set{\rho=\varepsilon}$ is
	$\nu=-\frac{1}{2}\varepsilon\partial_\rho$.
	Therefore, \eqref{eq:integral_of_variation_of_scalar_curvature} implies
	\begin{equation}
		\label{eq:VariationOfVolume}
		\begin{split}
			(n+2)\Vol(\set{\varepsilon\le\rho\le\varepsilon_0})^\bullet
			&=-\int_{\rho=\varepsilon}
			(\tensor{(g^\bullet)}{_I_J_,^I}-\tensor{(g^\bullet)}{_I^I_,_J})\tensor{\nu}{^J}d\sigma+O(1)\\
			&=
			\int_M(\tensor{(g^\bullet)}{_I_J_,^I}-\tensor{(g^\bullet)}{_I^I_,_J})\cdot
			\frac{1}{2}\varepsilon\tensor{\delta}{_\infty^J}\cdot\varepsilon^{-2n-2}dV_{k_\varepsilon}+O(1)\\
			&=\frac{1}{2}\varepsilon^{-2n-1}\int_M(\tensor{(g^\bullet)}{_I_\infty_,^I}
			-\tensor{(g^\bullet)}{_I^I_,_\infty})
			dV_{k_\varepsilon}+O(1).
		\end{split}
	\end{equation}
	We compare the $\log(1/\varepsilon)$-terms of the both sides of \eqref{eq:VariationOfVolume}.
	That of the left-hand side is obviously $(n+2)L^\bullet$.
	As for the right-hand side, we use
	\begin{equation}
		\label{eq:volume_variation_integrand}
			\tensor{(g^\bullet)}{_I_\infty_,^I}-\tensor{(g^\bullet)}{_I^I_,_\infty}
			=-\frac{1}{2}\tensor{(k_\rho^{-1})}{^i^j}\tensor{(k_\rho^{-1})}{^k^l}
			\tensor{{(k'_\rho)}}{_j_l}\tensor{{(k_\rho^\bullet)}}{_i_k}
			-(\tensor{(k_\rho^{-1})}{^i^j}\tensor{{(k_\rho^\bullet)}}{_i_j})',
	\end{equation}
	where the primes denotes differentiations in $\rho$.
	Since $(\det k_\rho)^\bullet=(\det k_\rho)\tensor{(k_\rho^{-1})}{^i^j}\tensor{{(k_\rho^\bullet)}}{_i_j}$,
	we conclude that
	$\tensor{(k_\rho^{-1})}{^i^j}\tensor{{(k_\rho^\bullet)}}{_i_j}$ contains no
	$(\rho^{2n+2}\log\rho)$-term, which implies that the $\log(1/\varepsilon)$-term of the variation of the volume
	expansion may come only from the first term of the right-hand side of \eqref{eq:volume_variation_integrand}.
	The logarithmic term that appears in the expansion of $\tensor{{(k'_\rho)}}{_j_l}$ is
	$8\mathcal{O}\cdot\rho^{2n+1}\log\rho$, and hence, modulo logarithm-free terms,
	\begin{equation*}
		\begin{split}
			&-\frac{1}{2}\tensor{(k_\rho^{-1})}{^i^j}\tensor{(k_\rho^{-1})}{^k^l}
			\tensor{{(k'_\rho)}}{_j_l}\tensor{{(k_\rho^\bullet)}}{_i_k}\\
			&= -4\rho^{2n+1}\log\rho\cdot
			\tensor{(k_\rho^{-1})}{^{\conj{\alpha}}^\beta}\tensor{(k_\rho^{-1})}{^{\conj{\gamma}}^\sigma}
			\tensor{\mathcal{O}}{_\beta_\sigma}\tensor{{(k_\rho^\bullet)}}{_{\conj{\alpha}}_{\conj{\gamma}}}
			+(\text{the complex conjugate})+O(\rho^{2n+3}\log\rho)\\
			&= 8\rho^{2n+1}\log\rho\cdot
			(\tensor{\mathcal{O}}{^\alpha^\beta}\tensor{\psi}{_\alpha_\beta}
			+\tensor{\mathcal{O}}{^{\conj{\alpha}}^{\conj{\beta}}}
			\tensor{\psi}{_{\conj{\alpha}}_{\conj{\beta}}})+O(\rho^{2n+3}\log\rho).
		\end{split}
	\end{equation*}
	Therefore the $\log(1/\varepsilon)$-coefficient of the right-hand side of \eqref{eq:VariationOfVolume} is
	$-4$ times the integral of
	$(\tensor{\mathcal{O}}{^\alpha^\beta}\tensor{\psi}{_\alpha_\beta}
	+\tensor{\mathcal{O}}{^{\conj\alpha}^{\conj\beta}}\tensor{\psi}{_{\conj\alpha}_{\conj\beta}})dV_{\theta^2+h}$.
	Thus we conclude that
	\begin{equation*}
		L^\bullet=-\frac{4}{n+2}\int_M
		(\tensor{\mathcal{O}}{^\alpha^\beta}\tensor{\psi}{_\alpha_\beta}
		+\tensor{\mathcal{O}}{^{\conj\alpha}^{\conj\beta}}\tensor{\psi}{_{\conj\alpha}_{\conj\beta}})
		dV_{\theta^2+h}.
	\end{equation*}
	By combining this with Proposition \ref{prop:volume_log_term_and_Q}, we obtain \eqref{eq:Variation_TotalQ}.
\end{proof}

\section{Linearized obstruction operator}

\subsection{Asymptotic K\"ahlerity}

Although ACH metrics are not K\"ahler in general,
there is an asymptotic K\"ahler-like phenomenon about them, which introduces some insight into our
computation in the next subsection.

Let $g=4\rho^{-2}d\rho^2+k_\rho$ be a $C^\infty$-smooth ACH metric normalized with respect to $\theta$.
Let $\set{Z_\alpha}$ be a local frame of $T^{1,0}M$, and take the associated
local frame $\set{\bm{Z}_I}=\set{\bm{Z}_\infty,\bm{Z}_0,\bm{Z}_\alpha,\bm{Z}_{\conj{\alpha}}}$ of $\thetatangent$.
We define the complex subbundle $(\thetatangent|_M)^{1,0}$ of $\mathbb{C}\thetatangent|_M$
to be the one spanned by $\bm{Z}_\alpha|_M$ and $\bm{Z}_\tau|_M$, where
\begin{equation*}
	\bm{Z}_\tau:=\frac{1}{2}\bm{Z}_\infty+i\bm{Z}_0\qquad\text{and}\qquad
	\bm{Z}_{\conj{\tau}}:=\conj{\bm{Z}_\tau}=\frac{1}{2}\bm{Z}_\infty-i\bm{Z}_0.
\end{equation*}
Then $(\thetatangent|_M)^{1,0}$ is actually independent of $\theta$ with which we normalize $g$.
The complex structure endomorphism of $\thetatangent|_M$ with $i$-eigenbundle $(\thetatangent|_M)^{1,0}$
is denoted by $J|_M$.

In the sequel, $\set{\bm{Z}_P}=\set{\bm{Z}_\tau,\bm{Z}_\alpha,\bm{Z}_{\conj{\tau}},\bm{Z}_{\conj{\alpha}}}$ is
mainly used instead of $\set{\bm{Z}_I}=\set{\bm{Z}_\infty,\bm{Z}_0,\bm{Z}_\alpha,\bm{Z}_{\conj{\alpha}}}$.
The next lemma shows the efficiency of this approach.
From now on, the indices $P$, $Q$, $R$, $\dots$ run $\set{\tau, 1, \dots, n, \conj{\tau},\conj{1},\dots,\conj{n}}$,
while $A$, $B$, $C$, $\dots$ run $\set{\tau, 1, \dots, n}$.
The barred indices $\conj{A}$, $\conj{B}$, $\conj{C}$, $\dots$ of course run
$\set{\conj{\tau},\conj{1},\dots,\conj{n}}$.
We also make an agreement that $\tensor{h}{_\alpha_{\conj{\beta}}}$ denotes $h(Z_\alpha,Z_{\conj{\beta}})$;
it is not $h(\bm{Z}_\alpha,\bm{Z}_{\conj{\beta}})$ (which does not even make sense).
Therefore,
\begin{equation*}
	\tensor{g}{_\tau_{\conj{\tau}}}=2,\qquad\tensor{g}{_\tau_{\conj{\alpha}}}=0,\qquad
	\tensor{g}{_\alpha_{\conj{\beta}}}=\tensor{h}{_\alpha_{\conj{\beta}}}\qquad\text{on $M$},
\end{equation*}
and $\tensor{g}{_A_B}$ vanishes on the boundary.

\begin{lem}
	\label{prop:asymptotic_kahlerity}
	Under the situation above,
	let $J\in C^\infty(M,\End(\thetatangent))$ be any $(1,1)$-$\Theta$-tensor on $X$ that extends $J|_M$.
	Then the $(2,1)$-$\Theta$-tensor $\nabla J$ vanishes on $M$.
\end{lem}

\begin{proof}
	Let $\tensor{\omega}{_Q^R}=\tensor{\Gamma}{^R_P_Q}\bm{\theta}^P$ be the connection 1-$\Theta$-forms of
	$\nabla$ with respect to the local frame $\set{\bm{Z}_P}$,
	where $\set{\bm{\theta}^P}$ denotes the dual coframe.
	Then $\tensor{\Gamma}{_R_P_Q}:=\tensor{g}{_R_S}\tensor{\Gamma}{^S_P_Q}$ are given by
	\begin{equation}
		\label{eq:Levi_Civita}
		\tensor{\Gamma}{_R_P_Q}
		=\frac{1}{2}(\bm{Z}_P\tensor{g}{_Q_R}+\bm{Z}_Q\tensor{g}{_P_R}-\bm{Z}_R\tensor{g}{_P_Q}
		+[\bm{Z}_P,\bm{Z}_Q]_R-[\bm{Z}_P,\bm{Z}_R]_Q-[\bm{Z}_Q,\bm{Z}_R]_P),
	\end{equation}
	where $[\bm{Z}_P,\bm{Z}_Q]^R=\bm{\theta}^R([\bm{Z}_P,\bm{Z}_Q])$ and
	$[\bm{Z}_P,\bm{Z}_Q]_R=\tensor{g}{_R_S}[\bm{Z}_P,\bm{Z}_Q]^S$.
	Since $\bm{Z}_PF=O(\rho)$ for any function $F\in C^\infty(X)$, we have
	\begin{equation*}
		\tensor{\Gamma}{_R_P_Q}
		=\frac{1}{2}([\bm{Z}_P,\bm{Z}_Q]_R-[\bm{Z}_P,\bm{Z}_R]_Q-[\bm{Z}_Q,\bm{Z}_R]_P)
		\qquad\text{on $M$}.
	\end{equation*}
	By explicit computation we obtain
	\begin{alignat*}{3}
		[\bm{Z}_\tau,\bm{Z}_\tau]&=0,&\qquad
		[\bm{Z}_\tau,\bm{Z}_\alpha]&=\frac{1}{2}\bm{Z}_\alpha,&\qquad
		[\bm{Z}_\alpha,\bm{Z}_\beta]&=0,\\
		[\bm{Z}_\tau,\bm{Z}_{\conj{\tau}}]&=-(\bm{Z}_\tau-\bm{Z}_{\conj{\tau}}),&\qquad
		[\bm{Z}_\tau,\bm{Z}_{\conj{\alpha}}]&=\frac{1}{2}\bm{Z}_{\conj{\alpha}},&\qquad
		[\bm{Z}_\alpha,\bm{Z}_{\conj{\beta}}]
		&=-\frac{1}{2}\tensor{h}{_\alpha_{\conj{\beta}}}(\bm{Z}_\tau-\bm{Z}_{\conj{\tau}})
		\qquad\text{on $M$},
	\end{alignat*}
	and we can conclude from this that the Christoffel symbols of the type $\tensor{\Gamma}{_C_P_B}$
	all vanish on $M$.
	Hence $\tensor{\omega}{_B^{\conj{C}}}=0$ on $M$, which is equivalent to $\nabla J$ being zero on $M$.
\end{proof}

We record here the boundary values of the remaining Christoffel symbols: on $M$,
\begin{subequations}
	\label{eq:Christoffel}
\begin{alignat}{4}
	\tensor{\Gamma}{^\tau_\tau_\tau}&=-1,&\qquad
	\tensor{\Gamma}{^\tau_{\conj{\tau}}_\tau}&=1,&\qquad
	\tensor{\Gamma}{^\tau_\alpha_\tau}&=0,&\qquad
	\tensor{\Gamma}{^\tau_{\conj{\alpha}}_\tau}&=0,\\
	\tensor{\Gamma}{^\gamma_\tau_\tau}&=0,&\qquad
	\tensor{\Gamma}{^\gamma_{\conj{\tau}}_\tau}&=0,&\qquad
	\tensor{\Gamma}{^\gamma_\alpha_\tau}&=-\tensor{\delta}{_\alpha^\gamma},&\qquad
	\tensor{\Gamma}{^\gamma_{\conj{\alpha}}_\tau}&=0,\\
	\tensor{\Gamma}{^\tau_\tau_\beta}&=0,&\qquad
	\tensor{\Gamma}{^\tau_{\conj{\tau}}_\beta}&=0,&\qquad
	\tensor{\Gamma}{^\tau_\alpha_\beta}&=0,&\qquad
	\tensor{\Gamma}{^\tau_{\conj{\alpha}}_\beta}&=\tfrac{1}{2}\tensor{h}{_\beta_{\conj{\alpha}}},\\
	\tensor{\Gamma}{^\gamma_\tau_\beta}&=-\tfrac{1}{2}\tensor{\delta}{_\beta^\gamma},&\qquad
	\tensor{\Gamma}{^\gamma_{\conj{\tau}}_\beta}&=\tfrac{1}{2}\tensor{\delta}{_\beta^\gamma},&\qquad
	\tensor{\Gamma}{^\gamma_\alpha_\beta}&=0,&\qquad
	\tensor{\Gamma}{^\gamma_{\conj{\alpha}}_\beta}&=0.
\end{alignat}
\end{subequations}
We give an observation on the curvature tensor of $g$, which can be seen as a good reason for
the name ``asymptotically complex hyperbolic metric.''

\begin{prop}
	\label{prop:asymptotic_complex_hyperbolicity}
	Let $R$ be the Riemann curvature tensor, regarded as a $\Theta$-tensor,
	of a normalized $C^\infty$-smooth ACH metric $g$.
	Among its components with respect to $\set{\bm{Z}_P}$,
	the only ones that are non-vanishing on $M$ are those of K\"ahler-curvature
	type, i.e., $\tensor{R}{_A_{\conj{B}}_C_{\conj{D}}}$, $\tensor{R}{_{\conj{B}}_A_C_{\conj{D}}}$,
	$\tensor{R}{_A_{\conj{B}}_{\conj{D}}_C}$, and $\tensor{R}{_{\conj{B}}_A_{\conj{D}}_C}$.
	Moreover, the boundary value of $\tensor{R}{_A_{\conj{B}}_C_{\conj{D}}}$ is given as follows:
	\begin{equation}
		\label{eq:curvature_of_ACH}
		\left.(\tensor{R}{_A_{\conj{B}}_C_{\conj{D}}})\right|_M
		=-\frac{1}{2}\left.(\tensor{g}{_A_{\conj{B}}}\tensor{g}{_C_{\conj{D}}}
		+\tensor{g}{_A_{\conj{D}}}\tensor{g}{_C_{\conj{B}}})\right|_M.
	\end{equation}
\end{prop}

\begin{proof}
	Lemma \ref{prop:asymptotic_kahlerity} implies that the curvature $R$,
	seen as an $\End(\thetatangent)$-valued 2-$\Theta$-form, respects $J$ on $M$.
	Hence $\tensor{R}{_P_Q_C_D}$ and $\tensor{R}{_P_Q_{\conj{C}}_{\conj{D}}}$ vanish on $M$,
	and by the curvature symmetry, only the K\"ahler-curvature type components survive on $M$.
	The proof of \eqref{eq:curvature_of_ACH} is given by a direct computation using \eqref{eq:Christoffel}.
\end{proof}

\subsection{Dirichlet-to-Neumann-type characterization of $\mathcal{O}^\bullet$}

Let $\sigma$ be an arbitrary symmetric 2-$\Theta$-tensor on $X$ equipped with a $C^\infty$-smooth ACH metric $g$.
Taking the normalization of $g$ with respect to a contact form $\theta$, we obtain a 2-tensor
in $\tensor{\mathcal{E}}{_(_\alpha_\beta_)}$ whose components are $\sigma(\bm{Z}_\alpha,\bm{Z}_\beta)|_M$.
If we consider another contact form $\Hat{\theta}=e^{\Upsilon}\theta$ and
the local frame
$\set{\Hat{\bm{Z}}_P}
=\set{\Hat{\bm{Z}}_\tau,\Hat{\bm{Z}}_\alpha,\Hat{\bm{Z}}_{\conj{\tau}},\Hat{\bm{Z}}_{\conj{\alpha}}}$
associated to the normalization with respect to $\Hat{\theta}$,
then since the model boundary defining functions are related as $\log(\Hat{\rho}/\rho)|_M=\Upsilon/2$,
it holds that
\begin{equation*}
	\sigma(\Hat{\bm{Z}}_\alpha,\Hat{\bm{Z}}_\beta)|_M=e^\Upsilon\cdot\sigma(\bm{Z}_\alpha,\bm{Z}_\beta)|_M.
\end{equation*}
Therefore, $\sigma(\bm{Z}_\alpha,\bm{Z}_\beta)|_M$ is more naturally interpreted as defining a weighted tensor
in $\tensor{\mathcal{E}}{_(_\alpha_\beta_)}(1,1)$, which we write $\tensor{(\sigma|_M)}{_\alpha_\beta}$.
Moreover, if $\sigma=O(\rho^{2k})$, then since
\begin{equation*}
	\Hat{\rho}^{-2k}\sigma(\Hat{\bm{Z}}_\alpha,\Hat{\bm{Z}}_\beta)|_M
	=e^{(1-k)\Upsilon}\cdot\rho^{-2k}\sigma(\bm{Z}_\alpha,\bm{Z}_\beta)|_M,
\end{equation*}
$\rho^{-2k}\sigma(\bm{Z}_\alpha,\bm{Z}_\beta)|_M$ is considered to define a weighted tensor that belongs to
$\tensor{\mathcal{E}}{_(_\alpha_\beta_)}(1-k,1-k)$, which we write $\tensor{(\rho^{-2k}\sigma|_M)}{_\alpha_\beta}$.

\begin{prop}
	\label{prop:lichnerowicz_equation}
	Let $g$ be a $C^\infty$-smooth ACH metric on a $(2n+2)$-dimensional $\Theta$-manifold $X$.
	Then there exists, for any
	$\tensor{\psi}{_\alpha_\beta}\in\tensor{\mathcal{E}}{_(_\alpha_\beta_)}(1,1)$,
	a real $C^\infty$-smooth symmetric 2-$\Theta$-tensor $\sigma$
	such that $\tensor{(\sigma|_M)}{_\alpha_\beta}=\tensor{\psi}{_\alpha_\beta}$ and
	\begin{equation}
		\label{eq:linearized_Einstein_equation}
		(\Delta_\mathrm{L}+n+2)\sigma=O(\rho^{2n+2}).
	\end{equation}
	The $\Theta$-tensor $\sigma$ is uniquely determined modulo $O(\rho^{2n+2})$ by these conditions,
	and it is automatically approximately trace-free and divergence-free:
	\begin{equation}
		\label{eq:transverse_tracelessness}
		\tr\sigma=O(\rho^{2n+2}),\qquad
		\delta\sigma=O(\rho^{2n+2}).
	\end{equation}
	Furthermore, suppose that $g$ satisfies the approximate Einstein condition \eqref{eq:approx_Einstein}.
	In this case, for any $\sigma$ satisfying \eqref{eq:linearized_Einstein_equation},
	if we write $(\Delta_\mathrm{L}+n+2)\sigma=\tilde{\sigma}$, then
	$\tensor{(\rho^{-2n-2}\tilde{\sigma}|_M)}{_\alpha_\beta}=-\tensor*{\mathcal{O}}{^\bullet_\alpha_\beta}$,
	where
	$\tensor*{\mathcal{O}}{^\bullet_\alpha_\beta}=\tensor{(\mathcal{O}^\bullet\psi)}{_\alpha_\beta}$
	is the variation of the obstruction tensor.
\end{prop}

Note that, if $\sigma$ satisfies \eqref{eq:linearized_Einstein_equation} and \eqref{eq:transverse_tracelessness},
then
\begin{equation}
	\label{eq:linearized_Einstein_equation_restated}
	\Ric^\bullet\sigma+\frac{1}{2}(n+2)\sigma
	=\frac{1}{2}((\Delta_\mathrm{L}+n+2)\sigma-\mathcal{K}\mathcal{B}\sigma)=O(\rho^{2n+2}),
\end{equation}
where $\Ric^\bullet$ denotes the linearized Ricci operator and
$\mathcal{K}$, $\mathcal{B}$ are the Killing and Bianchi operators:
\begin{equation*}
	\tensor{(\mathcal{K}\eta)}{_P_Q}:=2\tensor{\nabla}{_(_P}\tensor{\eta}{_Q_)}\qquad\text{and}\qquad
	\tensor{(\mathcal{B}\sigma)}{_P}
	:=\tensor{(\delta\sigma)}{_P}+\frac{1}{2}\tensor{\nabla}{_P}(\tr\sigma).
\end{equation*}

To prove Proposition \ref{prop:lichnerowicz_equation}, we need the following lemma.

\begin{lem}
	\label{lem:Laplacian_on_tensors}
	Let $g$ be a $C^\infty$-smooth normalized ACH metric and $j\ge 0$ an integer.
	If $\mu$ is an $O(\rho^j)$ 1-$\Theta$-form,
	then the components of $(\Delta_\mathrm{H}+n+2)\mu$ with respect to the local frame
	$\set{\bm{Z}_P}$, where $\Delta_\mathrm{H}$ is the Hodge Laplacian, are
	\begin{subequations}
		\label{eq:Hodge_Laplacian_on_1-forms}
	\begin{align}
		(\Delta_\mathrm{H}+n+2)\tensor{\mu}{_\tau}
		&=-\tfrac{1}{4}(j+2)(j-2n-4)\tensor{\mu}{_\tau}+O(\rho^{j+1}),\\
		(\Delta_\mathrm{H}+n+2)\tensor{\mu}{_\alpha}
		&=-\tfrac{1}{4}\left(j^2-(2n+2)j-2n-7\right)\tensor{\mu}{_\alpha}+O(\rho^{j+1}).
	\end{align}
	\end{subequations}
	If $\sigma$ is an $O(\rho^j)$ symmetric 2-$\Theta$-form, then the components of $\delta\sigma$ are
	\begin{subequations}
		\label{eq:divergence_of_2-form}
	\begin{align}
		\tensor{(\delta\sigma)}{_\tau}
		&=-\tfrac{1}{4}(j-2n-4)\tensor{\sigma}{_\tau_\tau}-\tfrac{1}{4}(j-2n)\tensor{\sigma}{_\tau_{\conj{\tau}}}
		-\tr_h\tensor{\sigma}{_\alpha_{\conj{\beta}}}+O(\rho^{j+1}),\\
		\tensor{(\delta\sigma)}{_\alpha}
		&=-\tfrac{1}{4}(j-2n-5)\tensor{\sigma}{_\tau_\alpha}
		-\tfrac{1}{4}(j-2n-1)\tensor{\sigma}{_{\conj{\tau}}_\alpha}+O(\rho^{j+1}).
	\end{align}
	\end{subequations}
	Under the same assumption, the components of $(\Delta_\mathrm{L}+n+2)\sigma$ are
	\begin{subequations}
		\label{eq:Lichnerowicz_Laplacian_of_2-forms}
	\begin{align}
		(\Delta_\mathrm{L}+n+2)\tensor{\sigma}{_\tau_\tau}
		&=-\tfrac{1}{4}(j+2)(j-2n-4)\tensor{\sigma}{_\tau_\tau}+O(\rho^{j+1}),\\
		(\Delta_\mathrm{L}+n+2)\tensor{\sigma}{_\tau_\alpha}
		&=-\tfrac{1}{4}\left(j^2-(2n+2)j-2n-7\right)\tensor{\sigma}{_\tau_\alpha}+O(\rho^{j+1}),\\
		(\Delta_\mathrm{L}+n+2)\tensor{\sigma}{_\alpha_\beta}
		&=-\tfrac{1}{4}j(j-2n-2)\tensor{\sigma}{_\alpha_\beta}+O(\rho^{j+1}),\\
		(\Delta_\mathrm{L}+n+2)\tensor{\sigma}{_\tau_{\conj{\tau}}}
		&=-\tfrac{1}{4}(j+2)(j-2n-4)\tensor{\sigma}{_\tau_{\conj{\tau}}}+O(\rho^{j+1}),\\
		(\Delta_\mathrm{L}+n+2)\tensor{\sigma}{_\tau_{\conj{\alpha}}}
		&=-\tfrac{1}{4}\left(j^2-(2n+2)j-2n-7\right)\tensor{\sigma}{_\tau_{\conj{\alpha}}}+O(\rho^{j+1}),\\
		\tr_h(\Delta_\mathrm{L}+n+2)\tensor{\sigma}{_\alpha_{\conj{\beta}}}
		&=-\tfrac{1}{4}(j+2)(j-2n-4)\tr_h\tensor{\sigma}{_\alpha_{\conj{\beta}}}+O(\rho^{j+1}),\\
		\tf_h(\Delta_\mathrm{L}+n+2)\tensor{\sigma}{_\alpha_{\conj{\beta}}}
		&=-\tfrac{1}{4}\left(j^2-(2n+2)j-8\right)\tf_h\tensor{\sigma}{_\alpha_{\conj{\beta}}}+O(\rho^{j+1}).
	\end{align}
	\end{subequations}
	Here, $\tr_h$ and $\tf_h$ denote the trace and the trace-free part with respect to
	$\tensor{h}{_\alpha_{\conj{\beta}}}$.
\end{lem}

\begin{proof}
	We first explain that it suffices to assume $j=0$. Let $\nu=\rho^j\tilde{\nu}$ be any $O(\rho^j)$
	$\Theta$-tensor. Then $\nabla\nu=\rho^j\nabla\tilde{\nu}+j\rho^jd(\log\rho)\otimes\tilde{\nu}$, so we have
	\begin{equation*}
		\delta\nu
		=\rho^j\delta\tilde{\nu}-\tfrac{1}{4}j\rho^j(\rho\partial_\rho)\contraction\tilde{\nu}
		=\rho^j\delta\tilde{\nu}-\tfrac{1}{4}j\rho^j(\bm{Z}_\tau+\bm{Z}_{\conj{\tau}})\contraction\tilde{\nu},
	\end{equation*}
	which implies that \eqref{eq:divergence_of_2-form} follows from the $j=0$ case. Next we compute
	\begin{equation*}
		\nabla^2\nu=\rho^j\nabla^2\tilde{\nu}+2j\rho^jd(\log\rho)\otimes\nabla\tilde{\nu}
		+j^2\rho^jd(\log\rho)\otimes d(\log\rho)\otimes\tilde{\nu}
		+j\rho^j\nabla^2(\log\rho)\otimes\tilde{\nu},
	\end{equation*}
	and hence
	\begin{equation*}
		\nabla^*\nabla\nu
		=\rho^j\nabla^*\nabla\tilde{\nu}
		-\tfrac{1}{2}j\rho^j\nabla_{\rho\partial_\rho}\tilde{\nu}
		-\tfrac{1}{4}j^2\rho^j\tilde{\nu}
		+j\rho^j\Delta(\log\rho)\cdot\tilde{\nu}.
	\end{equation*}
	Proposition \ref{prop:ACH_Laplacian} shows that $\Delta(\log\rho)=(n+1)/2+O(\rho)$, while
	\eqref{eq:Christoffel} implies that $\nabla_{\rho\partial_\rho}\tilde{\nu}=O(\rho)$.
	Since the difference between $\Delta_\mathrm{H}$ and $\nabla^*\nabla$ is just a linear action of
	the curvature tensor, we obtain
	$\Delta_\mathrm{H}\nu=\rho^j\Delta_\mathrm{H}\tilde{\nu}-\tfrac{1}{4}j(j-2n-2)\rho^j\tilde{\nu}+O(\rho^{j+1})$.

	Now we assume $j=0$, and in the following computation we omit the $O(\rho)$ terms,
	which is symbolically indicated by $\equiv$.
	By \eqref{eq:Christoffel} we obtain
	\begin{alignat*}{4}
		\nabla_\tau\tensor{\mu}{_\tau}&\equiv\tensor{\mu}{_\tau},&\qquad
		\nabla_{\conj{\tau}}\tensor{\mu}{_\tau}&\equiv -\tensor{\mu}{_\tau},&\qquad
		\nabla_\beta\tensor{\mu}{_\tau}&\equiv\tensor{\mu}{_\beta},&\qquad
		\nabla_{\conj{\beta}}\tensor{\mu}{_\tau}&\equiv 0,\\
		\nabla_\tau\tensor{\mu}{_\alpha}&\equiv\tfrac{1}{2}\tensor{\mu}{_\alpha},&\qquad
		\nabla_{\conj{\tau}}\tensor{\mu}{_\alpha}&\equiv -\tfrac{1}{2}\tensor{\mu}{_\alpha},&\qquad
		\nabla_\beta\tensor{\mu}{_\alpha}&\equiv 0,&\qquad
		\nabla_{\conj{\beta}}\tensor{\mu}{_\alpha}
		&\equiv -\tfrac{1}{2}\tensor{h}{_\alpha_{\conj{\beta}}}\tensor{\mu}{_\tau},
	\end{alignat*}
	and hence
	\begin{equation*}
		\nabla_\tau\nabla_{\conj{\tau}}\tensor{\mu}{_\tau} \equiv 0,\qquad
		\nabla_\beta\nabla_{\conj{\gamma}}\tensor{\mu}{_\tau} \equiv 0,\qquad
		\nabla_\tau\nabla_{\conj{\tau}}\tensor{\mu}{_\alpha} \equiv \tfrac{1}{4}\tensor{\mu}{_\alpha},\qquad
		\nabla_\beta\nabla_{\conj{\gamma}}\tensor{\mu}{_\alpha}
		\equiv \tfrac{1}{4}\tensor{h}{_\beta_{\conj{\gamma}}}\tensor{\mu}{_\alpha}.
	\end{equation*}
	Thus we have
	$\tensor{\nabla}{_B}\tensor{\nabla}{^B}\tensor{\mu}{_\tau}\equiv 0$ and
	$\tensor{\nabla}{_B}\tensor{\nabla}{^B}\tensor{\mu}{_\alpha}
	\equiv\tfrac{1}{8}(2n+1)\tensor{\mu}{_\alpha}$.
	Hence we obtain \eqref{eq:Hodge_Laplacian_on_1-forms} in the $j=0$ case, because
	by Proposition \ref{prop:asymptotic_complex_hyperbolicity},
	\begin{equation*}
		\begin{split}
			(\Delta_\mathrm{H}+n+2)\tensor{\mu}{_A}
			&\equiv \nabla^*\nabla\tensor{\mu}{_A}+\tensor{\Ric}{_A^B}\tensor{\mu}{_B}
				+(n+2)\tensor{\mu}{_A}\\
			&\equiv -2\tensor{\nabla}{_B}\tensor{\nabla}{^B}\tensor{\mu}{_A}
				-\tensor{R}{_B^B_A^C}\tensor{\mu}{_C}+\tensor{\Ric}{_A^B}\tensor{\mu}{_B}
				+(n+2)\tensor{\mu}{_A}\\
			&\equiv -2\tensor{\nabla}{_B}\tensor{\nabla}{^B}\tensor{\mu}{_A}+(n+2)\tensor{\mu}{_A}.
		\end{split}
	\end{equation*}
	Similarly,
	\allowdisplaybreaks
	\begin{alignat*}{4}
		\nabla_\tau\tensor{\sigma}{_\tau_\tau}&\equiv 2\tensor{\sigma}{_\tau_\tau},&\quad
		\nabla_{\conj{\tau}}\tensor{\sigma}{_\tau_\tau}&\equiv -2\tensor{\sigma}{_\tau_\tau},&\quad
		\nabla_\gamma\tensor{\sigma}{_\tau_\tau}&\equiv 2\tensor{\sigma}{_\tau_\gamma},&\quad
		\nabla_{\conj{\gamma}}\tensor{\sigma}{_\tau_\tau}&\equiv 0,\\
		\nabla_\tau\tensor{\sigma}{_\tau_\alpha}&\equiv \tfrac{3}{2}\tensor{\sigma}{_\tau_\alpha},&\quad
		\nabla_{\conj{\tau}}\tensor{\sigma}{_\tau_\alpha}
		&\equiv -\tfrac{3}{2}\tensor{\sigma}{_\tau_\alpha},&\quad
		\nabla_\gamma\tensor{\sigma}{_\tau_\alpha}&\equiv \tensor{\sigma}{_\alpha_\gamma},&\quad
		\nabla_{\conj{\gamma}}\tensor{\sigma}{_\tau_\alpha}
		&\equiv -\tfrac{1}{2}\tensor{h}{_\alpha_{\conj{\gamma}}}\tensor{\sigma}{_\tau_\tau},\\
		\nabla_\tau\tensor{\sigma}{_\alpha_\beta}&\equiv\tensor{\sigma}{_\alpha_\beta},&\quad
		\nabla_{\conj{\tau}}\tensor{\sigma}{_\alpha_\beta}&\equiv -\tensor{\sigma}{_\alpha_\beta},&\quad
		\nabla_\gamma\tensor{\sigma}{_\alpha_\beta}&\equiv 0,&\quad
		\nabla_{\conj{\gamma}}\tensor{\sigma}{_\alpha_\beta}
		&\equiv -\tensor{h}{_(_\alpha_|_{\conj{\gamma}}}\tensor{\sigma}{_\tau_|_\beta_)},\\
		\nabla_\tau\tensor{\sigma}{_\tau_{\conj{\tau}}}&\equiv 0,&\quad
		\nabla_{\conj{\tau}}\tensor{\sigma}{_\tau_{\conj{\tau}}}&\equiv 0,&\quad
		\nabla_\gamma\tensor{\sigma}{_\tau_{\conj{\tau}}}&\equiv \tensor{\sigma}{_\gamma_{\conj{\tau}}},&\quad
		\nabla_{\conj{\gamma}}\tensor{\sigma}{_\tau_{\conj{\tau}}}
		&\equiv \tensor{\sigma}{_\tau_{\conj{\gamma}}},\\
		\nabla_\tau\tensor{\sigma}{_\tau_{\conj{\alpha}}}
		&\equiv \tfrac{1}{2}\tensor{\sigma}{_\tau_{\conj{\alpha}}},&\quad
		\nabla_{\conj{\tau}}\tensor{\sigma}{_\tau_{\conj{\alpha}}}
		&\equiv -\tfrac{1}{2}\tensor{\sigma}{_\tau_{\conj{\alpha}}},&\quad
		\nabla_\gamma\tensor{\sigma}{_\tau_{\conj{\alpha}}}
		&\equiv \tensor{\sigma}{_\gamma_{\conj{\alpha}}}
			-\tfrac{1}{2}\tensor{h}{_\gamma_{\conj{\alpha}}}\tensor{\sigma}{_\tau_{\conj{\tau}}},&\quad
		\nabla_{\conj{\gamma}}\tensor{\sigma}{_\tau_{\conj{\alpha}}}&\equiv 0,\\
		\nabla_\tau\tensor{\sigma}{_\alpha_{\conj{\beta}}}&\equiv 0,&\quad
		\nabla_{\conj{\tau}}\tensor{\sigma}{_\alpha_{\conj{\beta}}}&\equiv 0,&\quad
		\nabla_\gamma\tensor{\sigma}{_\alpha_{\conj{\beta}}}
		&\equiv -\tfrac{1}{2}\tensor{h}{_\gamma_{\conj{\beta}}}\tensor{\sigma}{_\alpha_{\conj{\tau}}},&\quad
		\nabla_{\conj{\gamma}}\tensor{\sigma}{_\alpha_{\conj{\beta}}}
		&\equiv -\tfrac{1}{2}\tensor{h}{_\alpha_{\conj{\gamma}}}\tensor{\sigma}{_\tau_{\conj{\beta}}},
	\end{alignat*}
	\allowdisplaybreaks[0]%
	and we obtain \eqref{eq:divergence_of_2-form} for $j=0$. Moreover,
	\allowdisplaybreaks
	\begin{alignat*}{2}
		\nabla_\tau\nabla_{\conj{\tau}}\tensor{\sigma}{_\tau_\tau}
		&\equiv -2\tensor{\sigma}{_\tau_\tau},&\qquad
		\nabla_\gamma\nabla_{\conj{\delta}}\tensor{\sigma}{_\tau_\tau}
		&\equiv 0,\\
		\nabla_\tau\nabla_{\conj{\tau}}\tensor{\sigma}{_\tau_\alpha}
		&\equiv -\tfrac{3}{4}\tensor{\sigma}{_\tau_\alpha},&\qquad
		\nabla_\gamma\nabla_{\conj{\delta}}\tensor{\sigma}{_\tau_\alpha}
		&\equiv \tfrac{1}{4}\tensor{h}{_\gamma_{\conj{\delta}}}\tensor{\sigma}{_\tau_\alpha}
			-\tfrac{1}{2}\tensor{h}{_\alpha_{\conj{\delta}}}\tensor{\sigma}{_\tau_\gamma},\\
		\nabla_\tau\nabla_{\conj{\tau}}\tensor{\sigma}{_\alpha_\beta}
		&\equiv 0,&\qquad
		\nabla_\gamma\nabla_{\conj{\delta}}\tensor{\sigma}{_\alpha_\beta}
		&\equiv \tfrac{1}{2}\tensor{h}{_\gamma_{\conj{\delta}}}\tensor{\sigma}{_\alpha_\beta},\\
		\nabla_\tau\nabla_{\conj{\tau}}\tensor{\sigma}{_\tau_{\conj{\tau}}}
		&\equiv 0,&\qquad
		\nabla_\gamma\nabla_{\conj{\delta}}\tensor{\sigma}{_\tau_{\conj{\tau}}}
		&\equiv \tensor{\sigma}{_\alpha_{\conj{\beta}}}
			-\tfrac{1}{2}\tensor{h}{_\alpha_{\conj{\beta}}}\tensor{\sigma}{_\tau_{\conj{\tau}}},\\
		\nabla_\tau\nabla_{\conj{\tau}}\tensor{\sigma}{_\tau_{\conj{\alpha}}}
		&\equiv \tfrac{1}{4}\tensor{\sigma}{_\tau_{\conj{\alpha}}},&\qquad
		\nabla_\gamma\nabla_{\conj{\delta}}\tensor{\sigma}{_\tau_{\conj{\alpha}}}
		&\equiv -\tfrac{1}{4}\tensor{h}{_\gamma_{\conj{\delta}}}\tensor{\sigma}{_\tau_{\conj{\alpha}}}
			-\tfrac{1}{2}\tensor{h}{_\gamma_{\conj{\alpha}}}\tensor{\sigma}{_\tau_{\conj{\delta}}},\\
		\nabla_\tau\nabla_{\conj{\tau}}\tensor{\sigma}{_\alpha_{\conj{\beta}}}
		&\equiv 0,&\qquad
		\nabla_\gamma\nabla_{\conj{\delta}}\tensor{\sigma}{_\alpha_{\conj{\beta}}}
		&\equiv -\tfrac{1}{2}\tensor{h}{_\gamma_{\conj{\beta}}}\tensor{\sigma}{_\alpha_{\conj{\delta}}}
			+\tfrac{1}{4}\tensor{h}{_\alpha_{\conj{\delta}}}\tensor{h}{_\gamma_{\conj{\beta}}}
			\tensor{\sigma}{_\tau_{\conj{\tau}}},
	\end{alignat*}
	\allowdisplaybreaks[0]%
	and therefore
	\begin{subequations}
		\label{eq:partial_Laplacian_on_2-tensors}
	\begin{alignat}{2}
		\nabla_C\nabla^C\tensor{\sigma}{_\tau_\tau}
		&\equiv -\tensor{\sigma}{_\tau_\tau},&\qquad
		\nabla_C\nabla^C\tensor{\sigma}{_\tau_\alpha}
		&\equiv \tfrac{1}{8}(2n-7)\tensor{\sigma}{_\tau_\alpha},\\
		\nabla_C\nabla^C\tensor{\sigma}{_\alpha_\beta}
		&\equiv \tfrac{1}{2}n\tensor{\sigma}{_\alpha_\beta},&\qquad
		\nabla_C\nabla^C\tensor{\sigma}{_\tau_{\conj{\tau}}}
		&\equiv -\tfrac{1}{2}n\tensor{\sigma}{_\tau_{\conj{\tau}}}+\tr_h\tensor{\sigma}{_\gamma_{\conj{\delta}}},\\
		\nabla_C\nabla^C\tensor{\sigma}{_\tau_{\conj{\alpha}}}
		&\equiv -\tfrac{1}{8}(2n+3)\tensor{\sigma}{_\tau_{\conj{\alpha}}},&\qquad
		\nabla_C\nabla^C\tensor{\sigma}{_\alpha_{\conj{\beta}}}
		&\equiv -\tfrac{1}{2}\tensor{\sigma}{_\alpha_{\conj{\beta}}}
			+\tfrac{1}{4}\tensor{h}{_\alpha_{\conj{\beta}}}\tensor{\sigma}{_\tau_{\conj{\tau}}}.
	\end{alignat}
	\end{subequations}
	Furthermore, by Proposition \ref{prop:asymptotic_complex_hyperbolicity},
	\begin{equation*}
		\begin{split}
			(\Delta_\mathrm{L}+n+2)\tensor{\sigma}{_A_B}
			&\equiv \nabla^*\nabla\tensor{\sigma}{_A_B}
			+2\tensor{\Ric}{_(_A^C}\tensor{\sigma}{_B_)_C}
			+2\tensor{R}{_A^C_B^D}\tensor{\sigma}{_C_D}
			+(n+2)\tensor{\sigma}{_A_B}\\
			&\equiv -2\nabla_C\nabla^C\tensor{\sigma}{_A_B}
			-\tensor{R}{_C^C_A^D}\tensor{\sigma}{_D_B}-\tensor{R}{_C^C_B^D}\tensor{\sigma}{_A_D}\\
			&\phantom{\equiv}\quad+2\tensor{\Ric}{_(_A^C}\tensor{\sigma}{_B_)_C}
			+2\tensor{R}{_A^C_B^D}\tensor{\sigma}{_C_D}+(n+2)\tensor{\sigma}{_A_B}\\
			&\equiv -2\nabla_C\nabla^C\tensor{\sigma}{_A_B}+n\tensor{\sigma}{_A_B}
		\end{split}
	\end{equation*}
	and
	\begin{equation*}
		\begin{split}
			(\Delta_\mathrm{L}+n+2)\tensor{\sigma}{_A_{\conj{B}}}
			&\equiv \nabla^*\nabla\tensor{\sigma}{_A_{\conj{B}}}
			+\tensor{\Ric}{_A^C}\tensor{\sigma}{_C_{\conj{B}}}
			+\tensor{\Ric}{_{\conj{B}}^{\conj{C}}}\tensor{\sigma}{_A_{\conj{C}}}
			+2\tensor{R}{_A^C_{\conj{B}}^{\conj{D}}}\tensor{\sigma}{_C_{\conj{D}}}
			+(n+2)\tensor{\sigma}{_A_{\conj{B}}}\\
			&\equiv -2\nabla_C\nabla^C\tensor{\sigma}{_A_{\conj{B}}}
			-\tensor{R}{_C^C_A^D}\tensor{\sigma}{_D_{\conj{B}}}
			+\tensor{R}{_C^C^{\conj{D}}_{\conj{B}}}\tensor{\sigma}{_A_{\conj{D}}}\\
			&\phantom{\equiv}\quad+\tensor{\Ric}{_A^C}\tensor{\sigma}{_C_{\conj{B}}}
			+\tensor{\Ric}{_{\conj{B}}^{\conj{C}}}\tensor{\sigma}{_A_{\conj{C}}}
			-2\tensor{R}{_A^C^{\conj{D}}_{\conj{B}}}\tensor{\sigma}{_C_{\conj{D}}}
			+(n+2)\tensor{\sigma}{_A_{\conj{B}}}\\
			&\equiv -2\nabla_C\nabla^C\tensor{\sigma}{_A_{\conj{B}}}
			+\tensor{\sigma}{_A_{\conj{B}}}+\tensor{g}{_A_{\conj{B}}}\tensor{\sigma}{_C^C}.
		\end{split}
	\end{equation*}
	These equalities and \eqref{eq:partial_Laplacian_on_2-tensors} imply the desired result.
\end{proof}

Before we go to the proof of Proposition \ref{prop:lichnerowicz_equation}, we recall one general formula
valid on any Riemannian manifold. Define the 3-($\Theta$-)tensor $D\Ric$ by
\begin{equation*}
	\tensor{(D\Ric)}{_P_Q_R}
	:=\nabla_P\tensor{\Ric}{_Q_R}-\nabla_Q\tensor{\Ric}{_P_R}-\nabla_R\tensor{\Ric}{_P_Q}
\end{equation*}
and its action on arbitrary symmetric 2-($\Theta$-)tensors by
$\tensor{((D\Ric)^\circ\sigma)}{_P}:=\tensor{(D\Ric)}{_P_Q_R}\tensor{\sigma}{^Q^R}$,
where the indices are raised by $g$. Then, by direct calculation, one can show that
\begin{equation}
	\label{eq:divergence_Laplacian_commutation}
	\delta\circ\Delta_\mathrm{L}=\Delta_\mathrm{H}\circ\delta+(D\Ric)^\circ.
\end{equation}

\begin{proof}[Proof of Proposition \ref{prop:lichnerowicz_equation}]
	We take the normalization with respect to a contact form $\theta$.
	We have imposed the condition that $\sigma$ should satisfy
	$\tensor{\sigma}{_\alpha_\beta}|_M=\tensor{\psi}{_\alpha_\beta}$, and since $\sigma$ is real,
	$\tensor{\sigma}{_{\conj{\alpha}}_{\conj{\beta}}}|_M$ must be of course
	$\tensor{\psi}{_{\conj{\alpha}}_{\conj{\beta}}}$.
	Equations \eqref{eq:Lichnerowicz_Laplacian_of_2-forms} show that the boundary values of the other components
	have to be zero in order that $(\Delta_\mathrm{L}+n+2)\sigma=O(\rho)$ is satisfied.
	Then by using \eqref{eq:Lichnerowicz_Laplacian_of_2-forms} recursively,
	it can be shown that $\sigma$ is uniquely determined modulo $O(\rho^{2n+2})$
	so that \eqref{eq:linearized_Einstein_equation} holds.

	We can prove that this approximate solution in fact satisfies
	$\tr\sigma=O(\rho^{2n+2})$ and $\delta\sigma=O(\rho^{2n+2})$ as follows.
	Firstly, taking the trace of \eqref{eq:linearized_Einstein_equation} we obtain
	\begin{equation*}
		(\Delta+n+2)(\tr\sigma)=O(\rho^{2n+2}).
	\end{equation*}
	On the other hand, \eqref{eq:ACHLaplacianWithSublaplacian} shows that if $F=O(\rho^j)$ is a function then
	$(\Delta+n+2)F=-\frac{1}{4}(j+2)(j-2n-4)F+O(\rho^{j+1})$. Hence we can conclude that
	$\tr\sigma$ is actually $O(\rho^{2n+2})$.
	Secondly, if we take the divergence of \eqref{eq:linearized_Einstein_equation}, then
	\eqref{eq:divergence_Laplacian_commutation} shows that the 1-$\Theta$-tensor $\delta\sigma$ satisfies
	\begin{equation*}
		(\Delta_\mathrm{H}+n+2)(\delta\sigma)
		=\delta(\Delta_\mathrm{L}+n+2)\sigma-(D\Ric)^\circ\sigma=O(\rho^{2n+2}).
	\end{equation*}
	Then, by recursively applying \eqref{eq:Hodge_Laplacian_on_1-forms},
	we can show that $\delta\sigma$ must be $O(\rho^{2n+2})$.

	Now let $T^{1,0}_t$ be a smooth 1-parameter family of partially integrable CR structures
	that is tangent to $\tensor{\psi}{_\alpha_\beta}$,
	and we take an associated $C^\infty$-smooth normalized ACH metric $g^t$
	satisfying \eqref{eq:approx_Einstein} for each $T^{1,0}_t$ so that $g^t$ is also smooth in $t$.
	Let $\sigma_\mathrm{normal}:=-\frac{1}{2}(d/dt)g^t|_{t=0}$.
	Then it of course solves \eqref{eq:linearized_Einstein_equation_restated}, satisfies
	$\tensor{(\sigma_\mathrm{normal}|_M)}{_\alpha_\beta}=\tensor{\psi}{_\alpha_\beta}$, and if we write
	\begin{equation*}
		2\left(\Ric^\bullet+\frac{n+2}{2}\right)\sigma_\mathrm{normal}=\tilde{\sigma}_\mathrm{normal},
	\end{equation*}
	then $\tensor{(\rho^{-2n-2}\tilde{\sigma}_\mathrm{normal}|_M)}{_\alpha_\beta}$ equals
	minus of $\tensor*{\mathcal{O}}{^\bullet_\alpha_\beta}$.
	We want to prove that we can take a 1-$\Theta$-form $\eta$ satisfying
	\begin{equation}
		\label{eq:modification_of_normalized_linearized_Einstein_equation}
		(\Delta_\mathrm{L}+n+2)\sigma=O(\rho^{2n+2}),
		\qquad\text{where}\quad \sigma=\sigma_\mathrm{normal}+\mathcal{K}\eta,
	\end{equation}
	and moreover, if $\eta^\sharp$ is the dual $\Theta$-vector field of $\eta$,
	\begin{equation}
		\label{eq:property_of_modification}
		\tensor{(\mathcal{L}_{\eta^\sharp}E)}{_\alpha_\beta}=O(\rho^{2n+3}),
	\end{equation}
	where $\mathcal{L}_{\eta^\sharp}E$ is the Lie derivative of $E=\Ric+\frac{1}{2}(n+2)g$.
	Suppose that such an $\eta$ exists. Then
	$\tensor{(\sigma|_M)}{_\alpha_\beta}=\tensor{\psi}{_\alpha_\beta}$ and
	\begin{equation*}
		(\Delta_\mathrm{L}+n+2)\sigma-2\left(\Ric^\bullet+\frac{n+2}{2}\right)\sigma_\mathrm{normal}
		=(2\Ric^\bullet{}+n+2)\mathcal{K}\eta+\mathcal{K}\mathcal{B}\sigma
		=2\mathcal{L}_{\eta^\sharp}E+\mathcal{K}\mathcal{B}\sigma.
	\end{equation*}
	As $\mathcal{B}\sigma=\delta\sigma+\frac{1}{2}d(\tr\sigma)=O(\rho^{2n+2})$ and hence
	$\tensor{(\mathcal{K}\mathcal{B}\sigma)}{_\alpha_\beta}=O(\rho^{2n+3})$ by \eqref{eq:Christoffel},
	we can conclude that $\tensor{(\rho^{-2n-2}\tilde{\sigma}|_M)}{_\alpha_\beta}
	=\tensor{(\rho^{-2n-2}\tilde{\sigma}_\mathrm{normal}|_M)}{_\alpha_\beta}
	=-\tensor*{\mathcal{O}}{^\bullet_\alpha_\beta}$.

	Suppose $\eta$ is taken so that $\mathcal{B}(\sigma_\mathrm{normal}+\mathcal{K}\eta)=O(\rho^{2n+2})$.
	Then since $\Ric^\bullet\mathcal{K}\eta=\mathcal{L}_{\eta^\sharp}\Ric=O(\rho^{2n+2})$,
	\eqref{eq:linearized_Einstein_equation_restated} implies that $\sigma=\sigma_\mathrm{normal}+\mathcal{K}\eta$
	satisfies \eqref{eq:modification_of_normalized_linearized_Einstein_equation}.
	To construct such an $\eta$, the equation to be solved is
	\begin{equation}
		\label{eq:gauge_equation}
		\mathcal{B}\mathcal{K}\eta=-\mathcal{B}\sigma_\mathrm{normal}+O(\rho^{2n+2}).
	\end{equation}
	The left-hand side is actually rewritten as
	\begin{equation*}
		\mathcal{B}\mathcal{K}\eta
		=\Delta_\mathrm{H}\xi-2\Ric^\circ\xi=(\Delta_\mathrm{H}+n+2)\eta+O(\rho^{2n+2}).
	\end{equation*}
	Therefore, by using \eqref{eq:Hodge_Laplacian_on_1-forms} recursively
	we can construct a solution of \eqref{eq:gauge_equation}.
	It remains to show that our solution $\eta$ satisfies \eqref{eq:property_of_modification}.
	Since the boundary value of $\sigma_\mathrm{normal}$ has no components other than
	$\tensor{(\sigma_\mathrm{normal})}{_\alpha_\beta}$ and its complex conjugate,
	\eqref{eq:divergence_of_2-form} shows that
	$\delta\sigma_\mathrm{normal}=O(\rho)$, and hence $\mathcal{B}\sigma_\mathrm{normal}=O(\rho)$.
	Therefore, $\eta=O(\rho)$. Let $\eta^\sharp=\rho\xi$, where $\xi$ is a $\Theta$-vector field. Then we obtain
	\begin{equation*}
		\begin{split}
			\tensor{(\mathcal{L}_{\eta^\sharp}E)}{_\alpha_\beta}
			&=(\mathcal{L}_{\eta^\sharp}E)(\bm{Z}_\alpha,\bm{Z}_\beta)
			=\eta^\sharp(E(\bm{Z}_\alpha,\bm{Z}_\beta))-E([\eta^\sharp,\bm{Z}_\alpha],\bm{Z}_\beta)
			-E(\bm{Z}_\alpha,[\eta^\sharp,\bm{Z}_\beta])\\
			&=\rho\xi(E(\bm{Z}_\alpha,\bm{Z}_\beta))-\rho E([\xi,\bm{Z}_\alpha],\bm{Z}_\beta)
			-\rho E(\bm{Z}_\alpha,[\xi,\bm{Z}_\beta])=O(\rho^{2n+3}).
		\end{split}
	\end{equation*}
	This finishes the proof.
\end{proof}

\subsection{Heisenberg principal part}
\label{subsec:Heisenberg_principal_part}

As the first application of Proposition \ref{prop:lichnerowicz_equation}, we prove the following theorem.

\begin{thm}
	\label{thm:linearized_obstruction}
	Let $(M,T^{1,0}M)$ be a partially integrable CR manifold of dimension $2n+1\ge 5$.
	Then, the linearized obstruction operator $\mathcal{O}^\bullet$ has the following expression
	for any choice of $\theta$:
	\begin{equation}
		\label{eq:linearized_obstruction}
		\mathcal{O}^\bullet
		=\mathcal{O}^\bullet_\mathrm{pr}+(\text{a differential operator with Heisenberg order $\le 2n+1$}),
	\end{equation}
	where
	\begin{equation}
		\begin{split}
			\label{eq:principal_part_of_linearized_obstruction}
			\tensor{(\mathcal{O}^\bullet_\mathrm{pr}\psi)}{_\alpha_\beta}
			&=\frac{(-1)^{n+1}}{(n!)^2}\left[
				\left(\prod_{k=0}^n (\Delta_b+i(n+2-2k)\nabla^\mathrm{TW}_T)\right)\tensor{\psi}{_\alpha_\beta}
				\right.\\
			&\phantom{\;=\;}
			+\frac{4(n+1)}{n+2}\left(\prod_{k=0}^{n-1}(\Delta_b+i(n+2-2k)\nabla^\mathrm{TW}_T)\right)
			\tensor*{\nabla}{^{\mathrm{TW}}_(_\alpha}\tensor{(\nabla^\mathrm{TW})}{^\gamma}
			\tensor{\psi}{_\beta_)_\gamma}\\
			&\phantom{\;=\;}
			+\left.\frac{4n}{n+2}\left(\prod_{k=0}^{n-2}(\Delta_b+i(n+2-2k)\nabla^\mathrm{TW}_T)\right)
			\tensor*{\nabla}{^{\mathrm{TW}}_\alpha}\tensor*{\nabla}{^{\mathrm{TW}}_\beta}
			\tensor{(\nabla^\mathrm{TW})}{^\gamma}\tensor{(\nabla^\mathrm{TW})}{^\delta}
			\tensor{\psi}{_\gamma_\delta}\right].
		\end{split}
	\end{equation}
	Here $\nabla^\mathrm{TW}$ denotes the Tanaka--Webster connection.
\end{thm}

We start with some preliminary considerations.
Since $\tensor{\mathcal{O}}{_\alpha_\beta}$ has a universal expression in terms of the Tanaka--Webster connection,
so does its linearization $\mathcal{O}^\bullet$. The expression of the Heisenberg principal part of
$\mathcal{O}^\bullet$ cannot involve $N$, $A$, or $R$, which is shown as follows. Suppose that
\begin{multline*}
	\contr((h^{-1})^{\otimes a}
	\otimes(\nabla^\mathrm{TW}_{1,0})^{\otimes b}
	\otimes(\nabla^\mathrm{TW}_{0,1})^{\otimes b'}
	\otimes(\nabla^\mathrm{TW}_T)^{\otimes b''}
	\otimes N^{\otimes c}
	\otimes A^{\otimes c'}
	\otimes R^{\otimes c''}
	\otimes(\text{$\psi$ or $\conj{\psi}$}))
\end{multline*}
is a term in the expression of
$\tensor*{\mathcal{O}}{^\bullet_\alpha_\beta}=\tensor{(\mathcal{O}^\bullet\psi)}{_\alpha_\beta}$.
Here $N$, $A$, $R$ denote $\tensor{N}{_\alpha_\beta^{\conj{\gamma}}}$, $\tensor{A}{_\alpha_\beta}$,
$\tensor{R}{_\alpha^\beta_\gamma_{\conj{\delta}}}$ or their complex conjugates so that they are invariant
under constant scalings of $\theta$.
Moreover, each $\nabla^\mathrm{TW}_{1,0}$, $\nabla^\mathrm{TW}_{0,1}$, and $\nabla^\mathrm{TW}_T$ is understood to
be applied to one of $N$, $A$, $R$, and $(\text{$\psi$ or $\conj{\psi}$})$.
Then, since $\tensor{\psi}{_\alpha_\beta}\in\tensor{\mathcal{E}}{_(_\alpha_\beta_)}(1,1)$ and
$\tensor*{\mathcal{O}}{^\bullet_\alpha_\beta}\in\tensor{\mathcal{E}}{_(_\alpha_\beta_)}(-n,-n)$,
$a+b''$ has to be $n+1$.
On the other hand, since the contraction is taken so that two holomorphic indices remain downstairs,
counting the number of indices gives $2a=b+b'+c+2c'+2c''$.
Hence $(b+b'+2b'')+(c+2c'+2c'')=2n+2$, which shows that this term has Heisenberg order $\le 2n+1$ unless
$c=c'=c''=0$.
Therefore, to determine the Heisenberg principal part of $\mathcal{O}^\bullet$,
it is sufficient to prove the following.

\begin{prop}
	\label{prop:linearized_obstruction_heisenberg}
	Let $M$ be the Heisenberg group of dimension $2n+1\ge 5$. Then, with respect to the standard contact form
	$\theta$, the right-hand side of \eqref{eq:principal_part_of_linearized_obstruction} gives the exact formula of
	$\tensor*{\mathcal{O}}{^\bullet_\alpha_\beta}$.
\end{prop}

Let the complex hyperbolic metric $g$ be normalized by the standard contact form $\theta$
(see Example \ref{ex:complex_hyperbolic}).
The natural complex structure on the complex hyperbolic space induces a section
$J\in C^\infty(X,\End(\thetatangent))$,
whose $i$-eigenbundle is actually spanned by $\set{\bm{Z}_\tau,\bm{Z}_\alpha}$. Let
\begin{equation}
	\label{eq:Tanaka_Webster_expression_of_linearized_obstruction}
	\tensor*{\mathcal{O}}{^\bullet_\alpha_\beta}
	=\tensor{(P')}{_\alpha_\beta^\gamma^\delta}\tensor{\psi}{_\gamma_\delta}
	+\tensor{(P'')}{_\alpha_\beta^{\conj{\gamma}}^{\conj{\delta}}}\tensor{\psi}{_{\conj{\gamma}}_{\conj{\delta}}}
\end{equation}
be the expression of $\tensor*{\mathcal{O}}{^\bullet_\alpha_\beta}$ as a sum of contractions of
Tanaka--Webster covariant derivatives of $\tensor{\psi}{_\alpha_\beta}$ and
$\tensor{\psi}{_{\conj{\alpha}}_{\conj{\beta}}}$.
By Proposition \ref{prop:lichnerowicz_equation}, this is the obstruction to the existence of
$C^\infty$-smooth solutions of \eqref{eq:linearized_Einstein_equation}.
Now, since $J$ is parallel in the current case, if $\sigma$ solves \eqref{eq:linearized_Einstein_equation}
then $\sigma(J\entrydot,\entrydot)$ is also a solution to \eqref{eq:linearized_Einstein_equation} with
boundary data $i\tensor{\psi}{_\alpha_\beta}$. Therefore,
\begin{equation*}
	i\tensor*{\mathcal{O}}{^\bullet_\alpha_\beta}
	=\tensor{(P')}{_\alpha_\beta^\gamma^\delta}(i\tensor{\psi}{_\gamma_\delta})
	+\tensor{(P'')}{_\alpha_\beta^{\conj{\gamma}}^{\conj{\delta}}}
	(-i\tensor{\psi}{_{\conj{\gamma}}_{\conj{\delta}}})
	=i(\tensor{(P')}{_\alpha_\beta^\gamma^\delta}\tensor{\psi}{_\gamma_\delta}
	-\tensor{(P'')}{_\alpha_\beta^{\conj{\gamma}}^{\conj{\delta}}}
	\tensor{\psi}{_{\conj{\gamma}}_{\conj{\delta}}}).
\end{equation*}
Combined with \eqref{eq:Tanaka_Webster_expression_of_linearized_obstruction}, this shows that $P''=0$,
or equivalently, $\tensor*{\mathcal{O}}{^\bullet_\alpha_\beta}$ is a certain sum of covariant derivatives of
$\tensor{\psi}{_\alpha_\beta}$.

Note that, again by the parallelity of $J$, we can always take a solution $\sigma$ to
\eqref{eq:linearized_Einstein_equation} that is anti-hermitian, i.e., such that $\tensor{\sigma}{_A_{\conj{B}}}=0$.
This is because the uniqueness statement of Proposition \ref{prop:lichnerowicz_equation} implies that
$\sigma$ must agree with $-\sigma(J\entrydot,J\entrydot)$ modulo $O(\rho^{2n+2})$, which means that
$\tensor{\sigma}{_A_{\conj{B}}}=O(\rho^{2n+2})$;
then, since the higher-order terms are arbitrary, we can set $\tensor{\sigma}{_A_{\conj{B}}}=0$.
The same reasoning as in the preceding paragraph shows that,
if we expand $\tensor{\sigma}{_A_B}$ in the powers of $\rho$,
then all the coefficients can be written in terms of $\tensor{\psi}{_\alpha_\beta}$ only.

We prepare a precise version of Lemma \ref{lem:Laplacian_on_tensors}.

\begin{lem}
	\label{lem:Laplacian_on_complex_hyperbolic_space}
	Let $\sigma$ be an anti-hermitian symmetric 2-$\Theta$-tensor on the complex hyperbolic space
	normalized by the standard contact form $\theta$ on the Heisenberg group $\mathcal{H}$.
	Suppose a local frame $\set{Z_\alpha}$ of $T^{1,0}\mathcal{H}$ is taken so that the Tanaka--Webster connection
	forms are zero. Then, with respect to the local frame
	$\set{\bm{Z}_\tau,\bm{Z}_\alpha,\bm{Z}_{\conj{\tau}},\bm{Z}_{\conj{\alpha}}}$,
	\begin{subequations}
		\label{eq:divergence_of_2-tensor_complex_hyperbolic}
	\begin{align}
		\tensor{(\delta\sigma)}{_\tau}
		&=-\tfrac{1}{4}(\rho\partial_\rho-2n-4)\tensor{\sigma}{_\tau_\tau}
			+\tfrac{i}{2}\rho^2T\tensor{\sigma}{_\tau_\tau}
			-\rho\tensor{Z}{^\alpha}\tensor{\sigma}{_\tau_\alpha},\\
		\tensor{(\delta\sigma)}{_\alpha}
		&=-\tfrac{1}{4}(\rho\partial_\rho-2n-5)\tensor{\sigma}{_\tau_\alpha}
			+\tfrac{i}{2}\rho^2T\tensor{\sigma}{_\tau_\alpha}
			-\rho\tensor{Z}{^\beta}\tensor{\sigma}{_\alpha_\beta},
	\end{align}
	\end{subequations}
	where $T$ is the Reeb vector field and $Z^\alpha=\tensor{h}{^\alpha^{\conj{\beta}}}Z_{\conj{\beta}}$.
	Moreover, under the condition $\delta\sigma=O(\rho^{2n+2})$,
	$\tilde{\sigma}=(\Delta_\mathrm{L}+n+2)\sigma$ is given by, modulo $O(\rho^{2n+2})$,
	\begin{subequations}
		\label{eq:lichnerowicz_Laplacian_of_2-tensor_complex_hyperbolic}
	\begin{align}
		\label{eq:lichnerowicz_Laplacian_of_2-tensor_complex_hyperbolic_1}
			\tensor{\tilde{\sigma}}{_\tau_\tau}
			&\equiv -\tfrac{1}{4}(\rho\partial_\rho-2)(\rho\partial_\rho-2n-4)\tensor{\sigma}{_\tau_\tau}
				+\rho^2(\Delta_b+2iT)\tensor{\sigma}{_\tau_\tau}
				-\rho^4T^2\tensor{\sigma}{_\tau_\tau},\\
		\label{eq:lichnerowicz_Laplacian_of_2-tensor_complex_hyperbolic_2}
			\tensor{\tilde{\sigma}}{_\tau_\alpha}
			&\equiv -\tfrac{1}{4}(\rho\partial_\rho-1)(\rho\partial_\rho-2n-3)\tensor{\sigma}{_\tau_\alpha}
				+\rho^2(\Delta_b+2iT)\tensor{\sigma}{_\tau_\tau}-\rho^4T^2\tensor{\sigma}{_\tau_\tau}
				+\rho Z_\alpha\tensor{\sigma}{_\tau_\tau},\\
		\label{eq:lichnerowicz_Laplacian_of_2-tensor_complex_hyperbolic_3}
			\tensor{\tilde{\sigma}}{_\alpha_\beta}
			&\equiv -\tfrac{1}{4}\rho\partial_\rho(\rho\partial_\rho-2n-2)\tensor{\sigma}{_\alpha_\beta}
				+\rho^2(\Delta_b+2iT)\tensor{\sigma}{_\alpha_\beta}-\rho^4T^2\tensor{\sigma}{_\alpha_\beta}
				+2\rho\tensor{Z}{_(_\alpha_|}\tensor{\sigma}{_\tau_|_\beta_)}.
	\end{align}
	\end{subequations}
\end{lem}

\begin{proof}
	In this setting, one can check that \eqref{eq:Christoffel} gives not only the boundary values of
	the Christoffel symbols of the Levi-Civita connection of $g$ but the exact formula of them.
	Then \eqref{eq:divergence_of_2-tensor_complex_hyperbolic} follows by a direct computation.
	Another long computation shows that
	\begin{subequations}
	\begin{align}
		\label{eq:lichnerowicz_Laplacian_of_2-tensor_complex_hyperbolic_raw_1}
		\begin{split}
			\tensor{\tilde{\sigma}}{_\tau_\tau}
			&=-\tfrac{1}{4}((\rho\partial_\rho)^2-(2n+2)\rho\partial_\rho-4n-8)\tensor{\sigma}{_\tau_\tau}
				+\rho^2(\Delta_b+4iT)\tensor{\sigma}{_\tau_\tau}-\rho^4T^2\tensor{\sigma}{_\tau_\tau}\\
			&\phantom{\;=\;}\qquad-4\rho Z^\gamma\tensor{\sigma}{_\tau_\gamma},
		\end{split}\\
		\label{eq:lichnerowicz_Laplacian_of_2-tensor_complex_hyperbolic_raw_2}
		\begin{split}
			\tensor{\tilde{\sigma}}{_\tau_\alpha}
			&=-\tfrac{1}{4}((\rho\partial_\rho)^2-(2n+2)\rho\partial_\rho-2n-7)\tensor{\sigma}{_\tau_\alpha}
				+\rho^2(\Delta_b+3iT)\tensor{\sigma}{_\tau_\tau}-\rho^4T^2\tensor{\sigma}{_\tau_\tau}\\
			&\phantom{\;=\;}\qquad+\rho Z_\alpha\tensor{\sigma}{_\tau_\tau}
			-2\rho Z^\gamma\tensor{\sigma}{_\alpha_\gamma},
		\end{split}\\
			\tensor{\tilde{\sigma}}{_\alpha_\beta}
			&=-\tfrac{1}{4}\rho\partial_\rho(\rho\partial_\rho-2n-2)\tensor{\sigma}{_\alpha_\beta}
				+\rho^2(\Delta_b+2iT)\tensor{\sigma}{_\alpha_\beta}-\rho^4T^2\tensor{\sigma}{_\alpha_\beta}
				+2\rho\tensor{Z}{_(_\alpha_|}\tensor{\sigma}{_\tau_|_\beta_)}.
	\end{align}
	\end{subequations}
	The last equation is nothing but \eqref{eq:lichnerowicz_Laplacian_of_2-tensor_complex_hyperbolic_3}
	(this is an exact equality actually).
	To show \eqref{eq:lichnerowicz_Laplacian_of_2-tensor_complex_hyperbolic_1} and
	\eqref{eq:lichnerowicz_Laplacian_of_2-tensor_complex_hyperbolic_2}, we use the following identities
	that follow from \eqref{eq:divergence_of_2-tensor_complex_hyperbolic} and the assumption
	$\delta\sigma=O(\rho^{2n+2})$:
	\begin{align*}
		\rho Z^\gamma\tensor{\sigma}{_\tau_\gamma}
		&=-\tfrac{1}{4}(\rho\partial_\rho-2n-4)\tensor{\sigma}{_\tau_\tau}
			+\tfrac{i}{2}\rho^2T\tensor{\sigma}{_\tau_\tau}+O(\rho^{2n+2}),\\
		\rho Z^\gamma\tensor{\sigma}{_\alpha_\gamma}
		&=-\tfrac{1}{4}(\rho\partial_\rho-2n-5)\tensor{\sigma}{_\tau_\alpha}
			+\tfrac{i}{2}\rho^2T\tensor{\sigma}{_\tau_\alpha}+O(\rho^{2n+2}).
	\end{align*}
	In fact, putting them into \eqref{eq:lichnerowicz_Laplacian_of_2-tensor_complex_hyperbolic_raw_1} and
	\eqref{eq:lichnerowicz_Laplacian_of_2-tensor_complex_hyperbolic_raw_2} shows
	\eqref{eq:lichnerowicz_Laplacian_of_2-tensor_complex_hyperbolic_1} and
	\eqref{eq:lichnerowicz_Laplacian_of_2-tensor_complex_hyperbolic_2}.
\end{proof}

\begin{proof}[Proof of Proposition \ref{prop:linearized_obstruction_heisenberg}]
	Let $\sigma$ be an anti-hermitian solution to \eqref{eq:linearized_Einstein_equation} and
	$\tilde{\sigma}=(\Delta_\mathrm{L}+n+2)\sigma$.
	It holds that $\delta\sigma=O(\rho^{2n+2})$ as stated in Proposition \ref{prop:lichnerowicz_equation}.
	We may further improve this solution so that the following holds:
	\begin{equation}
		\label{eq:divergence_improved_solution}
		\delta\sigma=O(\rho^{2n+4}).
	\end{equation}
	In fact, \eqref{eq:divergence_of_2-tensor_complex_hyperbolic} shows
	that the $\rho^{2n+2}$- and $\rho^{2n+3}$-coefficients of $\delta\sigma$ can be controlled by
	those of $\tensor{\sigma}{_\tau_\tau}$ and $\tensor{\sigma}{_\tau_\alpha}$
	(actually \eqref{eq:divergence_of_2-form} suffices for this purpose).
	Equation \eqref{eq:divergence_improved_solution} together with \eqref{eq:Hodge_Laplacian_on_1-forms} implies
	$(\Delta_\mathrm{H}+2n+4)\tensor{(\delta\sigma)}{_\tau}=O(\rho^{2n+5})$ and
	$(\Delta_\mathrm{H}+2n+4)\tensor{(\delta\sigma)}{_\alpha}=O(\rho^{2n+4})$.
	Since $\Ric$ is parallel, \eqref{eq:divergence_Laplacian_commutation} shows that
	$\delta\circ\Delta_\mathrm{L}=\Delta_\mathrm{H}\circ\delta$, which gives us
	\begin{equation*}
		\tensor{(\delta\tilde{\sigma})}{_\tau}=O(\rho^{2n+5})\qquad\text{and}\qquad
		\tensor{(\delta\tilde{\sigma})}{_\alpha}=O(\rho^{2n+4}).
	\end{equation*}
	Then, by \eqref{eq:divergence_of_2-tensor_complex_hyperbolic} we conclude that
	$\tensor{\tilde{\sigma}}{_\tau_\alpha}=O(\rho^{2n+3})$ and
	$\tensor{\tilde{\sigma}}{_\tau_\tau}=O(\rho^{2n+4})$.
	Let $\tensor*{k}{^{(2n+3)}_\tau_\alpha}\in\tensor{\mathcal{E}}{_\alpha}$ and
	$\tensor*{k}{^{(2n+2)}_\alpha_\beta}\in\tensor{\mathcal{E}}{_(_\alpha_\beta_)}$ be defined by
	\begin{equation*}
		\tensor{\tilde{\sigma}}{_\tau_\alpha}
		=\rho^{2n+3}\tensor*{k}{^{(2n+3)}_\tau_\alpha}+O(\rho^{2n+4})\qquad\text{and}\qquad
		\tensor{\tilde{\sigma}}{_\alpha_\beta}
		=\rho^{2n+2}\tensor*{k}{^{(2n+2)}_\alpha_\beta}+O(\rho^{2n+3}).
	\end{equation*}
	Again by \eqref{eq:divergence_of_2-tensor_complex_hyperbolic}, if $\set{Z_\alpha}$ is such that
	the Tanaka--Webster connection forms vanish,
	\begin{subequations}
	\begin{equation}
		\label{eq:divergence_of_high_coefficient_1}
		2Z^\beta\tensor*{k}{^{(2n+2)}_\alpha_\beta}=\tensor*{k}{^{(2n+3)}_\tau_\alpha}
	\end{equation}
	and
	\begin{equation}
		\label{eq:divergence_of_high_coefficient_2}
		Z^\alpha\tensor*{k}{^{(2n+3)}_\tau_\alpha}=0.
	\end{equation}
	\end{subequations}

	We want to compute $\tensor*{k}{^{(2n+2)}_\alpha_\beta}\in\tensor{\mathcal{E}}{_\alpha_\beta}$ explicitly,
	which by Proposition \ref{prop:lichnerowicz_equation} equals $-\tensor*{\mathcal{O}}{^\bullet_\alpha_\beta}$
	trivialized by $\theta$.
	As a step toward this, we first write $\tensor*{k}{^{(2n+3)}_\tau_\alpha}$ down.
	By \eqref{eq:divergence_of_high_coefficient_1} and the discussion preceding
	Lemma \ref{lem:Laplacian_on_complex_hyperbolic_space},
	$\tensor*{k}{^{(2n+3)}_\tau_\alpha}$ is a certain sum of contractions of derivatives of
	$\tensor{\psi}{_\alpha_\beta}$.
	We divide the terms into two groups depending whether the two indices of $\psi$ are both contracted or
	one of them is uncontracted; we call those terms that belong to the first group doubly-contracted.
	Since all the coefficients of $\tensor{\sigma}{_\tau_\tau}$ have to be doubly-contracted,
	\eqref{eq:lichnerowicz_Laplacian_of_2-tensor_complex_hyperbolic_2} shows that
	\begin{equation*}
			\tensor{\tilde{\sigma}}{_\tau_\alpha}
			\equiv-\tfrac{1}{4}(\rho\partial_\rho-1)(\rho\partial_\rho-2n-3)\tensor{\sigma}{_\tau_\alpha}
				+\rho^2(\Delta_b+2iT)\tensor{\sigma}{_\tau_\alpha}-\rho^4T^2\tensor{\sigma}{_\tau_\alpha}
				+O(\rho^{2n+2}),
	\end{equation*}
	where $\equiv$ means that we omit the doubly-contracted terms.
	This implies that, omitting doubly-contracted terms, we can write down the expansion of
	$\tensor{\sigma}{_\tau_\alpha}$ as
	\begin{equation*}
		\tensor{\sigma}{_\tau_\alpha}
		\equiv\frac{2}{n+2}\sum_{l=0}^{n-1}\rho^{2l+1}c'_{2l+1}P'_{2l+1}
		Z^\beta\tensor{\psi}{_\alpha_\beta}+O(\rho^{2n+3}),\qquad
		c'_l=\left(l!\prod_{\nu=0}^{l-1}(l-n-1-\nu)\right)^{-1},
	\end{equation*}
	where we define the differential operators
	$P'_{2l+1}\colon\tensor{\mathcal{E}}{_\alpha}\longrightarrow\tensor{\mathcal{E}}{_\alpha}$ by
	\begin{equation*}
		P'_1=1,\qquad P'_3=\Delta_b,\qquad
		P'_{2l+1}=(\Delta_b+2iT)P'_{2l-1}-(l-1)(l-n-2)T^2P'_{2l-3}.
	\end{equation*}
	Consequently, by Lemma \ref{lem:Masakis_lemma},
	\begin{equation*}
		\tensor*{k}{^{(2n+3)}_\tau_\alpha}
		\equiv \frac{2}{n+2}c'_nP'_{2n+3}Z^\beta\tensor{\psi}{_\alpha_\beta}
		=\frac{2\cdot(-1)^n}{(n+2)\cdot n!^2}
		\left(\prod_{l=0}^n(\Delta_b+(n+2-2l)iT)\right)Z^\beta\tensor{\psi}{_\alpha_\beta}.
	\end{equation*}
	Now we determine the omitted doubly-contracted terms using \eqref{eq:divergence_of_high_coefficient_2}.
	There is some polynomial $q'(\Delta_b,T)$ of $\Delta_b$ and $T$ such that
	\begin{equation*}
		\tensor*{k}{^{(2n+3)}_\tau_\alpha}
		=\frac{2\cdot(-1)^n}{(n+2)\cdot n!^2}
		\left(\prod_{l=0}^n(\Delta_b+(n+2-2l)iT)\right)Z^\beta\tensor{\psi}{_\alpha_\beta}
		+q'(\Delta_b,T)Z_\alpha Z^\beta Z^\gamma\tensor{\psi}{_\beta_\gamma}.
	\end{equation*}
	Since \eqref{eq:divergence_of_high_coefficient_2} holds,
	by the commutation relations $[Z^\alpha,\Delta_b]=-2iTZ^\alpha$ and $[Z^\alpha,T]=0$ we can compute that
	\begin{equation*}
		\begin{split}
			Z^\alpha q'(\Delta_b,T)Z_\alpha Z^\beta Z^\gamma\tensor{\psi}{_\beta_\gamma}
			&=\frac{2\cdot(-1)^{n+1}}{(n+2)\cdot n!^2}
			Z^\alpha\left(\prod_{l=0}^n(\Delta_b+(n+2-2l)iT)\right)Z^\beta\tensor{\psi}{_\alpha_\beta}\\
			&=\frac{2\cdot(-1)^{n+1}}{(n+2)\cdot n!^2}
			\left(\prod_{l=0}^n(\Delta_b+(n-2l)iT)\right)Z^\alpha Z^\beta\tensor{\psi}{_\alpha_\beta}.
		\end{split}
	\end{equation*}
	This implies that
	\begin{equation*}
		\frac{2\cdot(-1)^{n+1}}{(n+2)\cdot n!^2}\prod_{l=0}^n(\Delta_b+(n-2l)iT)
		=Z^\alpha q'(\Delta_b,T)Z_\alpha=q'(\Delta_b-2iT,T)Z^\alpha Z_\alpha.
	\end{equation*}
	Since $Z^\alpha Z_\alpha=-\frac{1}{2}(\Delta_b-inT)$, we obtain
	\begin{equation*}
		q'(\Delta_b-2iT,T)=\frac{4\cdot(-1)^n}{(n+2)\cdot n!^2}\prod_{l=0}^{n-1}(\Delta_b+(n-2l)iT),
	\end{equation*}
	and hence
	\begin{multline}
		\label{eq:expression_tau_alpha}
		\tensor*{k}{^{(2n+3)}_\tau_\alpha}
		=\frac{2\cdot(-1)^n}{(n+2)\cdot n!^2}\left[
		\left(\prod_{l=0}^n(\Delta_b+(n+2-2l)iT)\right)Z^\beta\tensor{\psi}{_\alpha_\beta}
		\right.\\
		+\left.2\left(\prod_{l=0}^{n-1}(\Delta_b+(n+2-2l)iT)\right)
			Z_\alpha Z^\beta Z^\gamma\tensor{\psi}{_\beta_\gamma}\right].
	\end{multline}

	Finally we compute $\tensor*{k}{^{(2n+2)}_\alpha_\beta}$.
	It is expressed as a sum of contractions of derivatives of $\tensor{\psi}{_\alpha_\beta}$,
	and we divide the terms into two groups: those in which at least one of the two indices of $\psi$ is
	contracted (which we call contracted terms),
	and those in which the two indices of $\psi$ are both uncontracted.
	Then, if we omit the contracted terms,
	\eqref{eq:lichnerowicz_Laplacian_of_2-tensor_complex_hyperbolic_2} shows that
	\begin{equation*}
		\tensor{\tilde{\sigma}}{_\alpha_\beta}
			\equiv-\tfrac{1}{4}\rho\partial_\rho(\rho\partial_\rho-2n-2)\tensor{\sigma}{_\alpha_\beta}
				+\rho^2(\Delta_b+2iT)\tensor{\sigma}{_\tau_\alpha}-\rho^4T^2\tensor{\sigma}{_\tau_\alpha}
				+O(\rho^{2n+2}).
	\end{equation*}
	This implies that, omitting the contracted terms, we can write
	\begin{equation*}
		\tensor{\sigma}{_\alpha_\beta}
		\equiv\sum_{l=0}^{n-1}\rho^{2l}c'_lP''_{2l}
		\tensor{\psi}{_\alpha_\beta}+O(\rho^{2n+3}),
	\end{equation*}
	where we define
	$P''_{2l}\colon\tensor{\mathcal{E}}{_(_\alpha_\beta_)}\longrightarrow\tensor{\mathcal{E}}{_(_\alpha_\beta_)}$
	by
	\begin{equation*}
		P''_0=1,\qquad P''_2=\Delta_b,\qquad
		P''_{2l}=(\Delta_b+2iT)P''_{2l-2}-(l-1)(l-n-2)T^2P''_{2l-4}.
	\end{equation*}
	Consequently, again by using Lemma \ref{lem:Masakis_lemma}, we obtain
	\begin{equation*}
		\tensor*{k}{^{(2n+2)}_\alpha_\beta}
		\equiv c'_nP''_{2n+2}\tensor{\psi}{_\alpha_\beta}
		=\frac{(-1)^n}{n!^2}\left(\prod_{l=0}^n(\Delta_b+(n+2-2l)iT)\right)\tensor{\psi}{_\alpha_\beta}.
	\end{equation*}
	We shall determine the omitted terms. There are some polynomials $q''_1(\Delta_b,T)$ and $q''_2(\Delta_b,T)$
	such that
	\begin{multline*}
		\tensor*{k}{^{(2n+2)}_\alpha_\beta}
		=\frac{(-1)^n}{n!^2}\left(\prod_{l=0}^n(\Delta_b+(n+2-2l)iT)\right)\tensor{\psi}{_\alpha_\beta}\\
		+q''_1(\Delta_b,T)\tensor{Z}{_(_\alpha}Z^\gamma\tensor{\psi}{_\beta_)_\gamma}
		+q''_2(\Delta_b,T)\tensor{Z}{_\alpha}Z_\beta Z^\gamma Z^\delta\tensor{\psi}{_\gamma_\delta}.
	\end{multline*}
	Then,
	\begin{multline}
		\label{eq:expression_divergence_of_alpha_beta}
		Z^\beta\tensor*{k}{^{(2n+2)}_\alpha_\beta}
		=\frac{(-1)^n}{n!^2}\left(\prod_{l=0}^n(\Delta_b+(n-2l)iT)\right)Z^\beta\tensor{\psi}{_\alpha_\beta}\\
		+q''_1(\Delta_b-2iT,T)Z^\beta\tensor{Z}{_(_\alpha}Z^\gamma\tensor{\psi}{_\beta_)_\gamma}
		+q''_2(\Delta_b-2iT,T)Z^\beta Z_\alpha Z_\beta Z^\gamma Z^\delta\tensor{\psi}{_\gamma_\delta}.
	\end{multline}
	Note that the following holds:
	\begin{align*}
		Z^\beta\tensor{Z}{_(_\alpha}Z^\gamma\tensor{\psi}{_\beta_)_\gamma}
		&=-\frac{1}{4}(\Delta_b-i(n+2)T)Z^\beta\tensor{\psi}{_\alpha_\beta}
		+\frac{1}{2}Z_\alpha Z^\beta Z^\gamma\tensor{\psi}{_\beta_\gamma},\\
		Z^\beta Z_\alpha Z_\beta Z^\gamma Z^\delta\tensor{\psi}{_\gamma_\delta}
		&=-\frac{1}{2}(\Delta_b-inT)Z_\alpha Z^\beta Z^\gamma\tensor{\psi}{_\beta_\gamma}.
	\end{align*}
	Therefore, by \eqref{eq:divergence_of_high_coefficient_1}, \eqref{eq:expression_tau_alpha}, and
	\eqref{eq:expression_divergence_of_alpha_beta},
	\begin{multline*}
		\frac{(-1)^n}{(n+2)\cdot n!^2}\prod_{l=0}^n(\Delta_b+(n+2-2l)iT)\\
		=\frac{(-1)^n}{n!^2}\prod_{l=0}^n(\Delta_b+(n-2l)iT)
		-\frac{1}{4}(\Delta_b-i(n+2)T)q''_1(\Delta_b-2iT,T)
	\end{multline*}
	and
	\begin{equation*}
		\frac{2\cdot(-1)^n}{(n+2)\cdot n!^2}\prod_{l=0}^{n-1}(\Delta_b+(n+2-2l)iT)
		=\frac{1}{2}Q''_1(\Delta_b-2iT,T)-\frac{1}{2}(\Delta_b-inT)q''_2(\Delta_b-2iT,T).
	\end{equation*}
	Hence
	\begin{align*}
		q''_1(\Delta_b,T)
		&=\frac{(-1)^n}{(n!)^2}\cdot\frac{4(n+1)}{n+2}\prod_{l=0}^{n-1}(\Delta_b+(n+2-2l)iT)\\
		\intertext{and}
		q''_2(\Delta_b,T)
		&=\frac{(-1)^n}{(n!)^2}\cdot\frac{4n}{n+2}\left(\prod_{l=0}^{n-2}(\Delta_b+(n+2-2l)iT)\right),
	\end{align*}
	which show that $\tensor*{k}{^{(2n+2)}_\alpha_\beta}$ is minus of the
	right-hand side of \eqref{eq:principal_part_of_linearized_obstruction}.
\end{proof}

\subsection{Further properties of the linearized obstruction operator}
\label{subsec:further_properties}

We now consider the case in which $M$ is an integrable CR manifold and
prove Theorem \ref{thm:obstruction_operator_observation}.
By formal embedding, we may assume that $M$ is the boundary of a domain $\Omega$ in $\mathbb{C}^{n+1}$.
In this case we can take $X$ to be the square root of $\overline{\Omega}$ in the sense of
Epstein, Melrose, and Mendoza (see Example \ref{ex:Bergman_type}).
As described in Subsection \ref{subsec:approximate_ACHE}, Fefferman's approximate solution to the complex
Monge--Amp\`ere equation defines a Bergman-type metric $g$ that satisfies \eqref{eq:approx_Einstein_Kahler} if
considered as an ACH metric on $X$.
The complex structure $J$ on $\overline{\Omega}$, with respect to which $g$ is K\"ahler, is naturally regarded
as a section of $\End(\thetatangent)$. The $i$-eigenbundle is denoted by $(\thetatangent)^{1,0}$.

\begin{proof}[Proof of Theorem \ref{thm:obstruction_operator_observation} (1)]
	Since $g$ is K\"ahler, we can apply the argument for the complex hyperbolic metric
	in Subsection \ref{subsec:Heisenberg_principal_part} also in this case, and we conclude that
	$\tensor*{\mathcal{O}}{^\bullet_\alpha_\beta}=\tensor{(\mathcal{O}^\bullet\psi)}{_\alpha_\beta}$
	is written as a sum of contractions of Tanaka--Webster local invariants (i.e. covariant derivatives of
	$N$, $A$, and $R$) and covariant derivatives of $\tensor{\psi}{_\alpha_\beta}$,
	which means that $\tensor*{\mathcal{O}}{^\bullet_\alpha_\beta}$ is complex-linear in
	$\tensor{\psi}{_\alpha_\beta}$.

	In view of the fact that $\mathcal{O}^\bullet$ has a universal expression in terms of the Tanaka--Webster
	connection,
	to prove the formal self-adjointness of $\mathcal{O}^\bullet$, it suffices to consider the case in which
	$M$ is compact. We use Theorem \ref{thm:first_variational_formula} as follows.
	Let $\tensor{\chi}{_\alpha_\beta}$,
	$\tensor{\psi}{_\alpha_\beta}\in\tensor{\mathcal{E}}{_(_\alpha_\beta_)}(1,1)$ and, for sufficiently small
	$\varepsilon>0$, we define $T^{1,0}_{s,t}$ to be the partially integrable CR structure spanned by
	$\set{Z_\alpha+\tensor{\varphi}{_\alpha^{\conj{\beta}}}Z_{\conj{\beta}}}$,
	where $s$, $t\in(-\varepsilon,\varepsilon)$ and
	$\tensor{\varphi}{_\alpha_\beta}=s\tensor{\chi}{_\alpha_\beta}+t\tensor{\psi}{_\alpha_\beta}$.
	If $\smash{\overline{Q}}^{s,t}$ is the total CR $Q$-curvature of $(M,T^{1,0}_{s,t})$, then it is
	smooth in $s$, $t$ and hence
	$\partial_s\partial_t\smash{\overline{Q}}^{s,t}=\partial_t\partial_s\smash{\overline{Q}}^{s,t}$.
	Evaluated at $s=t=0$, this implies that
	\begin{equation*}
		\int_M \re\braket{\mathcal{O}^\bullet\chi,\psi}=\int_M \re\braket{\mathcal{O}^\bullet\psi,\chi},
	\end{equation*}
	where the bracket denotes the Hermitian inner product.
	The right-hand side is the same as the integral of $\re\braket{\chi,\mathcal{O}^\bullet\psi}$.
	By replacing $\psi$ with $i\psi$, we also obtain the equality between the integrals of the imaginary parts of
	$\braket{\mathcal{O}^\bullet\chi,\psi}$ and $\braket{\chi,\mathcal{O}^\bullet\psi}$,
	thereby showing the self-adjointness of $\mathcal{O}^\bullet$.
\end{proof}

Next we give a direct proof of the second equality of \eqref{eq:double_divergence_free}.
Here we abandon the previous notation of local frames and introduce a new one.
We first define the $(1,0)$-vector field $\xi$ on $\overline{\Omega}$ near the boundary by the requirement
\begin{equation*}
	\partial r(\xi)=1,\qquad\text{and}\qquad \xi\contraction\partial\conj{\partial}r=\kappa\conj{\partial}r
\end{equation*}
for some function real-valued function $\kappa$, which is called the \emph{transverse curvature} of
$r$~\cite{Graham_Lee_88}.
Then we take a local frame $\set{\zeta_\alpha}$ of $\ker\partial r\subset T^{1,0}\overline{\Omega}$
and set $\bm{\xi}=r\xi$ and $\bm{\zeta}_\alpha=\sqrt{r/2}\,\zeta_\alpha$, so that
$\set{\bm{\xi},\bm{\zeta}_\alpha}$ spans $(\thetatangent)^{1,0}$.
The index notation in this subsection is for this frame;
the index $\tau$ is associated to $\bm{\xi}$, which we also write $\bm{\zeta}_\tau$.

We define $\tilde{\theta}=\frac{i}{2}(\partial r-\conj{\partial}r)$ and write
$h(Z,\conj{W})=-i\,d\tilde{\theta}(Z,\conj{W})$ for vectors $Z$, $W$ in $T^{1,0}\overline{\Omega}$.
Then,
\begin{equation}
	\label{eq:ddbar_r}
	\partial\conj{\partial}r
	=\kappa\partial r\wedge\conj{\partial}r
	-\tensor{h}{_\alpha_{\conj{\beta}}}\theta^\alpha\wedge\theta^{\conj{\beta}},
\end{equation}
where $\tensor{h}{_\alpha_{\conj{\beta}}}=h(\zeta_\alpha,\zeta_{\conj{\beta}})$,
and $\theta^\alpha$ are taken so that
$\set{\partial r,\theta^\alpha}$ is the dual coframe of $\set{\xi,\zeta_\alpha}$.
Consequently, we have
\begin{equation}
	\label{eq:Bergman_type}
	g=4(1-\kappa r)\bm{\theta}^\tau\bm{\theta}^{\conj{\tau}}
	+2\tensor{h}{_\alpha_{\conj{\beta}}}\bm{\theta}^\alpha\bm{\theta}^{\conj{\beta}},
\end{equation}
where $\bm{\theta}^\tau=\partial r/r$ and $\bm{\theta}^\alpha=\theta^\alpha/\sqrt{r/2}$.

Equation \eqref{eq:approx_Einstein_Kahler} implies that $D\Ric=O(\rho^{2n+4})$.
Moreover, we remark that
\begin{equation}
	\label{eq:differentiated_Ricci_tensor_Kahler}
	\tensor{(D\Ric)}{_\tau_{\conj{\alpha}}_{\conj{\beta}}}=O(\rho^{2n+5}).
\end{equation}
This can be seen as follows. Since $\Ric$ has hermitian-type components only,
$\tensor{(D\Ric)}{_\tau_{\conj{\alpha}}_{\conj{\beta}}}$ equals
$-2\tensor{\nabla}{_(_{\conj{\alpha}}_|}\tensor{\Ric}{_\tau_|_{\conj{\beta}}_)}$, which is
$-2\tensor{\nabla}{_(_{\conj{\alpha}}_|}\tensor{E}{_\tau_|_{\conj{\beta}}_)}$ if we set $E={\Ric}+\frac{n+2}{2}g$.
Then, because
\begin{equation*}
	\tensor{\nabla}{_{\conj{\alpha}}}\tensor{E}{_\tau_{\conj{\beta}}}
	=\bm{\zeta}_{\conj{\alpha}}\tensor{E}{_\tau_{\conj{\beta}}}
	-\tensor{\Gamma}{^\tau_{\conj{\alpha}}_\tau}\tensor{E}{_\tau_{\conj{\beta}}}
	-\tensor{\Gamma}{^\gamma_{\conj{\alpha}}_\tau}\tensor{E}{_\gamma_{\conj{\beta}}}
	-\tensor{\Gamma}{^{\conj{\tau}}_{\conj{\alpha}}_{\conj{\beta}}}\tensor{E}{_\tau_{\conj{\tau}}}
	-\tensor{\Gamma}{^{\conj{\gamma}}_{\conj{\alpha}}_{\conj{\beta}}}\tensor{E}{_\tau_{\conj{\gamma}}},
\end{equation*}
by \eqref{eq:approx_Einstein_Kahler} and \eqref{eq:Christoffel} we obtain
\eqref{eq:differentiated_Ricci_tensor_Kahler}.

\begin{lem}
	\label{lem:divergence_on_KE}
	Let $\sigma$ be an $O(\rho^j)$ anti-hermitian symmetric 2-$\Theta$-tensor on $(X,g)$,
	where $X$ is the square root of $\cl{\Omega}$ and
	$g$ is the Bergman-type metric \eqref{eq:Bergman_type_metric} given by a $C^\infty$-smooth boundary defining
	function $r$ of $\cl{\Omega}$, which is regarded as a $\Theta$-metric.
	Let $\rho=\sqrt{r/2}$. Then, with respect to a local frame
	$\set{\bm{\zeta}_\tau,\bm{\zeta}_\alpha,\bm{\zeta}_{\conj{\tau}},\bm{\zeta}_{\conj{\alpha}}}$,
	\begin{subequations}
		\label{eq:divergence_of_2_tensor_KE}
	\begin{align}
		\tensor{(\delta\sigma)}{_\tau}
		&=-\tfrac{1}{4}(\rho\partial_\rho-2n-4)\tensor{\sigma}{_\tau_\tau}
			-\rho\tensor{(\nabla^\mathrm{TW})}{^\alpha}\tensor{\sigma}{_\tau_\alpha}+O(\rho^{j+2}),\\
		\tensor{(\delta\sigma)}{_\alpha}
		&=-\tfrac{1}{4}(\rho\partial_\rho-2n-5)\tensor{\sigma}{_\tau_\alpha}
			-\rho\tensor{(\nabla^\mathrm{TW})}{^\beta}\tensor{\sigma}{_\alpha_\beta}+O(\rho^{j+2}),
	\end{align}
	\end{subequations}
	where $\nabla^\mathrm{TW}$ is the Tanaka--Webster connection of the contact form
	$\theta=\tilde{\theta}|_{TM}$, which acts on $\tensor{\sigma}{_\tau_\alpha}$ and
	$\tensor{\sigma}{_\alpha_\beta}$ by interpreting them as tensors in
	$\tensor{\mathcal{E}}{_\alpha}$ and $\tensor{\mathcal{E}}{_(_\alpha_\beta_)}$ with parameter $r$.
\end{lem}

\begin{proof}
	We compute the Christoffel symbols $\tensor{\Gamma}{^R_P_Q}$ of the Levi-Civita connection of $g$
	modulo $O(\rho^2)$, i.e., modulo $O(r)$,
	which are considered as functions on $M$ with parameter $r$ rather than as functions on $X$.
	Recall that we can take such a local frame $\set{\zeta_\alpha}$ that the Tanaka--Webster connection
	forms for $\theta$ with respect to $\set{\zeta_\alpha|_M}$ vanish at a prescribed point $p\in M$
	(see~\cite{Lee_88}*{Lemma 2.1}). We compute using such a frame, and the following equalities
	are to be understood as equalities at $p$.
	Note that $(\zeta_\gamma\tensor{h}{_\alpha_{\conj{\beta}}})|_M$ and
	$(\zeta_{\conj{\gamma}}\tensor{h}{_\alpha_{\conj{\beta}}})|_M$ vanish at $p$ for such a local frame.
	First, $\tensor{g}{_\tau_{\conj{\tau}}}=2(1-\kappa r)$, $\tensor{g}{_\tau_{\conj{\alpha}}}=0$, and
	$\tensor{g}{_\alpha_{\conj{\beta}}}=\tensor{h}{_\alpha_{\conj{\beta}}}$ imply that
	$\bm{\zeta}_{\conj{C}}\tensor{g}{_A_{\conj{B}}}$ are $O(r)$. On the other hand,
	\begin{equation*}
		d\bm{\theta}^\tau=d\left(\frac{\partial r}{r}\right)
		=(1-\kappa r)\bm{\theta}^\tau\wedge\bm{\theta}^{\conj{\tau}}
		+\frac{1}{2}\tensor{h}{_\alpha_{\conj{\beta}}}\bm{\theta}^\alpha\wedge\bm{\theta}^{\conj{\beta}}
	\end{equation*}
	and
	\begin{equation}
		\label{eq:str_equation_Graham_Lee}
		d\bm{\theta}^\gamma=d\left(\frac{\theta^\gamma}{\sqrt{r/2}}\right)
		=-\frac{1}{2}(1-\kappa r)(\bm{\theta}^\tau+\bm{\theta}^{\conj{\tau}})\wedge\bm{\theta}^\gamma
		+\bm{\theta}^\beta\wedge\tensor{\varphi}{_\beta^\gamma}
		-ir\tensor{A}{_{\conj{\beta}}^\gamma}\bm{\theta}^\tau\wedge\bm{\theta}^{\conj{\beta}}+O(r^{3/2}),
	\end{equation}
	where $\tensor{\varphi}{_\beta^\gamma}$ is Graham--Lee's connection
	forms~\cite{Graham_Lee_88}*{Proposition 1.1} for $r$ and $\tensor{A}{_{\conj{\beta}}^\gamma}$ is, if
	restricted to each hypersurface $M_\varepsilon=\set{r=\varepsilon}$, the Tanaka--Webster torsion for
	$\theta_\varepsilon=\tilde{\theta}|_{TM_\varepsilon}$.
	We do not need the details about Graham--Lee's connection; the point here is that it restricts to
	the Tanaka--Webster connection on $M$. Therefore, because of our choice of frame,
	$\tensor{\varphi}{_\beta^\gamma}(\zeta_\alpha)|_M$ and
	$\tensor{\varphi}{_\beta^\gamma}(\zeta_{\conj{\alpha}})|_M$ are both zero at $p$. Consequently we obtain
	\begin{alignat*}{3}
		[\bm{\zeta}_\tau,\bm{\zeta}_\tau]&=0,&\qquad
		[\bm{\zeta}_\tau,\bm{\zeta}_\alpha]&=\frac{1}{2}\bm{\zeta}_\alpha+O(r),&\qquad
		[\bm{\zeta}_\alpha,\bm{\zeta}_\beta]&=O(r),\\
		[\bm{\zeta}_\tau,\bm{\zeta}_{\conj{\tau}}]&=-(\bm{\zeta}_\tau-\bm{\zeta}_{\conj{\tau}})+O(r),&\qquad
		[\bm{\zeta}_\tau,\bm{\zeta}_{\conj{\alpha}}]&=\frac{1}{2}\bm{\zeta}_{\conj{\alpha}}+O(r),&\qquad
		[\bm{\zeta}_\alpha,\bm{\zeta}_{\conj{\beta}}]
		&=-\frac{1}{2}\tensor{h}{_\alpha_{\conj{\beta}}}(\bm{\zeta}_\tau-\bm{\zeta}_{\conj{\tau}})+O(r).
	\end{alignat*}
	These results imply that the formulae \eqref{eq:Christoffel} of the Christoffel symbols remain
	to hold in a stronger sense, that is, modulo $O(r)$.
	Hence \eqref{eq:divergence_of_2-tensor_complex_hyperbolic} also holds in this case if we omit
	$O(r)$ times the components of $\sigma$, which are $O(\rho^{j+2})$.
\end{proof}

\begin{proof}[Direct proof of the second equality of Theorem \ref{thm:obstruction_operator_observation} (2)]
	By the same argument as in Subsection \ref{subsec:Heisenberg_principal_part},
	we may take an anti-hermitian solution $\sigma$ to \eqref{eq:linearized_Einstein_equation} such that
	$\delta\sigma=O(\rho^{2n+4})$, and hence \eqref{eq:Hodge_Laplacian_on_1-forms} implies that
	$\tensor{(\Delta_\mathrm{H}+n+2)(\delta\sigma)}{_\tau}=O(\rho^{2n+5})$ and
	$\tensor{(\Delta_\mathrm{H}+n+2)(\delta\sigma)}{_\alpha}=O(\rho^{2n+4})$.
	Because of \eqref{eq:approx_Einstein_Kahler} and \eqref{eq:differentiated_Ricci_tensor_Kahler},
	if we set $\tilde{\sigma}=(\Delta_\mathrm{L}+n+2)\sigma$, then
	\begin{equation}
		\label{eq:divergence_freeness_Kahler}
		\tensor{(\delta\tilde{\sigma})}{_\tau}=O(\rho^{2n+5})\qquad\text{and}\qquad
		\tensor{(\delta\tilde{\sigma})}{_\alpha}=O(\rho^{2n+4}).
	\end{equation}
	Since $\tilde{\sigma}$ is an $O(\rho^{2n+2})$ anti-hermitian symmetric $2$-$\Theta$-tensor for which
	$\delta\tilde{\sigma}=O(\rho^{2n+4})$, \eqref{eq:divergence_of_2_tensor_KE} shows
	that $\tensor{\tilde{\sigma}}{_\tau_\alpha}=O(\rho^{2n+3})$ and
	$\tensor{\tilde{\sigma}}{_\tau_\tau}=O(\rho^{2n+4})$.
	Let $\tensor*{k}{^{(2n+3)}_\tau_\alpha}\in\tensor{\mathcal{E}}{_\alpha}$ and
	$\tensor*{k}{^{(2n+2)}_\alpha_\beta}\in\tensor{\mathcal{E}}{_(_\alpha_\beta_)}$ be defined by
	\begin{equation*}
		\tensor{\tilde{\sigma}}{_\tau_\alpha}=\rho^{2n+3}\tensor*{k}{^{(2n+3)}_\tau_\alpha}+O(\rho^{2n+4})
		\qquad\text{and}\qquad
		\tensor{\tilde{\sigma}}{_\alpha_\beta}=\rho^{2n+2}\tensor*{k}{^{(2n+2)}_\alpha_\beta}+O(\rho^{2n+3}).
	\end{equation*}
	We shall write down what \eqref{eq:divergence_freeness_Kahler} means in terms of
	$\tensor*{k}{^{(2n+3)}_\tau_\alpha}$ and $\tensor*{k}{^{(2n+2)}_\alpha_\beta}$.
	As in the proof of the previous lemma, we take $\set{\zeta_\alpha}$ so that
	the Tanaka--Webster connection forms with respect to $\set{\zeta_\alpha|_M}$ vanish at a point $p\in M$,
	and conduct the computation at this point. Since
	\begin{align*}
		\tensor{\nabla}{_{\conj{\tau}}}\tensor{\tilde{\sigma}}{_\tau_\tau}
		&=\bm{\zeta}_{\conj{\tau}}\tensor{\tilde{\sigma}}{_\tau_\tau}
		-2\tensor{\Gamma}{^\tau_{\conj{\tau}}_\tau}\tensor{\tilde{\sigma}}{_\tau_\tau}
		-2\tensor{\Gamma}{^\alpha_{\conj{\tau}}_\tau}\tensor{\tilde{\sigma}}{_\tau_\alpha},\\
		\tensor{\nabla}{_{\conj{\beta}}}\tensor{\tilde{\sigma}}{_\tau_\alpha}
		&=\bm{\zeta}_{\conj{\beta}}\tensor{\tilde{\sigma}}{_\tau_\alpha}
		-\tensor{\Gamma}{^\tau_{\conj{\beta}}_\tau}\tensor{\tilde{\sigma}}{_\tau_\alpha}
		-\tensor{\Gamma}{^\gamma_{\conj{\beta}}_\tau}\tensor{\tilde{\sigma}}{_\gamma_\alpha}
		-\tensor{\Gamma}{^\tau_{\conj{\beta}}_\alpha}\tensor{\tilde{\sigma}}{_\tau_\tau}
		-\tensor{\Gamma}{^\gamma_{\conj{\beta}}_\alpha}\tensor{\tilde{\sigma}}{_\tau_\gamma},\\
		\tensor{\nabla}{_{\conj{\tau}}}\tensor{\tilde{\sigma}}{_\tau_\alpha}
		&=\bm{\zeta}_{\conj{\tau}}\tensor{\tilde{\sigma}}{_\tau_\alpha}
		-\tensor{\Gamma}{^\tau_{\conj{\tau}}_\tau}\tensor{\tilde{\sigma}}{_\tau_\alpha}
		-\tensor{\Gamma}{^\beta_{\conj{\tau}}_\tau}\tensor{\tilde{\sigma}}{_\beta_\alpha}
		-\tensor{\Gamma}{^\tau_{\conj{\tau}}_\alpha}\tensor{\tilde{\sigma}}{_\tau_\tau}
		-\tensor{\Gamma}{^\beta_{\conj{\tau}}_\alpha}\tensor{\tilde{\sigma}}{_\tau_\beta},\\
		\tensor{\nabla}{_{\conj{\gamma}}}\tensor{\tilde{\sigma}}{_\beta_\alpha}
		&=\bm{\zeta}_{\conj{\gamma}}\tensor{\tilde{\sigma}}{_\beta_\alpha}
		-\tensor{\Gamma}{^\tau_{\conj{\gamma}}_\beta}\tensor{\tilde{\sigma}}{_\tau_\alpha}
		-\tensor{\Gamma}{^\delta_{\conj{\gamma}}_\beta}\tensor{\tilde{\sigma}}{_\delta_\alpha}
		-\tensor{\Gamma}{^\tau_{\conj{\gamma}}_\alpha}\tensor{\tilde{\sigma}}{_\tau_\beta}
		-\tensor{\Gamma}{^\delta_{\conj{\gamma}}_\alpha}\tensor{\tilde{\sigma}}{_\tau_\delta},
	\end{align*}
	the fact that \eqref{eq:Christoffel} is true modulo $O(\rho^2)$ shows that
	\begin{subequations}
	\begin{align}
		\label{eq:divergence_freeness_KE_meaning}
		O(\rho^{2n+5})
		&=\tensor{(\delta\tilde{\sigma})}{_\tau}
		=-\rho\tensor{h}{^\alpha^{\conj{\beta}}}\zeta_{\conj{\beta}}\tensor{\tilde{\sigma}}{_\tau_\alpha}
		+\rho\tensor{h}{^\alpha^{\conj{\beta}}}
		\tensor{\Gamma}{^\gamma_{\conj{\beta}}_\tau}\tensor{\tilde{\sigma}}{_\gamma_\alpha}+O(\rho^{2n+5}),\\
		O(\rho^{2n+4})
		&=\tensor{(\delta\tilde{\sigma})}{_\alpha}
		=-\rho\tensor{h}{^\beta^{\conj{\gamma}}}\zeta_{\conj{\gamma}}\tensor{\tilde{\sigma}}{_\alpha_\beta}
		+\frac{1}{2}\tensor{\tilde{\sigma}}{_\tau_\beta}+O(\rho^{2n+4}).
	\end{align}
	\end{subequations}
	The latter equality implies that
	$\tensor{(\nabla^\mathrm{TW})}{^\beta}\tensor*{k}{^{(2n+2)}_\alpha_\beta}
	=\frac{1}{2}\tensor*{k}{^{(2n+3)}_\tau_\alpha}$.
	To squeeze out the meaning of the former, we need the $\rho^2$-coefficient of
	$\tensor{\Gamma}{^\gamma_{\conj{\beta}}_\tau}$.
	By \eqref{eq:str_equation_Graham_Lee},
	$[\bm{\zeta}_\tau,\bm{\zeta}_{\conj{\beta}}]=ir\tensor{A}{_{\conj{\beta}}^\gamma}\bm{\zeta}_\gamma+O(r^{3/2})
	=2i\rho^2\tensor{A}{_{\conj{\beta}}^\gamma}\bm{\zeta}_\gamma+O(\rho^3)$.
	Thus, using the symmetry of the Tanaka--Webster torsion, we obtain
	\begin{equation*}
		\begin{split}
			\tensor{\Gamma}{_{\conj{\gamma}}_{\conj{\beta}}_\tau}
			&=\frac{1}{2}(\bm{\zeta}_{\conj{\beta}}\tensor{g}{_\tau_{\conj{\gamma}}}
			-\bm{\zeta}_{\conj{\gamma}}\tensor{g}{_\tau_{\conj{\beta}}}
			-[\bm{\zeta}_{\conj{\beta}},\bm{\zeta}_\tau]_{\conj{\gamma}}
			+[\bm{\zeta}_{\conj{\beta}},\bm{\zeta}_{\conj{\gamma}}]_\tau
			+[\bm{\zeta}_\tau,\bm{\zeta}_{\conj{\gamma}}]_{\conj{\beta}})\\
			&=-\frac{1}{2}([\bm{\zeta}_{\conj{\beta}},\bm{\zeta}_\tau]_{\conj{\gamma}}
			-[\bm{\zeta}_{\conj{\beta}},\bm{\zeta}_{\conj{\gamma}}]_\tau
			-[\bm{\zeta}_\tau,\bm{\zeta}_{\conj{\gamma}}]_{\conj{\beta}})
			=2i\rho^2\tensor{A}{_{\conj{\beta}}_{\conj{\gamma}}}+O(\rho^3).
		\end{split}
	\end{equation*}
	Therefore, \eqref{eq:divergence_freeness_KE_meaning} implies that
	$\tensor{(\nabla^\mathrm{TW})}{^\alpha}\tensor*{k}{^{(2n+3)}_\tau_\alpha}
	=2i\tensor{A}{^\alpha^\beta}\tensor*{k}{^{(2n+2)}_\alpha_\beta}$. Thus we have
	\begin{equation*}
		\tensor{(\nabla^\mathrm{TW})}{^\alpha}\tensor{(\nabla^\mathrm{TW})}{^\beta}
		\tensor*{k}{^{(2n+2)}_\alpha_\beta}
		-i\tensor{A}{^\alpha^\beta}\tensor*{k}{^{(2n+2)}_\alpha_\beta}=0,
	\end{equation*}
	which means that $D^*\tensor*{\mathcal{O}}{^\bullet_\alpha_\beta}=0$.
\end{proof}

\begin{rem}
	In~\cite{Matsumoto_14} the author showed that, for partially integrable CR manifolds in general,
	\begin{equation}
		\label{eq:double_divergence_free_for_partially_integrable}
		\im(\tensor{\nabla}{^\alpha}\tensor{\nabla}{^\beta}\tensor{\mathcal{O}}{_\alpha_\beta}
		-i\tensor{A}{^\alpha^\beta}\tensor{\mathcal{O}}{_\alpha_\beta}
		-\tensor{N}{^\gamma^\alpha^\beta}\tensor{\nabla}{_\gamma}\tensor{\mathcal{O}}{_\alpha_\beta}
		-(\tensor{\nabla}{_\gamma}\tensor{N}{^\gamma^\alpha^\beta})\tensor{\mathcal{O}}{_\alpha_\beta})=0.
	\end{equation}
	Then, by differentiating it, we obtain $\im(D^*\tensor*{\mathcal{O}}{^\bullet_\alpha_\beta})=0$ for
	integrable CR manifolds; what we proved above is an improvement on this.
	Why are not we able to improve equation \eqref{eq:double_divergence_free_for_partially_integrable} itself?
	Recall that \eqref{eq:double_divergence_free_for_partially_integrable} comes out of the contracted
	second Bianchi identity of a normalized ACH metric $g$ satisfying \eqref{eq:approx_Einstein}.
	Namely, if we set $E={\Ric}+\frac{n+2}{2}g$ then $\mathcal{B}E=\delta E+\frac{1}{2}d(\tr E)=0$ holds,
	so we combine the identities $(\mathcal{B}E)(\bm{Z}_0)=0$ and $(\mathcal{B}E)(\bm{Z}_\alpha)=0$ to get
	\eqref{eq:double_divergence_free_for_partially_integrable}.
	A natural idea to refine this argument is to use the remaining identity: $(\mathcal{B}E)(\bm{Z}_\infty)=0$.
	Nevertheless, this does not work by the following reason.
	Let $\tensor*{E}{^{(2n+3)}_I_J}$ be the $\rho^{2n+3}$-coefficient of $\tensor{E}{_I_J}$.
	Then $\re\tensor{(\nabla^\mathrm{TW})}{^\alpha}\tensor*{E}{^{(2n+3)}_\infty_\alpha}$ appears in
	the $\rho^{2n+4}$-coefficient of $(\mathcal{B}E)(\bm{Z}_\infty)$,
	while $\re\tensor{(\nabla^\mathrm{TW})}{^\alpha}\tensor*{E}{^{(2n+3)}_0_\alpha}$ appears in the
	$\rho^{2n+4}$-coefficient of $(\mathcal{B}E)(\bm{Z}_0)$.
	But by the lack of K\"ahler structure, there is no means to compare
	$\re\tensor{(\nabla^\mathrm{TW})}{^\alpha}\tensor*{E}{^{(2n+3)}_\infty_\alpha}$ and
	$\re\tensor{(\nabla^\mathrm{TW})}{^\alpha}\tensor*{E}{^{(2n+3)}_0_\alpha}$, and consequently we cannot
	make a good use of these two equalities.
	In the proof of Theorem \ref{thm:obstruction_operator_observation} (2) above,
	\eqref{eq:divergence_freeness_Kahler} can be seen as the infinitesimal version of $\mathcal{B}E=0$.
	In this case however we have the extra equality $\tensor{\tilde{\sigma}}{_{\conj{\tau}}_\alpha}=0$,
	which is an outcome of the K\"ahlerity of the bulk metric $g$.
	(This comparison is not very parallel because we define $E$ by the \emph{normalized} approximate solution
	to the Einstein equation. However, even if we use the Bianchi gauge solution to define $E$,
	the same problem for improving \eqref{eq:double_divergence_free_for_partially_integrable} still happens.)
\end{rem}

To prove the last part of the theorem, recall the Weitzenb\"ock formula for $(0,q)$-forms on a
K\"ahler manifold $X$ with values in a holomorphic vector bundle $E$
(see, e.g.,~\cite{Moroianu_07}*{Theorem 20.2}).
In the case where $E=T^{1,0}X$ is the holomorphic tangent bundle, the formula reads as follows:
if $\eta=\tensor{\eta}{_{\conj{j}_1}_{\dotsb}_{\conj{j}_q}^k}$ is a $(T^{1,0}X)$-valued $(0,q)$-form,
then the Dolbeault Laplacian
$\Delta_{\conj{\partial}}=\smash{\conj{\partial}}^*\conj{\partial}+\conj{\partial}\smash{\conj{\partial}}^*$ acts
on $\eta$ as
\begin{equation*}
	2\tensor{(\Delta_{\conj{\partial}}\eta)}{_{\conj{i}_1}_{\dotsb}_{\conj{i}_q}^j}
	=\tensor{(\nabla^*\nabla\eta)}{_{\conj{i}_1}_{\dotsb}_{\conj{i}_q}^j}
	+\sum_{s=1}^q
	\tensor{\Ric}{^{\conj{k}}_{\conj{i}_s}}
	\tensor{\eta}{_{\conj{i}_1}_{\dotsb}_{\conj{k}}_{\dotsb}_{\conj{i}_q}^j}
	-\tensor{\Ric}{_l^j}\tensor{\eta}{_{\conj{i}_1}_{\dotsb}_{\conj{i}_q}^l}
	+2\sum_{s=1}^q\tensor{R}{^{\conj{k}}_{\conj{i}_s}_l^j}
	\tensor{\eta}{_{\conj{i}_1}_{\dotsb}_{\conj{k}}_{\dotsb}_{\conj{i}_q}^l}.
\end{equation*}
By taking the complex conjugate, we obtain a similar formula of $\Delta_\partial$ acting on
$(\conj{T^{1,0}X})$-valued $(p,0)$-forms.
In particular, if $\sigma=\tensor{\sigma}{_j_k}$ is a symmetric tensor of type $(2,0)$,
which is identified with a $(\conj{T^{1,0}X})$-valued $(1,0)$-form by raising an index, then
\begin{equation}
	\label{eq:Weitzenbock_in_2_tensors}
	2\tensor{(\Delta_\partial\sigma)}{_j_k}
	=\tensor{(\Delta_\mathrm{L}\sigma)}{_j_k}-2\tensor{\Ric}{_k^l}\tensor{\sigma}{_j_l}.
\end{equation}
Therefore, on a K\"ahler--Einstein manifold with $\Ric=\lambda g$, the operator $\Delta_\mathrm{L}-2\lambda$
on type $(2,0)$ symmetric tensors is equivalent to $2\Delta_\partial$,
as described in~\cite{Besse_87}*{Equation (12.93$'$)}.

Let us go back to our situation, in which $X$ is the square root of $\overline{\Omega}$.
The Weitzenb\"ock formula in particular implies that if $\eta$ is a
$(\thetatangent)^{1,0}$-valued $p$-$\Theta$-form of type $(p,0)$ then so is $\Delta_\partial\eta$,
because $\nabla$ is a $\Theta$-connection and $\Ric$, $R$ are $\Theta$-tensors.
Moreover, since \eqref{eq:approx_Einstein_Kahler} holds, \eqref{eq:Weitzenbock_in_2_tensors} implies that
\begin{equation*}
	2\Delta_\partial\sigma=(\Delta_\mathrm{L}+n+2)\sigma+O(\rho^{2n+4}).
\end{equation*}
Therefore, if $\sigma$ is a solution of \eqref{eq:linearized_Einstein_equation}, then it solves
\begin{equation*}
	\Delta_\partial\sigma=O(\rho^{2n+2}).
\end{equation*}
Since $\mathcal{V}_\Theta=C^\infty(X,\thetatangent)$ is closed under the Lie bracket,
one can check that $\partial\Delta_\partial\sigma=O(\rho^{2n+2})$. Therefore,
\begin{equation*}
	\Delta_\partial(\partial\sigma)=O(\rho^{2n+2}).
\end{equation*}
Thus we are led to studying $\Delta_\partial$ acting on what we call
\emph{Nijenhuis-type} 3-$\Theta$-tensors, i.e.,
real 3-$\Theta$-tensors $\nu$ whose components other than $\tensor{\nu}{_A_B_C}$ and
$\tensor{\nu}{_{\conj{A}}_{\conj{B}}_{\conj{C}}}$ are zero that satisfy
\begin{equation*}
	\tensor{\nu}{_(_A_B_)_C}=0\qquad\text{and}\qquad
	\tensor{\nu}{_(_A_B_C_)}=0.
\end{equation*}

\begin{lem}
	\label{lem:Laplacian_on_Nijenhuis_type_tensors}
	Suppose $\nu$ is a Nijenhuis-type 3-$\Theta$-tensor. Then, if $\nu=O(\rho^j)$,
	\begin{align*}
		2\tensor{(\Delta_\partial\nu)}{_\tau_(_\alpha_\beta_)}
		&=-\frac{1}{4}j(j-2n-2)\tensor{\nu}{_\tau_(_\alpha_\beta_)}+O(\rho^{j+1}),\\
		2\tensor{(\Delta_\partial\nu)}{_\tau_[_\alpha_\beta_]}
		&=-\frac{1}{4}(j^2-(2n+2)j-8)\tensor{\nu}{_\tau_[_\alpha_\beta_]}+O(\rho^{j+1}),\\
		2\tensor{(\Delta_\partial\nu)}{_\alpha_\beta_\gamma}
		&=-\frac{1}{4}(j-1)(j-2n-1)\tensor{\nu}{_\alpha_\beta_\gamma}+O(\rho^{j+1}).
	\end{align*}
\end{lem}

\begin{proof}
	By the Weitzenb\"ock formula, we have
	\begin{multline*}
		2\tensor{(\Delta_\partial\nu)}{_A_B_C}
		=\nabla^*\nabla\tensor{\nu}{_A_B_C}+\tensor{\Ric}{_A^D}\tensor{\nu}{_D_B_C}
		+\tensor{\Ric}{_B^D}\tensor{\nu}{_A_D_C}-\tensor{\Ric}{_C^D}\tensor{\nu}{_A_B_D}\\
		+2\tensor{R}{_A^D_C^E}\tensor{\nu}{_D_B_E}+2\tensor{R}{_B^D_C^E}\tensor{\nu}{_A_D_E}.
	\end{multline*}
	Then \eqref{eq:approx_Einstein_Kahler} and \eqref{eq:curvature_of_ACH} show that, under the assumption
	$\nu=O(\rho^j)$,
	\begin{equation*}
		2\tensor{(\Delta_\partial\nu)}{_A_B_C}=\left(\nabla^*\nabla-\frac{n+8}{2}\right)\tensor{\nu}{_A_B_C}
		+O(\rho^{j+1}).
	\end{equation*}
	The argument in the beginning of the proof of Lemma \ref{lem:Laplacian_on_tensors} shows that,
	if $\nu=\rho^{j}\tilde{\nu}$, then
	$\nabla^*\nabla\nu=\rho^j\nabla^*\nabla\tilde{\nu}-\frac{1}{4}j(j-2n-2)\rho^j\tilde{\nu}+O(\rho^{j+1})$.
	Therefore, it suffices to consider the $j=0$ case and check that the following hold on $M$:
	\begin{equation*}
		\tensor{(\nabla^*\nabla\nu)}{_\tau_(_\alpha_\beta_)}
		=\frac{n+8}{2}\tensor{\nu}{_\tau_(_\alpha_\beta_)},\qquad
		\tensor{(\nabla^*\nabla\nu)}{_\tau_[_\alpha_\beta_]}
		=\frac{n+12}{2}\tensor{\nu}{_\tau_[_\alpha_\beta_]},\qquad
		\tensor{(\nabla^*\nabla\nu)}{_\alpha_\beta_\gamma}
		=\frac{15}{4}\tensor{\nu}{_\alpha_\beta_\gamma}.
	\end{equation*}
	These are obtained by a straightforward calculation using \eqref{eq:Christoffel}.
\end{proof}

\begin{proof}[Proof of Theorem \ref{thm:obstruction_operator_observation} (3)]
	If $\sigma$ is a solution of \eqref{eq:linearized_Einstein_equation} with boundary data
	$\tensor{\psi}{_\alpha_\beta}$, then $\nu=\partial\sigma$ is given by
	\begin{equation*}
		\tensor{\nu}{_A_B_C}=\tensor{\nabla}{_[_A}\tensor{\sigma}{_B_]_C}.
	\end{equation*}
	We first compute the lowest coefficients of $\tensor{\nu}{_A_B_C}$.
	The boundary values of the Christoffel symbols \eqref{eq:Christoffel} show that
	\begin{equation*}
		\tensor{\nabla}{_\tau}\tensor{\sigma}{_\alpha_\beta}=\tensor{\sigma}{_\alpha_\beta}+O(\rho)
		\qquad\text{and}\qquad
		\tensor{\nabla}{_\alpha}\tensor{\sigma}{_\tau_\beta}=\tensor{\sigma}{_\alpha_\beta}+O(\rho),
	\end{equation*}
	and hence $\tensor{\nu}{_\tau_\alpha_\beta}=O(\rho)$, while
	\begin{equation*}
		\tensor{\nabla}{_\alpha}\tensor{\sigma}{_\beta_\gamma}
		=\rho\tensor*{\nabla}{^{\mathrm{TW}}_\alpha}\tensor{\psi}{_\beta_\gamma}+O(\rho^2),
	\end{equation*}
	which implies that
	$\tensor{\nu}{_\alpha_\beta_\gamma}=\frac{1}{2}\rho\tensor{(N^\bullet\psi)}{_\alpha_\beta_\gamma}+O(\rho^2)$.
	Then Lemma \ref{lem:Laplacian_on_Nijenhuis_type_tensors} shows that, if we write
	\begin{equation*}
		\tensor{\nu}{_\tau_\alpha_\beta}
		=\sum_{j=1}^{2n+1}\rho^j\tensor*{n}{^{(j)}_\tau_\alpha_\beta}+O(\rho^{2n+2})\qquad\text{and}\qquad
		\tensor{\nu}{_\alpha_\beta_\gamma}
		=\rho\tensor{(N^\bullet\psi)}{_\alpha_\beta_\gamma}
		+\sum_{j=2}^{2n+1}\rho^j\tensor*{n}{^{(j)}_\alpha_\beta_\gamma}+O(\rho^{2n+1}),
	\end{equation*}
	then $\tensor*{n}{^{(1)}_\tau_\alpha_\beta}$, $\dotsc$, $\tensor*{n}{^{(2n+1)}_\tau_\alpha_\beta}$ and
	$\tensor*{n}{^{(1)}_\alpha_\beta_\gamma}$, $\dotsc$, $\tensor*{n}{^{(2n)}_\alpha_\beta_\gamma}$ have
	universal expression in terms of the Tanaka--Webster connection and $N^\bullet\psi$.
	To determine $\tensor*{n}{^{(2n+1)}_\alpha_\beta_\gamma}$,
	we use the fact that $\partial\nu=\partial^2\sigma=0$. If $\nu=O(\rho^j)$, then we can compute
	\begin{equation*}
		\tensor{(\partial\nu)}{_\tau_\alpha_\beta_\gamma}
		=\tensor{\nabla}{_[_\tau}\tensor{\nu}{_\alpha_]_\beta_\gamma}
		=\frac{1}{4}(j+1)\tensor{\nu}{_\alpha_\beta_\gamma}+O(\rho^{j+1}).
	\end{equation*}
	Since this should be zero, $\tensor*{n}{^{(2n+1)}_\alpha_\beta_\gamma}$ is again written in terms of
	the Tanaka--Webster connection and $N^\bullet\psi$.

	Now we compute the $\rho^{2n+2}$-coefficient of
	$\tilde{\sigma}=(\Delta_\mathrm{L}+n+2)\sigma=\Delta_\partial\sigma+O(\rho^{2n+4})$.
	If we choose $\sigma$ for which $\delta\sigma=O(\rho^{2n+4})$, then
	\begin{equation*}
		\tilde{\sigma}=\partial^*\partial\sigma+O(\rho^{2n+4})=\partial^*\nu+O(\rho^{2n+4}).
	\end{equation*}
	If $\nu=O(\rho^j)$, then
	\begin{equation*}
		\tensor{(\partial^*\nu)}{_(_\alpha_\beta_)}
		=-\tensor{\nabla}{^\tau}\tensor{\nu}{_\tau_(_\alpha_\beta_)}
		-\tensor{\nabla}{^\gamma}\tensor{\nu}{_\gamma_(_\alpha_\beta_)}
		=-\frac{1}{4}(j-2n-2)\tensor{\nu}{_\tau_(_\alpha_\beta_)}+O(\rho^{j+1}),
	\end{equation*}
	which means that the $\rho^{2n+2}$-coefficient of $\tensor{\tilde{\sigma}}{_\alpha_\beta}$ is
	expressed by the coefficients of $\nu$ up to $(2n+1)$st order.
	To sum up, $\tensor*{\mathcal{O}}{^\bullet_\alpha_\beta}$ can be written as a universal polynomial of
	contractions of Tanaka--Webster local invariants and covariant derivatives of
	$\tensor{(N^\bullet\psi)}{_\alpha_\beta_\gamma}$.
\end{proof}

\bibliography{myrefs}

\end{document}